\documentclass[reqno]{amsart}
\usepackage{graphicx, amsmath,amsfonts,amsthm,amssymb, paralist, fullpage, setspace}
\usepackage[shortlabels]{enumitem}
\usepackage{hyperref}

\newtheorem{theorem}{Theorem}[section]
\newtheorem{corollary}[theorem]{Corollary}
\newtheorem{lemma}[theorem]{Lemma}

\theoremstyle{definition}
\newtheorem{definition}[theorem]{Definition}
\theoremstyle{remark}
\newtheorem{remark}[theorem]{Remark}
\numberwithin{equation}{section}

\newcommand{\g}{\geqslant}

\newcommand{\RR}{\mathbb{R}}\newcommand{\R}{\mathbb{R}}
\newcommand{\ZZ}{\mathbb{Z}}
\newcommand{\CC}{\mathbb{C}}

\newcommand{\sph}{\mathbb{S}}
\newcommand{\NN}{\mathbb{N}}
\newcommand{\p}{\partial}

\newcommand{\les}{\leqslant}
\newcommand{\lesa}{\lesssim}

\newcommand{\ma}{\measuredangle}

\newcommand{\mc}[1]{\mathcal{#1}}
\newcommand{\mb}[1]{\mathbf{#1}}

\newcommand{\lr}[1]{ \langle #1 \rangle}
\newcommand{\ind}{\mathbbold{1}}

\DeclareMathOperator{\Real}{Re}

\DeclareSymbolFont{bbold}{U}{bbold}{m}{n}
\DeclareSymbolFontAlphabet{\mathbbold}{bbold}

\DeclareMathOperator*{\supp}{supp}

\DeclareMathOperator*{\diag}{diag}
\DeclareMathOperator*{\dist}{dist}

\newcommand{\C}{\mathbb{C}}
\newcommand{\N}{\mathbb{N}}

\newcommand{\Z}{\mathbb{Z}}

\setdefaultenum{(i)}{(a)}{1}{A}

\begin{document}

\title{The massless and the non-relativistic limit for the cubic Dirac equation}

\author[T.~Candy]{Timothy Candy}
\address[T.~Candy]{Department of Mathematics and Statistics, University of Otago, PO Box 56, Dunedin 9054, New Zealand}
\email{tcandy@maths.otago.ac.nz}

\author[S.~Herr]{Sebastian Herr}
\address[S.~Herr]{Fakult\"at f\"ur
  Mathematik, Universit\"at Bielefeld, Postfach 10 01 31, 33501
  Bielefeld, Germany}
\email{herr@math.uni-bielefeld.de}

\begin{abstract}
Massive and massless Dirac equations with Lorentz-covariant cubic nonlinearities are considered in spatial dimension $d=2,3$. Global well-posedness of the Cauchy problem for small initial data in scale-invariant Sobolev spaces and scattering of solutions is proved by a new approach which uses bilinear Fourier restriction estimates and atomic function spaces. Furthermore, global uniform convergence results, both in the massless and in the non-relativistic limit, are proved at optimal regularity. In both regimes, these are the first results which imply convergence of scattering states and wave operators.
\end{abstract}

\maketitle


\section{Introduction}\label{sec:intro}
In the present paper, we address the Cauchy problem for cubic Dirac equations
        \begin{equation}\label{eq:cd}
        \begin{split}
        -i  \gamma^\mu \p_\mu \psi + m \psi &= F(\psi)\\
        \psi(0,x)&=f(x), \quad x\in \R^d
        \end{split}
        \end{equation}
       for a spinor $\psi:\RR^{1+d}\to \CC^{N_d}$, where $d\in\{2,3\}$ is the spatial dimension and the spinor dimension is $N_d=2$ if $d=2$, and $N_d=4$ if $d = 3$, and  $m \in \RR$ is a mass parameter.
The summation convention is in effect, thus repeated Greek indices are summed over $\mu=0,\ldots, 4$, and indices are raised and lowered using the Minkowski metric $ \eta = \diag(1, -1, \dots, -1)$. The Dirac matrices $\gamma^\mu\in
 \CC^{N_d\times N_d}$ are chosen (see Section \ref{sec:setup}) so that they satisfy the anti-commutativity property
 $$ \gamma^\mu \gamma^\nu + \gamma^\nu \gamma^\mu = 2 \eta^{\mu \nu}.$$
 Let $\overline{\psi} = \psi^\dagger \gamma^0$ be the Dirac adjoint, where $z^\dagger$ is the complex conjugate transpose of $z\in \CC^{N_d}$.

 The Dirac operator $\mc{D}_m=  -i  \gamma^\mu \p_\mu  + m$ has been introduced by Paul Dirac in \cite{Dirac-1928} to describe free relativistic particles with spin (e.g.\ electrons). We have $\gamma^0\mc{D}_m=-i\partial_t +\mc{H}_m$ with the Hamiltonian $\mc{H}_m=-i\gamma^0\gamma^j\partial_j+m\gamma^0$. $\mc{H}_m$ has the feature that  is unitarily equivalent to the ($N_d\times N_d$ system of) square-root Klein-Gordon Hamiltonians $\gamma^0\sqrt{-\Delta+m^2}$, see \cite{Thaller-1992} or Section \ref{sec:setup} for more details.

 We are primarily interested in the physically relevant cases where the nonlinearity $F(\psi)$ is given by
$$F(\psi)=\big( \overline{\psi} \psi \big) \psi \quad \text{(Soler model \cite{Soler-1970})}$$
or
$$F(\psi )=\big( \overline{\psi} \gamma^\mu \psi \big) \gamma_\mu \psi \quad \text{(Thirring model \cite{Thirring-1958})},$$
and our results are phrased for these cases.
However our arguments also apply to slightly more general Lorentz covariant cubic nonlinearities, see Remark \ref{rem:nonlinearities} below. Such nonlinear Dirac equations arise in relativistic quantum mechanics as models for self-interacting Dirac fermions \cite{Bjorken1964,Thaller-1992,Ranada1983}.

The present paper has three main results. The first result is a refined and unified global well-posedness and scattering result for small initial data in critical Sobolev spaces, which is discussed in Subsection \ref{subsec:gwp}. The second result concerns the massless limit, and gives a uniform and global convergence result as the mass parameter tends to zero, see Subsection \ref{subsec:ml}.
The third result concerns the non-relativistic limit and establishes a uniform and global convergence result towards a solution of a Schr\"odinger equation as the speed of light goes to infinity, see Subsection \ref{subsec:nrl}. As a consequence, we conclude the convergence of scattering states and wave operators for both the massless and non-relativistic limits.

\subsection{A unified and refined global well-posedness and scattering theory for small initial data}\label{subsec:gwp}
Let $\alpha > 0$ and $\psi$ be a solution to \eqref{eq:cd}. The rescaled function
            $ (t,x) \mapsto \alpha^{\frac{1}{2}}\psi( \alpha t, \alpha x)$ solves the cubic Dirac equation with mass parameter $\alpha m$. The homogeneous Sobolev norm $\|\psi\|_{\dot{H}^{s}}$ is invariant under this rescaling if and only if $s=\frac{d-1}{2}$, hence the scale-invariant (or critical) Sobolev regularity is $s_d:=\frac{d-1}{2}$.

For $m\ne 0$, global well-posedness and scattering for small initial data in the critical Sobolev space $H^{s_d}(\R^d)$ was proved in previous work of the second author and Bejenaru \cite{Bejenaru2014a,Bejenaru2016} and if $m=0$ for small initial data in $\dot{H}^{s_d}(\R^d)$ in previous work of the first author and Bournaveas \cite{Bournaveas2015}.
In both cases Tataru's null frame space construction  \cite{Tataru2001} was adapted from the Wave Maps to the Dirac setting as a remedy for the failure of the endpoint Strichartz estimate.

In general, for a dispersive PDE with characteristic hypersurface $\Sigma$, the dispersive decay is determined by the number of non-vanishing principal curvatures of $\Sigma$. In our application here,  $\Sigma$ is a cone (if $m=0$) or a hyperboloid (if $m\ne 0$). The dispersive decay implies mixed norm $L^p$-estimates (aka Strichartz estimates) \cite{Str-77,Ginibre1995,Keel1998} for free solutions which play an important role in the perturbative analysis of nonlinear dispersive PDEs. From another perspective, the solution operator is (up to isometry) dual to the Fourier restriction operator $f \mapsto \widehat{f}|_{\Sigma}$. Proving optimal $L^p$-estimates for the Fourier restriction operator is a classical topic in harmonic analysis.
 In the course of proving such estimates, in turned out that it is beneficial to pass to a bilinear setting \cite{Wolff2000,Tao2001b}, where the product of two free solutions is considered. In this case, additional transversality considerations can be exploited which lead to stronger decay of the product.

 In previous work, we have proved bilinear Fourier restriction estimates for perturbations of free solutions and used these in the analysis of nonlinear dispersive PDE, such as the Dirac--Klein-Gordon system \cite{Candy2018a} and the Wave Maps equation \cite{Candy2018b}. We also refer to these papers for further references and to \cite{Candy2019a} for a more general result on perturbative bilinear Fourier restriction estimates.

In the present paper we provide an alternative and unified proof of the small data global well-posedness theory in the critical Sobolev spaces $H^{s_d}$ for the cubic Dirac equation \eqref{eq:cd} which covers both the massive $m\ne 0$ and massless $m=0$ cases simultaneously. It is worth noting that although the small data theory obtained in \cite{Bejenaru2014a,Bejenaru2016, Bournaveas2015} in both the massive and massless cases adapted the null frame construction of Tataru, the spaces used were very different when $m=0$ and $m\ne 0$. Instead, here we give a unified argument which is based on the adapted function spaces $U^p$ and $V^p$ and bilinear Fourier restriction theory. In particular, we obtain a result which is uniform in the mass parameter $m$. While the approach here is conceptually similar to  \cite{Candy2018b}, it turns out that for the cubic Dirac equation the bilinear $L^2_{t,x}$-estimates derived for Wave Maps \cite{Candy2018b} and their analogues for the hyperboloid (if $m\ne 0$) are not strong enough and it is necessary to use the quadrilinear structure of the (dualised) nonlinearity. The key new ingredient here to break the symmetry is a bilinear $L^p_tL^2_x$-estimate for certain $p<2$, which requires the use of the adapted atomic space $U^{a}$ for some $a<2$.

As we eventually apply the global theory obtained below to the massless and non-relativistic limits, we require bounds which are uniform in $m$, and are both scale invariant and have the correct scaling to match both the wave and Schr\"odinger limits. This requirement means that the standard homogeneous and inhomogeneous Sobolev spaces are not sufficient. Instead we require norms which can detect improved low frequency behaviour.  To this end,
given $m \in \RR$, we define the Sobolev space $H^{s,\sigma}_m=H^{s,\sigma}_m(\R^d)$ via
        $$ \| f \|_{H^{s,\sigma}_m} := \big\| |\cdot |^{\sigma}\lr{\cdot}_m^{s-\sigma} \widehat{f} \big\|_{L^2},$$
where $\lr{\xi}_m := ( m^2 + |\xi|^2)^\frac{1}{2}$. Note that $H^{s, \sigma}_0 = H^{s, s}_m = \dot{H}^s$ and $H^{s, 0}_1 = H^s$ where $\dot{H}^s$ and $H^s$ are the standard homogeneous and inhomogeneous Sobolev spaces.  The low frequency parameter $\sigma$ is used to formulate our results in an optimal way in both the high frequency regime $|\xi|\g m$,  and the low frequency regime $|\xi|<|m|$. This adaptability is crucial to ensure that we have uniform estimates in both the massless and non-relativistic limits.

Although the space $H^{s, \sigma}_m$ may initially seem slightly artificial, it in fact arises naturally in the context of dispersive estimates. For example the standard Klein-Gordon Strichartz estimate
		$$ \| e^{it\lr{\nabla}_m} f \|_{L^{\frac{2d+2}{d-1}}_{t,x}(\RR^{1+d})} \lesa \| f \|_{H^{\frac{1}{2}}_x(\RR^d)}  $$
can be sharpened at low frequencies (via a simple application of Sobolev embedding) to the improved bound
$$
\| e^{it \lr{\nabla}_m} f \|_{L^{\frac{2d+2}{d-1}}_{t,x}(\R\times \R^d)} \lesa \|f\|_{H^{\frac{1}{2}, \frac{1}{d+1}}_m(\R^d)},
$$
see for instance  \cite[Lemma 4]{Machihara2003} or  Lemma \ref{lem:wave stri} below.

To formulate our first main theorem, we let
					$$ \mathcal{U}_m(t) := e^{-it\mc{H}_m} = e^{it(\gamma^0 \gamma^j \p_j - m \gamma^0)}$$
denote the free propagator for \eqref{eq:cd}. Let $S^{s,\sigma}_m$ be the function space defined via the norm
			$$ \| \psi \|_{S^{s,\sigma}_m} := \big\||\nabla|^{\sigma}\lr{\nabla}_m^{s-\sigma}\mc{U}_m(-t) \psi \big\|_{\ell^2 U^a}. $$
Here $1<a<2$ is a fixed constant, and $\ell^2 U^a$ is the atomic space $U^a$ together with an $\ell^2$ sum over frequencies $\lambda \in 2^\NN$, see Section \ref{sec:setup} for a precise definition. The parameters $s$ and $\sigma$ give the regularity at high ($|\xi| \g m$) and low ($|\xi| \les m$) frequencies respectively. It is important to emphasise that the only dependence on the mass in the above construction appears through the free solution operator $\mc{U}_m$ and the derivatives $\lr{\nabla}_m$. This fact is crucial in our later applications to the (global in time) massless and non-relativistic limits. We state our results for the forward in time Cauchy problem on $\RR_+=[0,\infty)$ and remark that by time reversibility the analogous results hold on $\RR_-$.

\begin{theorem}\label{thm:gwp}
Let $d \in\{2,3\}$, $s\g s_d=\frac{d-1}{2}$, and $0\les \sigma \les \sigma_d:=\frac{d-2}{2}$. There exists $\delta>0$ such that for any $m\in \RR$  and any initial data $f \in H_m^{s_d,\sigma_d}(\R^d)$ satisfying
                $$\| f \|_{H_m^{s_d,\sigma_d}} \les \delta$$
there exists a unique global solution $\psi^{(m)}=\psi^{(m)}[f]\in S^{s_d,\sigma_d}_m \cap C(\RR_+; H^{s_d,\sigma_d}_m(\R^d))$ to \eqref{eq:cd} on $\RR_+$ which scatters, i.e. there exist scattering states $f_{\infty}\in H^{s_d,\sigma_d}_m(\R^d)$ such that
$$\lim_{t \to  \infty}\|\psi^{(m)}(t)-\mathcal{U}_m(t)f_{ \infty}\|_{H^{s_d,\sigma_d}_m}=0.$$
If, in addition, $f \in H^{s,\sigma}_m(\R^d)$, then $\psi^{(m)} \in  S^{s,\sigma}_m\cap C(\RR_+; H^{s,\sigma}_m(\R^d))$ and the scattering claim also holds in $H^{s,\sigma}_m(\R^d)$.
Moreover, the flow map $f \to \psi^{(m)}[f]$ is $C^\infty$.
\end{theorem}

Here, the case $\sigma=0$ recovers the previous results in \cite{Bejenaru2014a,Bejenaru2016, Bournaveas2015} and the case $d=3$, $m\ne0$, and $0<\sigma\les \frac12$ improves upon \cite{Bejenaru2014a}.
The case of higher dimensions $d\geq 4$ is significantly easier because of the fact that the endpoint Strichartz estimate for the wave equation is available \cite{Keel1998}, therefore we do not address this here. It is also worth noting that the regularity $s_d = \frac{d-1}{2}$ is the scale invariant regularity in the massless limit (i.e. $m=0$), while $\sigma_d = \frac{d-2}{2}$ is the scale invariant regularity for the cubic nonlinear Schr\"odinger equation arising in the non-relativistic limit (roughly $m\to \infty$).

\begin{remark}
  \label{rem:nonlinearities}
Although we are primarily interested in the Soler and Thirring models, all results stated in this paper also apply to the slightly more general case where the nonlinearity $F$ is a linear combination of terms of the form
	$$(\overline{\psi} \mb{A}_1 \psi) \mb{A}_2 \psi $$
for any constant coefficient matrices $\mb{A}_1, \mb{A}_2 \in \CC^{N_d}$ satisfying the commutativity condition
			\begin{equation}\label{eqn:null cond}
             \gamma^0 \gamma^j \mb{A} = \mb{A} \gamma^0 \gamma^j,  \qquad j\in \{1, 2, 3\}.
        \end{equation}
Clearly the Soler model (where $\mb{A}_1 = \mb{A}_2 = I$) satisfies \eqref{eqn:null cond}. A slightly more involved computation via the Fierz identities \cite{Nieves2004} shows that the Thirring model can also be written as a sum of terms satisfying \eqref{eqn:null cond}. We give the details in Remark \ref{rem:nonlinearities II} below.
  \end{remark}

  \begin{remark}
    \label{rem:maj}\leavevmode
    \begin{enumerate}
      \item There is no large data theory for nonlinear Dirac equation because of the lack of a coercive energy. On the other hand, there are soliton solutions \cite{Soler-1970,Strauss1986}, which are obviously non-scattering.
      \item
For arbitrarily large initial data there is an algebraic condition (known as the  Majorana-condition) which ensures that the nonlinear term vanishes (at least for the Soler model $F(\psi) = (\overline{\psi}\psi)\psi$). Perturbing these solutions gives an open unbounded subset of $H^s(\R^d)$ for which we have  global scattering solutions to \eqref{eq:cd}. This has been analysed in detail in \cite{Candy2018} in $d=3$ and we remark here that the same argument works if $d=2$, see also \cite{Dancona2017} for radial data in $d=3$.

It would be of interest to extend Theorem \ref{thm:gwp} to include large initial data beyond this somewhat special class of data satisfying the Majorana condition. However, improving the nonlinear estimates for (say) dispersed solutions akin to \cite{Sterbenz2010} seems to very difficult for the cubic Dirac equation, see Remark \ref{rem:improving bound} for more details.
\item Without going into detail, we remark that we also obtain local well-posedness for arbitrarily large initial data as an easy consequence of the estimates proven in this paper (e.g.\ analogous to \cite[Theorem 1.2]{Hadac2009}).
\end{enumerate}
\end{remark}

\subsection{The massless limit}\label{subsec:ml}
The second result is a global in time convergence theory in the massless limit $m\to0$. As a consequence, we obtain the convergence of scattering states and wave operators, which is the first result of this kind. The arguments used rely crucially on the uniform (in mass) theory developed above, together with the key fact that composing the free solution operators map our function spaces for $m\ne 0$ to $m=0$. It is important to emphasize that the convergence as $m\to 0$ on compact time intervals is straightforward but far from sufficient to obtain the convergence of wave operators and scattering states that we obtain here. We expect that the new approach here, namely using adapted function spaces (instead of, say, Strichartz norms) and pulling back along the linear flow to prove global-in-time and uniform convergence, will have applications to other problems as well.

To motivate the formulation below, we first recall that the Klein-Gordon equation and wave equation display very different asymptotic behaviour. In particular (at least from smooth localised data) we expect that for $m\not =0$ the solution $\psi^{(m)}(t)$ should decay like $t^{-\frac{d}{2}}$, while $\psi^{(0)}(t)$ exhibits the slower decay rate $t^{-\frac{d-1}{2}}$. Consequently there is no reason to expect that the difference $\sup_{t\in \RR} \|\psi^{(m)}(t) - \psi^{(0)}(t)\|_{H^s_m}$ vanishes in the limit $m\to 0$. Instead, if we wish to understand the global-in-time convergence in the massless limit, we first have to account for the difference in the asymptotic behaviour by pulling back with the linear propagators. After also taking into account the difference in regularity at small frequencies, we arrive at the following statement.

\begin{theorem}\label{thm:limit}
Let $d\in \{2, 3\}$, $s_d = \frac{d-1}{2}$, and $0\les \sigma\les \sigma_d$. There exists $\delta>0$ such that for any $m\in \RR$, and any data $f^{(m)} \in H^{s_d,\sigma}_m$ satisfying
                $$ \| f^{(m)} \|_{H^{s_d,\sigma_d}_m} \les \delta, \qquad \text{and} \qquad  \lim_{m\to 0} \big\| |\nabla|^\sigma \lr{\nabla}_m^{s_d-\sigma} f^{(m)} - |\nabla|^{s_d} f^{(0)} \big\|_{L^2} = 0,$$
the solutions $\psi^{(m)}=\psi^{(m)}[f^{(m)}] \in S^{s_d,\sigma}_m$ to \eqref{eq:cd} on $\R_+$ given by Theorem \ref{thm:gwp} satisfy
        $$ \lim_{m\to 0}  \big\| |\nabla|^\sigma  \lr{\nabla}_m^{s_d-\sigma} \mc{U}_m(-\cdot ) \psi^{(m)}  - |\nabla|^{s_d} \mc{U}_0(-\cdot ) \psi^{(0)}\big\|_{\ell^2 U^a} = 0$$
\end{theorem}

In the following, we discuss some consequences of Theorem \ref{thm:limit}.
\begin{remark}\label{rmk:reg}\leavevmode
\begin{enumerate}
\item
Concerning the hypothesis in Theorem \ref{thm:limit}, we remark that if $f^{(m)}=f^{(0)}\in \dot{H}^\sigma$, then  $$\lim_{m\to 0} \big\| |\nabla|^\sigma \lr{\nabla}_m^{s_d-\sigma} f^{(m)} - |\nabla|^{s_d} f^{(0)} \big\|_{L^2} = 0.$$
\item The additional derivative factors can be avoided at the cost of slighlty weakening the statement. For instance, under the hypothesis of Theorem \ref{thm:limit}, we also have
  $$ \lim_{m\to 0} \big\| \mc{U}_m(-\cdot ) \psi^{(m)}  - \mc{U}_0(-\cdot ) \psi^{(0)}\big\|_{L^\infty_t \dot{H}^{s_d}}= 0, $$
  see Subsection  \ref{subsec:proof-cor-mlimit} for a proof.
\end{enumerate}
\end{remark}

As stated above, it is not possible to remove the solution operators $\mc{U}_m$ from the conclusion of Theorem \ref{thm:limit} due to the distinct asymptotic behaviour of the wave and Klein-Gordon equations. On the other hand,  it is not so difficult to prove that Theorem \ref{thm:limit} implies uniform convergence of the solutions on any bounded time interval.

\begin{corollary}\label{cor:mlimit}
 For any bounded interval $I=[0,T)$
 we have
 \begin{equation*}
  \lim_{m \to 0} \|  \psi^{(m)}  -   \psi^{(0)} \|_{L^\infty(I,\dot{H}^{s_d}(\RR^d))}= 0
\end{equation*}
under the hypothesis of Theorem \ref{thm:limit}.
\end{corollary}

A further consequence is most easily phrased in the language of scattering theory. Given $m\in \RR$, we let $B_\delta^{(m)}$ be the closed ball in $H^{s_d,\sigma_d}_m$ of  radius $\delta$ centered at the origin, and let $f^{(m)}_{\pm\infty} \in H^{s_d, \sigma}_m$ denote the scattering states at $t = \pm\infty$ of the solution $\psi^{(m)}$ to \eqref{eq:cd}. Thus
            $$ \lim_{t\to \pm \infty} \| \mc{U}_m(-t) \psi^{(m)}(t) - f_{\pm \infty}^{(m)} \|_{H^{s_d, \sigma}_m} = 0. $$
By Theorem \ref{thm:gwp} (and the analogue on $\R_-$), the operators
	$$\Omega_\pm^{(m)}:B_\delta^{(m)}  \cap H^{s_d,\sigma}_m\to H^{s_d,\sigma}_m, \; f \mapsto f^{(m)}_{\pm \infty} $$
are well-defined, $C^\infty$, and given by the formulae
 $$
 \Omega_\pm^{(m)}[f]=  f_{\pm \infty}^{(m)}=f+i \int_0^{\pm \infty} \mathcal{U}_m(-t') \gamma^0 F\big(\psi^{(m)}[f]\big)(t')dt'
 $$
  By the inverse function theorem \cite[Theorem 15.2]{Deimling1985}, their local inverses, the wave operators
$$W_\pm^{(m)}:B_\delta^{(m)} \cap H^{s_d,\sigma}_m \to H^{s_d,\sigma}_m, \; f^{(m)}_{\pm \infty}\mapsto \psi^{(m)}(0) $$
are well-defined and $C^\infty$, too. The convergence statement in Theorem \ref{thm:limit} then immediately implies that scattering states and wave operators must converge. More precisely, via Remark \ref{rmk:reg} and the fact that $B^{(1)}_\delta \subset B^{(m)}_{\delta}$ for all $|m|\les 1$, we have the following.

\begin{corollary}\label{cor:wo}
Let $f \in B_\delta^{(1)}$.
We have convergence of the scattering states
$$\lim_{m \to 0}\| \Omega_\pm^{(m)}[f]- \Omega_\pm^{(0)} [f] \|_{\dot{H}^{s_d}}=0,$$
and of the wave operators
$$\lim_{m \to 0}\| W_\pm^{(m)}[f]- W_\pm^{(0)} [f] \|_{\dot{H}^{s_d}}=0.$$
\end{corollary}

The above corollaries could be sharpened with respect to regularity in the low frequencies (see Remark \ref{rmk:reg}), but here we prefer to choose the homogeneous space $\dot{H}^{s_d}=H^{s_d}_0$ as a common domain for the operators $\Omega_{\pm}^{(m)}$ and $W_\pm^{(m)}$.

\subsection{The non-relativistic limit}\label{subsec:nrl} Our third main result concerns the non-relativistic limit. To make this precise, we begin by making explicit the dependence on the speed of light $c>0$ in \eqref{eq:cd}. This leads to the Dirac equation
    \begin{equation}\label{eq:dirac c}
    \begin{split}
        -i\gamma^0 \p_t \psi - i c \gamma^j \p_j \psi + c^2 \psi &= F(\psi)\\
                            \psi(0) &=f.
    \end{split}
  \end{equation}
For simplicity, we have set the mass $m=1$ in \eqref{eq:dirac c} but it is clear by rescaling that the arguments below apply to any (fixed) mass $m\in \RR$ with $m\not = 0$. It is well known that if $\psi$ solves \eqref{eq:dirac c},  then at least on compact intervals in time the phase modulated solution $e^{itc^2 \gamma^0} \psi$ converges to a solution to a cubic nonlinear Schr\"odinger equation \cite{Machihara2003,Matsuyama1995,Najman1992}. Substantial effort has also gone into proving convergence of the Klein-Gordon equation to the nonlinear Schr\"odinger equation, see for instance \cite{Nakanishi2002, Masmoudi2002} and the references therein. With the notable exception of  \cite{Nakanishi2002} (which obtained global convergence of wave operators in the non-relativistic limit for Klein-Gordon to Schr\"odinger via a compactness argument), the results cited above only apply to compact time intervals and the Soler nonlinearity  $F(\psi) = (\overline{\psi} \psi) \psi$.

Here we prove a global convergence result with the sharp regularity (i.e. scale invariant) for both \eqref{eq:dirac c} and the limiting Schr\"odinger equation. Moreover, similar to the massless limit, we obtain convergence of scattering states and wave operators. It is worth noting that the compactness argument from \cite{Nakanishi2002} does not apply to the cubic Dirac equation due to the lack of a coercive conserved energy. Instead we require scale invariant estimates for solutions to \eqref{eq:dirac c} which are uniform in $c\g 1$ in order to be able to conclude a global in time result. Identifying the correct regularity at both low and high frequencies is a key step in our argument, and has not been previously observed in the non-relativistic limit for either the Dirac or Klein-Gordon equations.

Before we come to a precise statement of our results, we first like to motivate the correct limiting equation for \eqref{eq:dirac c}. This is particularly important here as, unlike in the Soler model considered previously in the literature, for general nonlinearities $F$ satisfying the conditions in Remark \ref{rem:nonlinearities}, the non-resonant contributions vanish in the limit $c\to \infty$. In particular, the nonlinearity for the limiting Sch\"odinger equation may differ from the nonlinearity appearing in \eqref{eq:dirac c}. In the case of the Klein-Gordon to Schr\"odinger limit, this has previously been observed in \cite{Masmoudi2002}. We begin by observing that the free Hamiltonian for \eqref{eq:dirac c} is
            $$ c\mc{H}_c = - i c \gamma^0 \gamma^j \p_j + c^2 \gamma^0.$$
Clearly the leading order term is $c^2$, and hence as it stands the Hamiltonian $c\mc{H}_c$ does not converge as $c\to \infty$. In other words, to get a reasonable limit, we must first remove the leading order term. Physically, this corresponds to subtracting off the `rest energy' $c^2$ of the particle (here the mass $m=1$ and so $m c^2 = c^2$). More precisely, as the anti-commutativity properties of the $\gamma^\mu$ matrices implies that $ (c\mc{H}_c)^2 =  - c^2 \Delta +  c^4$, a short computation gives the resolvent identity
        $$ \big( c\mc{H}_c \mp c^2 - z\big)^{-1} = \Big( \frac{1}{2} ( I \pm \gamma^0) \mp \frac{i}{c} \gamma^0 \gamma^j \p_j \Big) \Big( 1 \mp \frac{z^2}{2c^2} \big( \mp \frac{1}{2} \Delta - z\Big)^{-1} \Big)^{-1} \Big( \mp \frac{1}{2} \Delta - z \Big)^{-1} $$
(see \cite{Thaller-1992}) and hence letting $c\to \infty$ we expect that
        $$ (c\mc{H}_c \mp c^2 - z)^{-1} \quad \longrightarrow \quad  E_\pm \Big( \mp \frac{1}{2} \Delta - z \Big)^{-1} $$
where $E_\pm = \frac{1}{2} ( I \pm \gamma^0)$. This limit can also be expressed rather concisely in terms of the free solution operator
        $$\mc{V}_c(t) = e^{-it c\mc{H}_c}$$
by noting that $E_\pm$ is a projection with $E_\pm \gamma^0 = \pm E_{\pm}$ and hence multiplying with the corrective factor $e^{it c^2 \gamma^0}$ and then letting $c\to \infty$ should give
        $$ e^{itc^2 \gamma^0} \mc{V}_c(t)  = E_+ e^{- i t (c\mc{H}_c - c^2)}  + E_- e^{-it (c\mc{H}_c + c^2)} \quad \longrightarrow \quad E_+ e^{ i \frac{t}{2} \Delta} + E_- e^{-i \frac{t}{2} \Delta} = e^{  i \frac{t}{2} \gamma^0 \Delta}=: \mc{V}_\infty(t).$$
In other words, in the non-relativistic limit $c\to \infty$, we expect that after modulating the phase via the factor $e^{itc^2 \gamma^0}$, a solution to the free Dirac equation (i.e. \eqref{eq:dirac c} with $F=0$) should converge to the Schr\"odinger equation
        $$ -i\gamma^0 \p_t \phi - \frac{1}{2} \Delta \phi = 0. $$

The above discussion explains the correct linear behaviour for the limit, but our goal is to understand the non-relativistic limit for the full nonlinear solution. To this end, we note that given any cubic nonlinearity $F$ satisfying the conditions described in Remark \ref{rem:nonlinearities} we can always write
       \begin{equation}\label{eqn:F res decomp}
            F( e^{itc^2 \gamma^0} \psi) =  e^{-3it c^2 \gamma^0} F_{-3}(\psi)+ e^{-it c^2 \gamma^0} F_{-1}(\psi)+ e^{it c^2 \gamma^0} F_{1}(\psi)+e^{3it c^2 \gamma^0} F_{3}(\psi).
          \end{equation}
From \eqref{eqn:F res decomp} we have the decomposition
        \begin{align*}
            \mc{V}_c&(-t) F(\psi) \\
            &= \big( e^{itc^2 \gamma^0} \mc{V}_c(t) \big)^\dagger \Big[ F_1( e^{itc^2 \gamma^0} \psi)  + e^{4it c^2 \gamma^0} F_{-3}(e^{itc^2 \gamma^0}\psi)+  e^{2it c^2 \gamma^0} F_{-1}(e^{itc^2 \gamma^0}\psi) +e^{-2it c^2 \gamma^0} F_{3}(e^{itc^2 \gamma^0}\psi) \Big].
        \end{align*}
In the limit $c\to \infty$, we expect $e^{itc^2\gamma^0} \mc{V}_c(t)$ and $e^{itc^2 \gamma^0} \psi$ to converge and the terms with additional oscillatory factors $e^{k it c^2 \gamma^0}$ to vanish. In particular, the only component of the nonlinearity that contributes is the `resonant' term $F_1$. Consequently, as $c\to \infty$, after correcting the phase via the oscillatory factor $e^{it c^2  \gamma^0}$, solutions to \eqref{eq:dirac c} should converge to solutions to the (cubic) nonlinear Schr\"{o}dinger equation
    \begin{equation}\label{eqn:NLS}
        -i\gamma^0 \p_t \phi - \frac{1}{2} \Delta \phi = F_1(\phi)
    \end{equation}
where the resonant component of the nonlinearity $F_1$ is defined via the formula \eqref{eqn:F res decomp}.
          For the Soler model $F(\psi) = (\overline{\psi} \psi) \psi$ we have $F=F_1$, because $e^{itc^2 \gamma^0} $ and $\gamma^0$ commute, and for the Thirring model
          \[
 F_1(\psi)=(\overline{\psi} \psi) \psi -  \big( \overline{ E_- \psi} \gamma^5  E_+ \psi\big) \gamma^5 E_-\psi -\big(\overline{ E_+ \psi} \gamma^5  E_- \psi\big)  \gamma^5E_+ \psi,
          \]
          see Remark \ref{rmk:res}.

 We remark that we also have a small data global well-posedness and scattering theory both for \eqref{eq:dirac c} and \eqref{eqn:NLS}, similar to the one for \eqref{eq:cd}. In the case of \eqref{eq:dirac c}, this follows immediately from Theorem \ref{thm:gwp} after noting that $\psi$ solves \eqref{eq:dirac c} if and only if $\varphi(t,x) = c^{-1} \psi(c^{-2} t, c^{-1} x)$ solves \eqref{eq:cd} with $m=1$. On the other hand, small data global well-posedness and scattering for the Schr\"odinger equation \eqref{eqn:NLS} is substantially easier and follows by Strichartz estimates and the usual Picard iteration (see Subsection \ref{subsec:proofn1} for the details). To state the uniform global theory slightly more precisely, as above we let $\mc{V}_\infty(t) = e^{ i \frac{t}{2} \gamma^0 \Delta}$ denote the homogeneous solution operator to \eqref{eqn:NLS} and given $1\les c \les \infty$, define
    $$X_c = |\nabla|^{-\sigma_d} \lr{c^{-1} \nabla}^{-\frac{1}{2}} \mc{V}_c(t) \ell^2 U^a ,\qquad \| \psi \|_{X_c} = \big\| |\nabla|^{\sigma_d} \lr{c^{-1} \nabla}^{\frac{1}{2}} \mc{V}_c(-t) \psi \big\|_{\ell^2 U^a}.$$
Here $\lr{a} = \lr{a}_1 = (1+|a|^2)^{\frac{1}{2}}$, and as in the statement of Theorem \ref{thm:gwp}, $\sigma_d = \frac{d-2}{2}$ denotes the scale invariant regularity for the cubic nonlinear Schr\"odinger equation \eqref{eqn:NLS}. Note that for $c<\infty$ we have
    $$ (\mc{V}_c(t) f)(x) = \big( \mc{U}_{c^2}(t) \Lambda_c f\big)( c^{-1} x), \qquad \| \psi \|_{X_c} = \| \Lambda_c \psi \|_{S^{s_d, \sigma_d}_{c^2}},$$
where $(\Lambda_c f)(x) = f(cx)$ denotes the spatial dilation. We then have a global and uniform convergence result which applies to general nonlinearities of the form discussed in Remark \ref{rem:nonlinearities}.

\begin{theorem}[Non-relativistic limit]\label{thm:limit c}
Let $d\in \{2, 3\}$ and $\sigma_d = \frac{d-2}{2}$. There exists $\delta>0$ such that for any $1\les c \les \infty$, and any data $f \in \dot{H}^{\sigma_d}$ satisfying
                $$ \| \lr{c^{-1} \nabla}^{\frac{1}{2}} f \|_{\dot{H}^{\sigma_d}} \les \delta$$
there exists a (unique) global solution $\psi^{(c)}=\psi^{(c)}[f] \in X_c$ to \eqref{eq:dirac c} (for $c<\infty$) and \eqref{eqn:NLS} (for $c=\infty$) on $\R_+$. Moreover, if  $f^{(c)} \in \dot{H}^{\sigma_d}$ with
        $$\| \lr{c^{-1} \nabla}^{\frac{1}{2}} f^{(c)} \|_{\dot{H}^{\sigma_d}} \les \delta, \qquad \text{and} \qquad  \lim_{c\to \infty} \big\| \lr{c^{-1} \nabla}^{\frac{1}{2}} f^{(c)} -  f^{(\infty)} \big\|_{\dot{H}^{\sigma_d}} = 0,$$
then the solutions $\psi^{(c)} = \psi^{(c)}[f^{(c)}] \in X_c$ satisfy
        $$ \lim_{c \to \infty}  \Big\| |\nabla|^{\sigma_d}\Big(  \lr{c^{-1} \nabla}^{\frac{1}{2}}  \mc{V}_c(-\cdot ) \psi^{(c)}  -  \mc{V}_\infty(-\cdot ) \psi^{(\infty)}\Big) \Big\|_{\ell^2 U^a} = 0.$$
\end{theorem}

On global time scales it is not possible to remove the solution operators $\mc{V}_c$ from the convergence statement in Theorem \ref{thm:limit c} due to the distinct asymptotic behaviour of the Dirac and Schr\"odinger equations. However, similar to Corollary \ref{cor:mlimit}, on compact intervals we can prove convergence without pulling back along the linear flow provided we include the corrective factor $e^{i t c^2 \gamma^0}$. In fact, after writing
        $$\mc{V}_c(-t ) \psi^{(c)} = [e^{itc^2 \gamma^0} \mc{V}_c(t)]^{-1} e^{itc^2 \gamma^0} \psi^{(c)}$$
and noting that the symbol of $e^{itc^2 \gamma^0} \mc{V}_c(t)$ converges to $\mc{V}_{\infty}(t)$, a short argument via Theorem \ref{thm:limit c} gives the following non-relativistic counterpart to Corollary \ref{cor:mlimit}.

\begin{corollary}\label{cor:non-rel conv}
Let $I\subset \RR$ be a compact interval. Then the solutions $\psi^{(c)}$ and $\psi^{(\infty)}$ in Theorem \ref{thm:limit c}
satisfy
        $$ \lim_{c\to \infty} \big\| \ind_I(t) \big(  \lr{c^{-1} \nabla}^{\frac{1}{2}} e^{i t c^2  \gamma^0} \psi^{(c)} - \psi^{(\infty)}\big)  \big\|_{L^\infty_t \dot{H}^{\sigma_d}} = 0. $$
\end{corollary}

\begin{remark}
  In subcritical spaces in dimension $d=3$, convergence on bounded time intervals has been obtained in \cite{Machihara2003} for the Soler model. Note that Corollary \ref{cor:non-rel conv} implies for any $\lambda \in 2^\Z$ that
  $$ \lim_{c\to \infty} \big\| \ind_I(t) \big( e^{i t c^2  \gamma^0} \psi^{(c)}_\lambda - \psi^{(\infty)}_\lambda\big)  \big\|_{L^\infty_t L^2_x} = 0,$$
  therefore any uniform bound in higher regularity implies convergence in higher regularity as well.
\end{remark}

As in the case of the  massless limit, Theorem \ref{thm:limit c} also implies that the scattering states and wave operators converge in the non-relativistic limit. Let $1\les c \les \infty$ and define the ball
            $$ \tilde{B}_\delta^{(c)} = \{ f \in \lr{c^{-1}\nabla}^{-\frac{1}{2}} \dot{H}^{\sigma_d} \mid \| \lr{c^{-1} \nabla}^{\frac{1}{2}} f \|_{\dot{H}^{\sigma_d}} \les \delta \}. $$
Given a solution $\psi^{(c)}$ to \eqref{eq:dirac c} (if $c<\infty$) or \eqref{eqn:NLS} (if $c=\infty$), we let $\lr{c^{-1} \nabla}^{\frac{1}{2}} f_{\pm\infty}^{(c)} \in \dot{H}^{\sigma_d}_c$ denote the scattering state at $t=\infty$. Thus
                $$ \lim_{t\to \pm \infty} \big\| \lr{c^{-1} \nabla}^{\frac{1}{2}} \big( \mc{V}_c(-t) \psi^{(c)}(t)  - f_{\pm\infty}^{(c)}\big) \big\|_{\dot{H}^{\sigma_d}} = 0.$$
In analogy with the massless limit, we then take $\tilde{\Omega}_\pm^{(c)}$ to be the operator mapping data $\lr{c^{-1} \nabla}^{\frac{1}{2}}  f\in \dot{H}^{\sigma_d}$ at $t=0$ to the scattering state $ \lr{c^{-1} \nabla}^{\frac{1}{2}} f_{\pm\infty}^{(c)} \in \dot{H}^{\sigma_d}$. As discussed above, there exists $\delta>0$ such that for any $1\les c \les \infty$ the operator $\Omega_\pm^{(c)}:\tilde{B}_\delta^{(c)} \to \lr{c^{-1} \nabla}^{-\frac{1}{2}} \dot{H}^{\sigma_d}$ is well-defined.  Moreover, for $1\les c < \infty$, we have the explicit formulae
         $$
 \tilde{\Omega}_\pm^{(c)}[f]= f + i \int_0^{\pm \infty} \mathcal{V}_c(-t') \gamma^0 F\big(\psi^{(c)}[f]\big)(t')dt', \qquad \tilde{\Omega}_\pm^{(\infty)}[f]= f + i \int_0^{\pm \infty} \mathcal{V}_\infty(-t') \gamma^0 F_1\big(\psi^{(\infty)}[f]\big)(t')dt'.
 $$
Again applying the inverse function theorem \cite[Theorem 15.2]{Deimling1985}, the operators $\tilde{\Omega}_\pm^{(c)}$ are locally invertible, and hence the local inverses $\tilde{W}_\pm^{(c)} = (\Omega_\pm^{(c)})^{-1} : \tilde{B}_\delta^{(c)} \to \lr{c^{-1} \nabla}^{-\frac{1}{2}} \dot{H}^{\sigma_d}$ are also well-defined. After writing
    \begin{align*}
      \| \lr{c^{-1} \nabla}^{\frac{1}{2}}& \Omega_\pm^{(c)}[f^{(c)}] - \Omega_\pm^{(\infty)}[f^{(\infty)}]\|_{\dot{H}^{\sigma_d}} \\
            &\lesa \| \lr{c^{-1}\nabla} ( \mc{V}_c(-t) \psi^{(c)}(t) -  \Omega_\pm^{(c)}[f^{(c)}])\|_{\dot{H}^{\sigma_d}} +\| \mc{V}_\infty(-t) \psi^{(\infty)}(t) -  \Omega_\pm^{(\infty)}[f^{(\infty)}]\|_{\dot{H}^{\sigma_d}}\\
            &\qquad +\| \lr{c^{-1}\nabla}\mc{V}_c(-t) \psi^{(c)}(t) - \mc{V}_{\infty}(-t) \psi^{(\infty)}(t)\|_{L^\infty_t \dot{H}^{\sigma_d}}
    \end{align*}
and letting $t \to \infty$, Theorem \ref{thm:limit c} implies the following.

\begin{corollary}\label{cor:wo-nr}
Let $f^{(c)} \in \tilde{B}_\delta^{(c)}$ and assume that
        $$ \lim_{c\to \infty} \| \lr{c^{-1}\nabla}^{\frac{1}{2}} f^{(c)} - f^{(\infty)}\|_{\dot{H}^{\sigma_d}} = 0. $$
Then we have convergence of the scattering states
$$\lim_{c \to 0}\| \lr{c^{-1} \nabla}^\frac{1}{2} \tilde{\Omega}_\pm^{(c)}[f^{(c)}]- \tilde{\Omega}_\pm^{(\infty)} [f^{(\infty)}] \|_{\dot{H}^{\sigma_d}}=0,$$
and of the wave operators
$$\lim_{c \to 0}\| \lr{c^{-1} \nabla}^\frac{1}{2} \tilde{W}_\pm^{(c)}[f^{(c)}]- \tilde{W}_\pm^{(\infty)} [f^{(\infty)}] \|_{\dot{H}^{\sigma_d}}=0.$$
\end{corollary}

\subsection{Key ideas and novelties}\label{subsec:keyideas}

The proof of Theorem \ref{thm:gwp} is based on an improved atomic version of the classical bilinear $L^2_{t,x}$ estimate for the Klein-Gordon/Wave equation. To motivate the required estimate, we first recall that the scale invariant $L^2_t L^\infty_x$ estimate would suffice to give an easy proof of small data global well-posedness and scattering. However, this bound only holds in dimensions $d\g 4$ and is thus not available here. The standard way to avoid this difficulty is to move to the bilinear setting. For instance, for any $m\in \RR$, and any $\mu, \lambda \in 2^\ZZ$ with $0<\mu \les \lambda $ we have under a suitable transversality condition
    $$ \big\| e^{it\lr{\nabla}_m} f_\mu e^{\pm i t\lr{\nabla}_m} g_\lambda \big\|_{L^2_{t,x}} \lesa \mu^{\frac{d-1}{2}}\Big( \frac{\lr{\lambda}_m}{\lambda}\Big)^{\frac{1}{2}} \| f_\mu \|_{L^2} \| g_\lambda \|_{L^2}, $$
see Lemma \ref{lem:bi L2 free KG}. Note that the implied constant here is independent of $m \in \RR$. This estimate already illustrates one key novelty of our argument, namely the low frequency gain in the Schr\"odinger regime $\lambda \les m$ where the Schr\"odinger scaling $\frac{d-2}{2}$ already appears. Previously, in the massive case $m>0$, this bilinear $L^2_{t,x}$ estimate had only been used in inhomogeneous Sobolev spaces, which obscures the connection to the Schr\"odinger regularity in the limit $m\to \infty$ (which roughly corresponds to the non-relativistic limit).

The second key novelty in our proof of Theorem \ref{thm:gwp} is the use of the atomic bilinear restriction estimates obtained in \cite{Candy2019a} as a suitable generalisation of the bilinear $L^2_{t,x}$ estimate for free solutions introduced above. Working with the spaces $U^a$ instead of the null frame spaces used in earlier work has the distinct advantage that comparing spaces in the massless and massive cases is straightforward. For instance, recalling the solution spaces $S^{s, \sigma}_m$ introduced earlier, we have
        \begin{equation}\label{eqn:conj flows map}
            \big\| \mc{U}_m(t) \mc{U}_0(t) \psi \|_{S^{\sigma, \sigma}_m}= \| |\nabla|^\sigma \mc{U}_m(-t) \big( \mc{U}_m(t) \mc{U}_0(t) \psi\big) \|_{\ell^2 U^a}  = \| |\nabla|^\sigma \mc{U}_0(t) \psi \|_{\ell^2 U^a} = \| \psi \|_{S^{\sigma, \sigma}_0}.
        \end{equation}
In particular we can easily move between different masses by conjugating with the relevant linear flows. This is far from straightforward in function spaces which are not adapted to the linear flow (i.e. Strichartz, null frames). As a further consequence of working in the atomic setting, the additional low frequency gain appears naturally as a consequence of the atomic bilinear restriction estimates obtained in \cite{Candy2019a}. Finally, the dependence of our spaces on the mass $m\in \RR$ is completely explicit which is instrumental in proving uniform in $m$ bounds in both asymptotic limits $m\to 0$ and $m\to \infty$.

The proof of the massless and non-relativistic limits in Theorem \ref{thm:limit} and Theorem \ref{thm:limit c} is rather general. The first key observation is that if we wish to compare solutions globally in time we can not work with the difference $\psi^{(m)} - \psi^{(0)}$ as this clearly does not converge to zero uniformly in $t\in \RR$. Instead we first have to pull back along the linear flows. This simple observation is crucial for obtaining global statements such as convergence of wave operators. Note that the argument given in \cite{Nakanishi2002} in the case of the Klein-Gordon to Schr\"odinger equation relies heavily on the existence of a coercive conserved energy which is not available in the Dirac setting. The second key observation is that the limits are a consequence of three properties: (a) a uniform (in the limit parameter) bound in the relevant adapted function space, (b) convergence on compact time intervals, and (c) large time uniform decay.  After exploiting the fact that we can easily move between say massive and massless spaces via \eqref{eqn:conj flows map}, the uniform bound (a) reduces the problem of convergence to understanding the finite time convergence together with an error term. But these are easily dealt with by (b) and (c) (see Sections \ref{sec:massless-limit} and \ref{sec:nonrelativistic-limit} for a precise statement). The most difficult step in the strategy is (a), namely the correct uniform bound in the adapted function spaces, as this forces bounds which have the correct scaling at both small $m$ and large $m$. This general method of using adapted function spaces for proving limits seems extremely useful, and we expect further applications. Aside from revealing the correct regularity to study the limit, and above method has important consequence of rather immediately implying fundamental asymptotic properties like convergence of scattering states and wave operators. A further novelty is that the resonant decomposition \eqref{eqn:F res decomp} allows us to  treat Dirac equations with general cubic nonlinearities with null-structure in the non-relativistic limit problem.

\subsection{Outline of the paper}
In Section \ref{sec:setup} we introduce notation, discuss the null-structure, and introduce the key functions spaces and their properties which are used throughout this paper. In particular, we give an bilinear interpolation theorem and a convenient operator bound for time dependent operators on $U^p$ and $V^p$ (see Lemma \ref{lem:bil-inter} and Lemma \ref{lem:comp-norms-gen}). Section \ref{sec:str} contains a low frequency refinement of the standard Strichartz estimates (relying on curvature properties only) and an application to a first set of bilinear estimates. In Section \ref{sec:be2} we prove bilinear estimates which rely on bilinear Fourier restriction estimates (also relying on transversality) in atomic function spaces. In Section \ref{sec:mult-est} we prove the crucial frequency localized quadrilinear estimates, which are a key building block for all results in this paper. In Section \ref{sec:proof-gwp} we prove the global nonlinear estimates and provide a proof of Theorem \ref{thm:gwp}.
The results on the massless limit, in particular Theorem \ref{thm:limit} and Corollary \ref{cor:mlimit}, are proved in Section \ref{sec:massless-limit}. Finally, Section \ref{sec:nonrelativistic-limit}
provides all proofs for the results on the non-relativistic limit.

\section{Setup}\label{sec:setup}

Unless specified otherwise, all functions in this paper take values in $\CC^{N_d}$ where $N_d =2 $ if $d=2$ and $N_d =4$ if $d=3$. We define the Fourier transform $\widehat{f}$ of $f\in L^1(\RR^d)$ as
				$$ \widehat{f}(\xi) = \frac{1}{(2\pi)^{\frac{d}{2}} } \int_{\RR^d} e^{ix\cdot \xi} f(x) dx. $$
Throughout this paper, we fix the constants $s_d = \frac{d-1}{2}$ and $\sigma_d = \frac{d-2}{2}$. Note that $\dot{H}^{s_d}$ and $\dot{H}^{\sigma_d}$ are the critical (i.e. scale invariant) norms for the massless and non-relativistic limits of \eqref{eq:cd}. We let $\lr{a}_m = (m^\frac{1}{2} + |a|^2)^{\frac{1}{2}}$ denote the Japanese bracket.

\subsection{Dirac equation}\label{subsec:dirac}
Define the Pauli matrices
    \[
    \sigma^1=\begin{pmatrix}0&1\\1&0
    \end{pmatrix}, \qquad
    \sigma^2=\begin{pmatrix}0&-i\\i&0
    \end{pmatrix}, \qquad
    \sigma^3=\begin{pmatrix}1&0\\0&-1
    \end{pmatrix}.
    \]
We take the Dirac representation of the gamma matrices, namely
    \[
    \gamma^0=\diag(1,1,-1,-1), \quad
    \gamma^j=\begin{pmatrix} 0 & \sigma^j\\
    -\sigma^j & 0
    \end{pmatrix},  \quad \gamma^5:=i\gamma^0\gamma^1\gamma^2\gamma^3, \quad \text{ if } \quad d=3
    \]
and
    $$\gamma^0 = \sigma^3, \qquad \gamma^1 = i \sigma^2, \qquad \gamma^2 = - i \sigma^1 \quad  \text{ if } \quad d=2.$$
We take $\mc{H}_m$ to denote the Dirac Hamiltonian
			$$ \mc{H}_m := -i \gamma^0 \gamma^j \p_j + m \gamma^0$$
and define the Fourier multipliers
			$$ \Pi_\pm := \frac{1}{2} \Big( I \mp \frac{1}{\lr{\nabla}_m} \mc{H}_m \Big) $$
The anti-commutativity properties of the matrices $\gamma^\mu$ implies that $ (\mc{H}_m )^2 = - \Delta + m^2 = \lr{\nabla}_m^2$ and $\mc{H}_m^\dagger = \mc{H}_m$. In particular we have $I = \Pi_+ + \Pi_-$ and a short computation gives
    $$ \Pi_\pm^2 = \Pi_\pm = \Pi_\pm^\dagger, \qquad \Pi_+ \Pi_- = \Pi_- \Pi_+ = 0, \qquad \Pi_\pm \mc{H} = \mc{H} \Pi_\pm = \mp \lr{\nabla}_m \Pi_\pm.  $$
We let $\mc{U}(t) := e^{-it \mc{H}_m}$ denote the free Dirac propagator. In view of the above properties of the projections $\Pi_\pm$, we can write the homogeneous solution operator $\mc{U}(t)$ in terms of the more familiar Klein-Gordon propagators $e^{\pm it \lr{\nabla}_m}$ via the identity
	  \begin{equation}\label{eqn:lin dirac flow}
	    \mc{U}_m(t) = e^{it\lr{\nabla}_m} \Pi_+ + e^{-it\lr{\nabla}_m} \Pi_-.
    \end{equation}
Given a spinor $\psi$, we define
        $$ \psi_\pm := \Pi_\pm \psi . $$
Note that we may decompose $\psi = \psi_+  + \psi_-$. Moreover, if $\psi$ solves the (inhomogeneous) Dirac equation
    $$ -i \gamma^\mu \psi  + m \psi = F$$
then a short computation gives
    $$  (i \p_t \pm \lr{\nabla}_m) \psi_\pm  = - \Pi_\pm \gamma^0 F. $$
In particular, if $\psi = \mc{U}_m(t) f$ is a solution to the free Dirac equation, we can write
	$$ \psi(t) = \psi_+ + \psi_- = e^{it\lr{\nabla}_m} f_+ + e^{-it\lr{\nabla}_m} f_-.  $$

As in  \cite{dfs2007,Bejenaru2014a,Bejenaru2016,Bournaveas2015}, the arguments in this paper heavily exploit the null structure exhibited by the Soler and Thirring models. To explain this point in more detail, by an abuse of notation, given $\xi \in \RR^d$ we let $\Pi_\pm(\xi)$ denote the symbol of $\Pi_\pm$. Explicitly we have
			$$ \Pi_\pm(\xi) = \frac{1}{2} \Big( I \mp \frac{1}{\lr{\xi}_m}\gamma^0 (  \gamma^j \xi_j +  m) \Big). $$
Since $\Pi_+(\xi)$ and $\Pi_-(\xi)$ are orthogonal projections, a computation gives
    \begin{align*}
         \big| \Pi_{\pm_1}(\xi)^\dagger \gamma^0  \Pi_{\pm_2}(\eta) \big| &= \big| \Pi_{\pm_1}(\xi) \big( \Pi_{\pm_1}(\xi)  \gamma^0 - \gamma^0 \Pi_{\mp_2}(\eta) \big)  \Pi_{\pm_2}(\eta) \big| \\
         &\lesa \big| \Pi_{\pm_1}(\xi)  \gamma^0 - \gamma^0 \Pi_{\mp_2}(\eta) \big|\\
         &\lesa \Big| \frac{\pm_1 \xi}{\lr{\xi}_m} - \frac{\pm_2 \eta}{\lr{\eta}_m}\Big| +  \Big| \frac{ \pm_1 m}{\lr{\xi}_m} + \frac{\pm_2 m}{\lr{\eta}_m}\Big|
    \end{align*}
and hence we have the null structure bounds\footnote{In the special case $\pm_1 = - \pm_2$, we have the slight improvement
            $$\big| \Pi_\pm(\xi) \gamma^0 \Pi_\mp(\eta) \big| \lesa \frac{|\xi| |\eta|}{\lr{\xi}_m \lr{\eta}_m} \ma(\xi, -\eta) + \frac{m}{\lr{\xi}_m \lr{\eta}_m} \big| |\xi| - |\eta| \big| $$
although it is not needed in the arguments in this paper. }
    \begin{equation}\label{eqn:Pi null struc}
        \begin{split}
        \big| \Pi_{\pm_1}(\xi)^\dagger \gamma^0 \Pi_{\pm_2}(\eta) \big| &\lesa \frac{|\xi| |\eta|}{\lr{\xi}_m \lr{\eta}_m} \ma(\pm_1\xi, \pm_2\eta) + \frac{|m|}{\lr{\xi}_m} + \frac{|m|}{\lr{\eta}_m}
        \end{split}
    \end{equation}
where the implied constant is independent of the mass $m\in \RR$. Here and in the sequel, we use
\[\ma(\xi, \eta):=\sqrt{1-\frac{\xi}{|\xi|}\cdot \frac{\eta}{|\eta|}}\]
to denote the angle between the vectors $\xi, \eta \in \RR^d \setminus \{0\}$. In particular, if $\psi$ and $\varphi$ are free Dirac waves, then the product
			$$ \overline{\psi_{\pm_1}} \varphi_{\pm_2} = (\psi_{\pm_1})^\dagger \gamma^0 \varphi_{\pm_2}    $$
has additional cancellation when $\psi_{\pm_1}$ and $\varphi_{\pm_2}$ propagate in either parallel (if $\pm_1 =\pm_2$) or orthogonal (if $\pm_1 \not = \pm_2$) directions. As this is the worst case interaction for free waves, the product $\overline{\psi} \varphi$ decays significantly faster then generic products like $|\psi|^2$. See for instance \cite{dfs2007} where null structure and its consequences for the Dirac equation are well illustrated.

We end this subsection with two remarks concerning the nonlinearity $F$. The first explains why the general condition \eqref{eqn:null cond} suffices in our arguments, while the second explicitly computes the resonant contribution $F_1$ for the Soler and Thirring models.

\begin{remark}[General nonlinearities]\label{rem:nonlinearities II}
In Sections \ref{sec:str},  \ref{sec:be2}, and \ref{sec:mult-est}, we only explicitly work with the Soler model nonlinearity $\overline{\psi} \varphi$. However all bilinear and multilinear bounds in these sections continue to hold in the case where the products $\overline{\psi} \varphi$ are replaced with $\overline{\psi} \mb{A} \varphi$ for any constant matrix $\mb{A} \in \CC^{N_d \times N_d}$ satisfying \eqref{eqn:null cond}. This follows due to the fact that all bilinear and multilinear estimates used in the proof of our main results only rely on the null structure bound \eqref{eqn:Pi null struc} and this bound continues to hold for the product $\Pi_{\pm_1}^\dagger(\xi) \gamma^0 \mb{A} \Pi_{\pm_2}(\eta)$. More precisely, the null condition \eqref{eqn:null cond} implies that
		$$ \big| \Pi_{\pm_1}(\xi)^\dagger \gamma^0 \mb{A} - \gamma^0 \mb{A} \Pi_{\mp_2}(\eta) \big| \lesa_{\mb{A}} \Big| \frac{\pm_1 \xi}{\lr{\xi}_m} - \frac{\pm_2 \eta}{\lr{\eta}_m}\Big| +  \Big| \frac{ \pm_1 m}{\lr{\xi}_m} + \frac{\pm_2 m}{\lr{\eta}_m}\Big|  $$
and hence repeating the derivation of \eqref{eqn:Pi null struc} gives
	\begin{align*}
		\big| \Pi_{\pm_1}(\xi)^\dagger \gamma^0 \mb{A} \Pi_{\pm_2}(\eta) \big|  \lesa \frac{|\xi| |\eta|}{\lr{\xi}_m \lr{\eta}_m} \ma(\pm_1\xi, \pm_2\eta) + \frac{|m|}{\lr{\xi}_m} + \frac{|m|}{\lr{\eta}_m}.
	\end{align*}
In particular, \eqref{eqn:Pi null struc} continues to hold with $\Pi_{\pm_1}^\dagger(\xi) \gamma^0 \Pi_{\pm_2}(\eta)$ replaced with $\Pi_{\pm_1}^\dagger(\xi) \gamma^0 \mb{A} \Pi_{\pm_2}(\eta)$. Consequently
the bilinear form $\overline{\psi} \mb{A} \psi$ is also a null form, and hence obeys all estimates in Sections \ref{sec:str} and \ref{sec:be2}. As a quick application, note that $\gamma^5$ satisfies \eqref{eqn:null cond} in the case $d=3$. In particular,  as the Fierz identities \cite{Nieves2004,Bournaveas2015} allow us to write
       \begin{equation}\label{eqn:thirring mod}
        (\overline{\psi} \gamma^\mu \psi) \gamma_\mu \psi = \begin{cases} (\overline{\psi} \psi) \psi, &d=2,\\ ( \overline{\psi} \psi ) \psi -( \overline{\psi} \gamma^5 \psi) \gamma^5 \psi, &d=3, \end{cases}
       \end{equation}
we conclude that our arguments also apply to the Thirring nonlinearity $F(\psi) = (\overline{\psi} \gamma^\mu \psi) \gamma_\mu \psi$.
\end{remark}

\begin{remark}\label{rmk:res}
In the two primary cases of interest, namely the Soler and Thirring models, we explicitly compute the resonant contribution $F_1$ defined via the formula \eqref{eqn:F res decomp}. Clearly, for the Soler model $F(\psi) = (\overline{\psi} \psi) \psi$ we have $F=F_1$, because $e^{itc^2 \gamma^0} $ and $\gamma^0$ commute. For the Thirring model, in view of \eqref{eqn:thirring mod}, it suffices to consider the $d=3$ and the $\gamma^5$ contribution $G(\psi)=(\overline{\psi}\gamma^5\psi) \gamma^5\psi$:
We observe that  $\gamma^0\gamma^5=-\gamma^5\gamma^0$, and hence $e^{itc^2\gamma^0}\gamma^5=\gamma^5e^{-itc^2\gamma^0}$, $\gamma^5 E_+ = E_- \gamma^5 $, and
        \[
 G( e^{itc^2 \gamma^0} \psi) =e^{ - i t c^2 \gamma^0} \big( \overline{\psi} \gamma^5  e^{2itc^2 \gamma^0} \psi \big) \gamma^5 \psi.
\]
Therefore
$$  \overline{\psi} \gamma^5  e^{2itc^2 \gamma^0} \psi =e^{2itc^2}  \overline{ E_- \psi} \gamma^5  E_+ \psi+e^{-2itc^2}  \overline{ E_+ \psi} \gamma^5  E_- \psi,
$$
which implies
\begin{align*}
G( e^{itc^2 \gamma^0} \psi) =&e^{itc^2} \big( \overline{ E_- \psi} \gamma^5  E_+ \psi\big)E_+\gamma^5\psi +e^{-3itc^2}  \big(\overline{ E_+ \psi} \gamma^5  E_- \psi \big)\gamma^5 E_+ \psi
  \\&{}+e^{3itc^2} \big( \overline{ E_- \psi} \gamma^5  E_+ \psi\big)E_- \gamma^5\psi +e^{-itc^2}  \big(\overline{ E_+ \psi} \gamma^5  E_- \psi\big)E_- \gamma^5 \psi\\
  =& e^{ i t c^2 \gamma^0}  \Big(\big( \overline{ E_- \psi} \gamma^5  E_+ \psi\big) \gamma^5 E_-\psi +\big(\overline{ E_+ \psi} \gamma^5  E_- \psi\big)  \gamma^5E_+ \psi \Big)\\
  &{}+ e^{ -3 i t c^2 \gamma^0}\Big(\big(\overline{ E_+ \psi} \gamma^5  E_- \psi \big)\gamma^5 E_-\psi+\big( \overline{ E_- \psi} \gamma^5  E_+ \psi\big)  \gamma^5 E_+\psi \Big).
\end{align*}
By \eqref{eqn:thirring mod}, for the full Thirring nonlinearity in $d=3$, namely $F(\psi) = (\overline{\psi} \gamma^\mu \psi) \gamma_\mu \psi$, we obtain the decomposition \eqref{eqn:F res decomp} with
\begin{align*}
F_{-3}(\psi)=&\big(\overline{ E_+ \psi} \gamma^5  E_- \psi \big)\gamma^5 E_-\psi+\big( \overline{ E_- \psi} \gamma^5  E_+ \psi\big)  \gamma^5 E_+\psi \\
 F_1(\psi)=&(\overline{\psi} \psi) \psi -  \big( \overline{ E_- \psi} \gamma^5  E_+ \psi\big) \gamma^5 E_-\psi -\big(\overline{ E_+ \psi} \gamma^5  E_- \psi\big)  \gamma^5E_+ \psi
\end{align*}
and $F_{-1}=F_3=0$.
\end{remark}

\subsection{Fourier Multipliers}\label{subsec:fm}

Given $0 < \alpha \les 1$ we let $\mc{C}_\alpha$ denote the collection of collection of finitely overlapping caps $\kappa \subset \sph^{d-1}$ of radius $\alpha$. For $\lambda, d>0$ and $\kappa \in C_{\alpha}$, we define the Fourier multipliers  $P_\lambda$, $R_{\kappa}$, and $C^{m, \pm}_{\les d}$ to restrict the Fourier transform to the sets
        $$ \{|\xi| \approx \lambda\}, \qquad \{ \tfrac{\xi}{|\xi|} \in \kappa \}, \qquad \{ |\tau - \pm \lr{\xi}_m |\les d \}. $$
 As usual, we choose the multipliers so that they are uniformly bounded in $L^p$ (disposable).
Similar, we define the multipliers $P_{\lambda, \kappa}$ and $C^{(m)}_{\les d}$ acting on spinors $\psi$ as
        $$  P_{\lambda, \kappa}  = P_\lambda R_\kappa \Pi_+  + P_\lambda R_{-\kappa} \Pi_- , \qquad C^{(m)}_{\les d} = C^{m,+}_{\les d} \Pi_+ + C^{m, -}_{\les d} \Pi_- $$
and we take $C_{\g d}^{(m)} = I - C_{\les d}^{(m)}$. We typically suppress the mass parameter $m$, and simply write  $C_{\les d} = C^{(m)}_{\les d}$. We introduce the short hand
            $$ f_\lambda = P_\lambda f, \qquad f_{\lambda, \kappa} = P_{\lambda, \kappa} f.$$
By choosing the multipliers $P_\lambda$ and $R_{\kappa}$ appropriately, we can assume that we have the pointwise decomposition
$$
 f = \sum_{\lambda \in 2^{\ZZ}} f_\lambda $$
 and the angular Whitney type decomposition
\begin{equation}\label{eqn:whitney decom}
    \qquad (f_\lambda)^\dagger g_\mu = \sum_{\alpha \in 2^{-\NN}} \sum_{\substack{ \kappa, \tilde{\kappa} \in \mc{C}_\alpha \\ \ma(\kappa, \tilde{\kappa}) + \frac{|m|}{\lr{\min\{\lambda, \mu\}}_m} \approx \alpha }} (f_{\lambda, \kappa})^\dagger  g_{\mu, \tilde{\kappa}}
            \end{equation}
            for  any $f, g \in L^2(\RR^d)$, together with the square sum bound
		$$ \sum_{\kappa \in \mc{C}_\alpha } \|f_{\lambda, \kappa}\|_{L^2}^2 \lesa \| f_\lambda \|_{L^2}^2. $$
Define the set
            $$
            \mathcal{W}_{\alpha}^{(m)}(\lambda,\mu):=\Big\{(\kappa, \tilde{\kappa} )\in \mc{C}_\alpha\times \mc{C}_\alpha: \ma(\kappa, \tilde{\kappa}) + \frac{|m|}{\lr{\min\{\lambda, \mu\}}_m} \approx \alpha.\Big\}
            $$
 To explain \eqref{eqn:whitney decom}, we first observe that $m=0$ this is the standard Whitney decomposition. In the low frequency case $0<\min\{\lambda, \mu\}\lesa |m|$, there is only a trivial decomposition into a finite number of caps of size $\alpha \sim 1$ (effectively no decomposition), whereas if $0<|m|\ll \min\{\lambda, \mu\}$, then
  $\mathcal{W}_{\alpha}^{(m)}(\lambda,\mu)$ is the set of all caps $(\kappa, \tilde{\kappa} )\in \mc{C}_\alpha\times \mc{C}_\alpha$ such that
  $$ \ma(\kappa, \tilde{\kappa}) \sim \alpha \text{ if }\alpha \gg |m| (\min\{\lambda, \mu\})^{-1}, \quad \ma(\kappa, \tilde{\kappa})  \lesa \alpha  \text{ if }\alpha \sim |m| (\min\{\lambda, \mu\})^{-1}, $$
  and
$\mathcal{W}_{\alpha}^{(m)}(\lambda,\mu)=\emptyset$ if $\alpha \ll |m| (\min\{\lambda, \mu\})^{-1}$.

For later use, we need to understand the Fourier support of a product of two spinors that are localised close to the hyperboloid/cone.

\begin{lemma}\label{lem:Fourier supp}
Let $\lambda \g \mu >0$ and $0<\alpha \les 1$. Let $\kappa, \tilde{\kappa} \in \mc{C}_\alpha$ with $ \frac{|m|}{\lr{\mu}_m} + \ma(\kappa, \tilde{\kappa}) \approx \alpha$.
Then
        $$ \supp \mc{F}_{t,x}\big[\overline{C_{\ll \alpha^2 \lr{\mu}_m} \varphi_{\lambda, \kappa}} C_{\ll \alpha^2 \lr{\mu}_m} \psi_{\mu, \tilde{\kappa}}\big]\subset \big\{ \big| |\tau|^2 - \lr{\xi}_m^2\big| \approx \alpha^2 \lr{\mu}_m \lr{\lambda}_m  \big\}$$
and, for any $\beta \gtrsim \alpha^2 \lr{\mu}_m$,
         $$ \supp\mc{F}_{t,x}\big[ \overline{C_{\les \beta} \varphi_{\lambda, \kappa}} C_{\les \beta} \psi_{\mu, \tilde{\kappa}}\big] \subset \big\{ \big| |\tau|^2 - \lr{\xi}_m^2\big| \lesa \beta (\beta + \lr{\lambda}_m) \big\}. $$
\end{lemma}
\begin{proof} We begin by noting that for any $\xi, \eta \in \RR^d$
        \begin{align*}
          \big| |\pm_1 \lr{\xi}_m &- \pm_2 \lr{\eta}_m|^2 - \lr{\xi - \eta}_m^2 \big| \\
                        &= \big| 2\big( \lr{\xi}_m \lr{\eta}_m - |\xi| |\eta| - m^2\big) + (2 - \pm_1 \pm_2 1)m^2 + 2\big( |\xi||\eta| - (\pm_1 \xi) \cdot (\pm_2 \eta)\big) \big| \\
                        &= \frac{ 2m^2 | |\xi| - |\eta| |^2 }{\lr{\xi}_m \lr{\eta}_m + m^2 + |\xi| |\eta|} + (2 - \pm_1 \pm_2 1)m^2 +  2\big( |\xi||\eta| - (\pm_1 \xi) \cdot (\pm_2 \eta)\big).
        \end{align*}
Consequently, we conclude that if $(\tau, \xi) \in \supp \widetilde{\Pi_{\pm_1}C_{\les \beta} \varphi_{\lambda, \kappa}}$ and $(\tilde{\tau}, \eta) \in \supp \widetilde{\Pi_{\pm_2}C_{\les \beta}\psi_{\mu, \tilde{\kappa}}}$, then we have
        $$ \big| |\pm_1 \lr{\xi}_m - \pm_2 \lr{\eta}_m|^2 - \lr{\xi - \eta}_m^2 \big| \approx \frac{ 2m^2 (| |\xi| - |\eta| |^2+\lr{\lambda}_m \lr{\eta}_m) }{\lr{\lambda}_m \lr{\mu}_m}  +  2\lambda \mu \ma^2(\kappa, \tilde{\kappa}) \approx  \alpha^2 \lr{\mu}_m \lr{\lambda}_m $$
and
    \begin{align*}
      \big| |\tau - \tilde{\tau}|^2 & - |\pm_1 \lr{\xi}_m - \pm_2 \lr{\eta}_m|^2 \big| \\
        &\les \Big( |\tau - \pm_1 \lr{\xi}_m| + |\tilde{\tau} - \pm_2 \lr{\eta}_m| + \lr{\xi}_m + \lr{\eta}_m\Big)\Big( |\tau - \pm_1 \lr{\xi}_m| + |\tilde{\tau} - \pm_2 \lr{\eta}_m|\Big) \lesa (\beta + \lr{\lambda}_m) \beta,
    \end{align*}
    as claimed.
\end{proof}

\subsection{The spaces $U^p$ and $V^p$}\label{subsec:fs}
The proof of Theorem \ref{thm:gwp} involves an iteration argument in suitably chosen solution space $S$. There are a number of choices that work here. One approach is use the null frame spaces introduced by Tataru \cite{Tataru2001}, as in \cite{Bejenaru2016,Bejenaru2014a,Bournaveas2015}. However, as explained in the introduction, we instead work in the adapted $U^p/V^p$ framework  and aim for a unified approach that allows us to treat the cases $m=0$ and $m\not =0$ simultaneously. Moreover, this approach has the advantage that it is well suited to the existence of the massless and non-relativistic limits. In the following we briefly outline the properties of the $U^p/V^p$ spaces, and refer to the papers \cite{Candy2018b,Hadac2009,Koch2005,Koch2014, Koch2018} for further details.

Let $1<p<\infty$. We say that a function $u:\RR^{1+d} \to \CC$ is a step-function (write $u \in \mathfrak{S}$) if there exists a finite sequence of times $t_1<t_2< \dots< t_N$ such that
        $$ u(t, x) = \sum_{j=1}^N \ind_{[t_j, t_{j+1})}(t) f_j(x)$$
where $t_{N+1} = \infty$ and $f_j:\RR^d \to \CC$. A step function $u\in \mathfrak{G}$ is called a  $U^p$-\emph{atom} if in addition we have the normalisation
        $$ \Big( \sum_{j=1}^N \| f_j \|_{L^2_x(\RR^d)}^p \Big)^{\frac{1}{p}} = 1.$$
We then define the atomic space $U^p$ as all possible sums of atoms, thus
        $$ U^p = \Big\{ \sum_{j\in \NN} c_j u_j \,\, \Big|\,\, (c_j)_{j\in \NN} \in \ell^1(\NN), \text{ $u_j$ is a $U^p$ atom} \Big\} $$
and equip $U^p$ with the obvious induced norm
        $$ \| u \|_{U^p} = \inf\big\{ \sum_{j\in \NN} |c_j| : u = \sum_{j\in \NN} c_j u_j, \; \text{ with } U^p\text{-atoms } u_j \big\}.$$

The dual to $U^p$ can be described using the spaces of finite $p$-variation $V^p$. More precisely, define the norm
        $$ \| v \|_{V^p} = \| v \|_{L^\infty_t L^2_x}  + |v|_{V^p}, \text{ where } |v|_{V^p}:=\sup_{t_1<t_2 < \dots < t_N} \Big( \sum_{j=1}^{N-1} \| v(t_{j+1}) - v(t_j)\|_{L^2_x}^p \Big)^{\frac{1}{p}}.$$
Here the supremum is over all finite increasing sequences. The Banach space $V^p$ is then defined to be collection of all right continuous $v:\RR^{1+d} \to \CC$ such that $\|v\|_{V^p} < \infty$ and $\lim_{t\to -\infty} v(t) = 0$ (in $L^2_x$). This last property is a normalisation condition, and ensures that the dual pairing with $U^p$ is well-defined. The above definitions can be extended to $\CC^N$ valued functions (i.e. spinors) or Hilbert space-valued functions in the obvious manner. The above definition  implies that if $u\in U^p$ (or $u \in V^p$) then $u(t)$ converges in $L^2_x$ as $t\to \infty$.

\begin{lemma}[Basic Properties of $U^p$ and $V^p$]\label{lem:Up basic prop}
Let $1<p\les q < r$.
\begin{enumerate}
    \item We have the continuous embeddings
	$$ U^p \subset V^q \subset U^r \subset L^\infty_tL^2_x.$$
    \item If $u\in L^p_t L^2_x$ has temporal Fourier support in the region $\{ |\tau| \approx \alpha\}$, then $u \in U^p$ and
		$$  \| u \|_{U^p} \approx \| u \|_{V^p} \approx \alpha^{\frac{1}{p}} \| u \|_{L^p_t L^2_x}. $$
    \item Let $\mc{L}(L^2)$ denote collection of all bounded operators on $L^2$. If $A\in \mc{L}(L^2)$ then
                $$ \| A u \|_{U^p} \les \|A\|_{\mc{L}(L^2)} \| u \|_{U^p}, \qquad \| Av \|_{V^p} \les \| A \|_{\mc{L}(L^2)} \| v \|_{V^p}.$$
    Moreover, for any $\alpha \in 2^{-\NN}$ we have the square sum bound
                    $$ \Big( \sum_{\kappa \in R_\alpha} \| R_\kappa u\|_{U^p}^2 \Big)^{\frac{1}{2}} \lesa \alpha^{-(d-1) \max\{0, \frac{1}{2}-\frac{1}{p}\}} \| u \|_{U^p}. $$
    \item If $u\in L^\infty_t L^2_x$ is right continuous with $u(t) \to 0$ as $t\to -\infty$, and
			$$ D(u) = \sup_{\substack{\phi \in C^\infty_0 \\ \|\phi\|_{V^{p'}}\les 1}} \Big| \int_{\RR^{1+d}} \p_t \phi u \,dt dx \Big|< \infty,$$
            then $u\in U^p$ and $\|u\|_{U^p} \approx D(u)$.
\end{enumerate}
\end{lemma}
\begin{proof}
The first properties can be found in (for instance) \cite{Hadac2009}. For the final dual pairing characterisation of $U^p$  see \cite[Theorem 5.1]{Candy2018b} and \cite{Koch2018}.
\end{proof}

We can apply the method of complex interpolation in a standard way, see e.g.\ \cite{Bergh1976}. \begin{lemma}\label{lem:bil-inter}
Suppose that $1\les p_j\les \infty$ and $1\les a_j,b_j<\infty$ for $j=0,1$. Also, suppose that $T:\mathfrak{S}\times \mathfrak{S}\to L^1_{loc}(\R^{1+d})$ is a bilinear or sesquilinear operator and for $j=0,1$ there exists $C_j$ such that
\begin{equation}\label{eq:bil-int-ass}
\|T(u,v)\|_{L^{p_j}_{t,x}}\les C_j \|u\|_{U^{a_j}} \|v\|_{U^{b_j}}
\end{equation}
for all step-functions $u,v \in \mathfrak{S}$.
Then, for $0\les \theta\les 1$ and $$(\frac1p,\frac1a,\frac1b)=(1-\theta)(\frac1p_0,\frac1a_0,\frac1b_0)+\theta (\frac1p_1,\frac1a_1,\frac1b_1)$$ we have
$$\|T(u,v)\|_{L^{p}_{t,x}}\les C_0^\theta C_1^{1-\theta} \|u\|_{U^{a}} \|v\|_{U^{b}}$$
for all $u \in U^a$ and $v \in U^b$.
\end{lemma}
\begin{proof}
 Let $T$ be bilinear. If $p=1$ the proof is easier, so let us suppose that $p>1$. Consider an elementary function
 $$F(t,x)=\sum_{m=1}^M e_m \chi_{E_m}(t,x) , \; E_m \subset  \R^{1+d}, \text{ pairwise disjoint, of finite measure},$$
 such that $ \|F\|_{L^{p'}_{t,x}}\les 1,$
 a $U^{a}$-atom $u=\sum_{k=1}^K \ind_{I_k} f_k$, and a $U^b$-atom $v=\sum_{l=1}^L \ind_{J_l} g_l$.
 Define, for $z \in \C$,
 $$(\frac1{p(z)},\frac1{a(z)},\frac1{b(z)})=(1-z)(\frac1p_0,\frac1a_0,\frac1b_0)+z (\frac1p_1,\frac1a_1,\frac1b_1),$$
 the step-functions (setting $\frac{0}{0}=0$ within this proof)
 $$
 u_z= \sum_{k=1}^K \ind_{I_k} \|f_k\|_{L^2_x}^{\frac{a}{a(z)}}\frac{f_k}{\|f_k\|_{L^2_x}}, \quad v_z= \sum_{l=1}^L \ind_{J_l} \|g_l\|_{L^2_x}^{\frac{b}{b(z)}}\frac{g_l}{\|g_l\|_{L^2_x}},
 $$
 and
 $$F_z=\sum_{m=1}^M  |e_m|^{\frac{p'}{p(z)'}}\frac{e_m}{|e_m|}  \chi_{E_m} $$
 and the holomorphic function
 $$
\Gamma:\C \to \C, \; \Lambda(z)=\int_{\R^{1+d}} T(u_z,v_z) F_z dtdx
 $$
 Note that $\Gamma$ is a finite linear combination of terms of the form $r^z$ for real numbers $r>0$, and because of $|r^{z}|=r^{\Real z}$ it is bounded and continuous within the strip $0\les \Real z \les 1$. Further, for $y \in \R$, $\Real p(iy)=p_0$ and therefore
 $$
 \|F_{iy}\|_{L^{p_0'}_{t,x}}^{p_0'} \les \sum_{m=1}^M \Big| |e_m|^{\frac{p'}{p(iy)'}}\Big|^{p_0'} |E_m| \les 1.
 $$
 Similarly,
 $$
 \|u_{iy}\|_{U^{a_0}}\les \sum_{k=1}^K \Big\| \|f_k\|_{L^2_x}^{\frac{a}{a(iy)}}\frac{f_k}{\|f_k\|_{L^2_x}}\Big\|_{L^2}^{a_0}=\sum_{k=1}^K \|f_k\|_{L^2_x}^{a}=1, \text{ and } \|v_{iy}\|_{U^{b_0}}\les 1.
 $$
 Therefore, assumption \eqref{eq:bil-int-ass} with $j=0$ implies
 $$
 |\Gamma(iy)|\les \|T(u_{iz},v_{iz})\|_{L^{p_0}_{t,x}}\|F\|_{L^{p_0'}_{t,x}}\les C_0.
 $$
 An analogous argument, using  \eqref{eq:bil-int-ass} with $j=1$, implies
 $$
 |\Gamma(1+iy)|\les \|T(u_{1+iz},v_{1+iz})\|_{L^{p_1}_{t,x}}\|F\|_{L^{p_1'}_{t,x}}\les C_1.
 $$
 Now, Hadamard's Three Lines Lemma \cite[Lemma 1.1.2]{Bergh1976} implies the bound
 $$
 \Big|\int_{\R^{1+d}}T(u,v) F dtdx \Big|=|\Gamma(\theta)|\les \sup_{0\les \Real z\les 1}|\Gamma(z)|\les C_0^\theta C_1^{1-\theta}.$$
 By duality, density, and the atomic structure of $U^{a}$ and $U^b$ the claim follows.

 If $T$ is sesquilinear, the operator $(u,v)\mapsto T(\overline{u},{v})$ (or $T(u,\overline{v})$) is bilinear. Since for any $b$ the $U^{b}$-norm is invariant under complex conjugation, the result also follows in this case.
\end{proof}

We remark that Lemma \ref{lem:bil-inter} can be generalized to mixed-norms $L^p_tL^r_x$ by a slight modification of the proof, but we will not need this here.

In the course of the proof of Theorem \ref{thm:gwp}, we decompose into frequencies $\lambda \in 2^\ZZ$ and consider frequency localised bilinear and multilinear estimates. As the $U^p$ norms only partially commute with square sums (see for instance (iii) in Lemma \ref{lem:Up basic prop}), decomposing and recombining square sums becomes technically inconvenient. Thus we instead introduce a more tractable `Besov' version of $U^p$ with an outer square sum over dyadic frequencies.

\begin{definition}\label{def:square-sum-upvp}
We define
        $$ \ell^2 U^p = \Big\{ u \in L^\infty_t L^2_x \,\, \Big|\,\,  u_\nu \in U^p \text{ and } \sum_{\nu \in 2^\ZZ} \| u_\nu \|_{U^p}^2<\infty \Big\} $$
with norm
        $$ \| u \|_{\ell^2 U^p} = \Big( \sum_{\nu \in 2^{\ZZ}} \| u_\nu \|_{U^p}^2 \Big)^{\frac{1}{2}}. $$
Similarly, we define
        $$ \ell^2 V^p = \Big\{ v \in L^\infty_t L^2_x \,\, \Big|\,\,  v_\nu \in U^p \text{ and } \sum_{\nu \in 2^\ZZ} \| v_\nu\|_{V^p}^2<\infty \Big\} $$
with the norm
        $$ \| v \|_{\ell^2 V^p} = \Big( \sum_{\nu \in 2^{\ZZ}} \| u_\nu \|_{U^p}^2 \Big)^{\frac{1}{2}}. $$
        \end{definition}

The next lemma provides a general result on the boundedness of time dependent linear operators on $U^p$, which generalizes Lemma \ref{lem:Up basic prop}, Part (iii).

\begin{lemma}\label{lem:comp-norms-gen}
Let $1\les p < \infty$, $A\in U^p(\R,\mathcal{L}(L^2))$, and $B \in V^p(\R, \mc{L}(L^2))$. Then, for all $u \in U^p(\R,L^2)$ and $v\in V^p(\R, L^2)$ we have $Au\in U^p(\R,L^2)$, $Bv \in V^p(\R, L^2)$ and
    $$\|Au\|_{U^p(\R,L^2)}\les   \|A\|_{U^p(\R,\mathcal{L}(L^2))} \|u\|_{U^p(\R,L^2)}, \qquad \| B v\|_{V^p(\R, L^2)} \les \| B\|_{V^p(\RR, \mc{L}(L^2))} \| v \|_{V^p(\RR, L^2)}. $$
    In addition, if $A$ respectively $B$  commutes with the Fourier multipliers $P_\lambda$ for all $\lambda \in 2^\ZZ$, then
     $$\|Au\|_{\ell^2 U^p(\R,L^2)}\les   \|A\|_{U^p(\R,\mathcal{L}(L^2))} \|u\|_{\ell^2 U^p(\R,L^2)}, \qquad \| B v\|_{\ell^2 V^p(\R, L^2)} \les \| B\|_{V^p(\RR, \mc{L}(L^2))} \| v \|_{\ell^2 V^p(\RR, L^2)}. $$
\end{lemma}
\begin{proof} We start by proving the $U^p$-bounds.
Consider $U^p$-atoms
\[u=\sum_j \ind_{[\tau_j,\tau_{j+1})}f_j \text{ and } A=\sum_k \ind_{[\sigma_k,\sigma_{k+1})}A_k. \]
Then, $Au$ is the step-function with $U^p$-norm
\[\|Au\|_{U^p}\les \Big(\sum_{j,k: [\tau_j,\tau_{j+1})\cap [\sigma_k,\sigma_{k+1})\not= \emptyset}\|A_kf_j\|_{L^2}^p\Big)^{\frac1p}\les 1.\]
This implies
$$\|Au\|_{U^p(\R,L^2)}\les   \|A\|_{U^p(\R,\mathcal{L}(L^2))} \|u\|_{U^p(\R,L^2)},$$
If we additionally assume that $A$ commutes with all $P_\lambda$, then we obtain
\begin{align*}\|Au\|_{\ell^2 U^p(\R,L^2)}
=\Big(\sum_{\lambda\in 2^\Z} \|A P_\lambda u\|_{U^p(\R,L^2)}^2 \Big)^{\frac12}
\les{}&  \|A\|_{U^p(\R,\mathcal{L}(L^2))} \Big(\sum_{\lambda\in 2^\Z} \| P_\lambda u\|_{U^p(\R,L^2)}^2 \Big)^{\frac12}\\
={}& \|A\|_{U^p(\R,\mathcal{L}(L^2))} \|u\|_{\ell^2 U^p(\R,L^2)}.
\end{align*}

Now we prove the $V^p$-bounds.
For any partition $(t_k)$ we have, by Minkowski's inequality,
\begin{align*}
\Big(\sum_{k} \|Bv(t_k)-Bv(t_{k-1})\|_{L^2}^p \Big)^{\frac1p}\les{}& \Big(\sum_{k} \big\|\big(B (t_k)-B(t_{k-1})\big)v(t_k)\big\|_{L^2}^p\Big)^{\frac1p}\\&+ \Big(\sum_{k} \big\|B(t_{k-1})\big(v(t_k)- v(t_{k-1})\big)\big\|_{L^2}^p \Big)^{\frac1p}
\end{align*}
which implies
\[
 | B v|_{V^p(\R, L^2)} \les |B|_{V^p(\R,\mathcal{L}(L^2))}\|v\|_{L^\infty(\R,L^2)}+\|B\|_{L^\infty(\R,\mathcal{L}(L^2))} | v |_{V^p(\RR, L^2)},
\]
and together with the obvious bound $$\| B v\|_{L^\infty (\R, L^2)} \les \|B\|_{L^\infty(\R,\mathcal{L}(L^2))}\|v\|_{L^\infty(\R,L^2)}$$ we obtain
\[
 \| B v\|_{V^p(\R, L^2)} \les \|B\|_{V^p(\R,\mathcal{L}(L^2))} \| v \|_{V^p(\RR, L^2)}.
\]
If we additionally assume that $B$ commutes with all $P_\lambda$, then we obtain
\begin{align*}
 \| B v\|_{\ell^2 V^p(\R, L^2)} \les \|B\|_{V^p(\R,\mathcal{L}(L^2))} \| v \|_{\ell^2 V^p(\RR, L^2)}
\end{align*}
by the same argument as for $\ell^2U^p$ presented above.
\end{proof}

\begin{remark}\label{rem:Up operator bounds}
Let $1\les p < \infty$ and $I = [0, T) \subset \RR$. In the special case where $A: \RR \to \mc{L}(L^2)$ is a Fourier multiplier with $C^1$-symbol $a(t,\xi)$, a short computation gives the upper bound
        \begin{equation}\label{eqn:rem Up operator bounds:Vp}
            \| \ind_I A \|_{V^p(\R,\mathcal{L}(L^2))} \lesa \sup_{t\in I} \| a(t) \|_{L^\infty_\xi}+ \Big( \int_I \| \p_t a(t)  \|_{L^\infty_\xi} dt \Big)^{\frac{1}{p}} \sup_{t\in I} \| a(t) \|_{L^\infty_\xi}^{1-\frac{1}{p}}.
        \end{equation}
More precisely, since
        $$ \| A(s_1) - A(s_2) \|_{\mc{L}(L^2)} \les \| a(s_1) - a(s_2) \|_{L^\infty_\xi} \les \int_{s_1}^{s_2} \| \p_t a \|_{L^\infty_\xi} dt$$
the bound \eqref{eqn:rem Up operator bounds:Vp} follows from
        \begin{align*}
            \| \ind_I A \|_{V^p(\R,\mathcal{L}(L^2))} &= \|\ind_I A\|_{L^\infty_t \mc{L}(L^2)}+\Big(\sup_{t_1< t_2< \dots < t_N} \sum_{j=1}^{N-1}  \| (\ind_I A)(t_{j+1}) - (\ind_I A)(t_j) \|_{\mc{L}(L^2)}^p\Big)^{\frac1p}  \\
                &\lesa \sup_{t\in I} \|  A(t) \|_{\mc{L}(L^2)} +  \sup_{t\in I} \| A(t) \|_{\mc{L}(L^2)}^{1-\frac1p}\Big( \sup_{0=t_1< t_2< \dots < t_N=T} \sum_{j=1}^{N-1}  \| A(t_{j+1}) - A(t_j) \|_{\mc{L}(L^2)}\Big)^{\frac1p}.
        \end{align*}
Similarly, in the $U^p$ case, for any $1\les q < p$  the continuous embedding $V^q \subset U^p$, see e.g.\ \cite[(4.1)]{Candy2018b} (holds for functions with values in Hilbert spaces) together with \eqref{eqn:rem Up operator bounds:Vp} implies that
        \begin{equation}\label{eqn:rem Up operator bounds:Up}
            \| \ind_I A \|_{U^p(\R,\mathcal{L}(L^2))} \lesa \sup_{t\in I} \| a(t) \|_{L^\infty_\xi}+\Big( \int_I \| \p_t a(t)  \|_{L^\infty_\xi} dt \Big)^{\frac{1}{q}} \sup_{t\in I} \| a(t) \|_{L^\infty_\xi}^{1-\frac{1}{q}}.
        \end{equation}
\end{remark}

To deal with the non-resonant contribution to the non-relativistic limit, we require a high temporal frequency version of the energy inequality.

\begin{lemma}\label{lem:high freq energy ineq}
Let $1\les a \les 2$ and $R>0$. If $G \in L^a_t L^2_x$ has temporal frequencies supported in $\{|\tau|\g R\}$ then for any $0\les t_1 < t_2 < \infty$ we have
		$$ \Big\| \int_0^t  \ind_{[t_1, t_2)}(s) G(s) ds \Big\|_{\ell^2 U^a} \lesa R^{\frac{1}{a} -1} \| G \|_{L^a_t L^2_x}.  $$
\end{lemma}
\begin{proof}
The support assumption on $G$ implies that $\p_t^{-1} G$ is well-defined (and belongs to say $L^a_t L^2_x$) and $G = \p_t (\p_t^{-1} G)$. Hence we have the identity
		$$\int_0^t \ind_{[t_1, t_2)}(s) G(s) ds = \ind_{[t_1, t_2)}(t) \Big( (\p_t^{-1} G)(t) - (\p_t^{-1} G)(t_1)\Big) + \ind_{[t_2, \infty)}(t) (\p_t^{-1} G)(t_2). $$
Therefore applying (ii) in Lemma \ref{lem:Up basic prop}, noting that $\ind_{[a, b)} G(a)$ is a rescaled atom, $\| \ind_{[a,b)} u \|_{U^2} \les \| u \|_{U^2}$, and letting $P^{(t)}_d$ denote a temporal Fourier cutoff to $|\tau|\approx d$, we see that for any $\lambda \in 2^\ZZ$
	\begin{align*}
			\Big\| \int_0^t  \ind_{[t_1, t_2)}(s) G_{\lambda}(s) ds \Big\|_{U^a}
					&\les \| \p_t^{-1} G \|_{U^a} + 2 \| \p_t^{-1} G \|_{L^\infty_t L^2_x} \\
					&\lesa \sum_{d \gtrsim R} d^{\frac{1}{a}} \| \p_t^{-1} P^{(t)}_d G_\lambda \|_{L^a_t L^2_x} \lesa R^{\frac{1}{a} -1} \| G_{\lambda } \|_{L^a_t L^2_x}.
	\end{align*}
The claimed bound now follows by summing up over $\lambda \in 2^{\ZZ}$.
\end{proof}

\subsection{Adapted function spaces} \label{subsec:adap func}

To construct our solution spaces, we adapt the spaces $\ell^2 U^a$ and $\ell^2 V^{a'}$ to the Dirac equation by pulling back along the linear Dirac flow $\mc{U}_m(t)$.

\begin{definition}\label{def:sol-spaces}
Given $m\in \RR$, $ s,\sigma \in \RR$,  and $1< a < 2$, we define $S^{s,\sigma}_m =|\nabla|^{-\sigma} \lr{\nabla}_m^{\sigma-s} \mc{U}_m(t) \ell^2 U^a$ and the weaker space $S_{w, m}^{s,\sigma} = |\nabla|^{-\sigma} \lr{\nabla}_m^{\sigma-s}  \mc{U}_m(t) V^{a'}$ with the norms
        $$ \| \psi \|_{S^{s,\sigma}_m} = \| |\nabla|^{\sigma} \lr{\nabla}_m^{s-\sigma} \mc{U}_m(-t) \psi \|_{\ell^2 U^a}, \qquad \| \psi \|_{S^{s,\sigma}_{w, m}} = \| |\nabla|^{\sigma} \lr{\nabla}_m^{s-\sigma} \mc{U}_m(-t) \psi \|_{\ell^2 V^{a'}}. $$
where $\frac{1}{a'} = 1 - \frac{1}{a}$ denotes the conjugate exponent to $a$.
\end{definition}
There is some freedom in the choice of the parameter $1<a<2$, but eventually  $a$ is chosen fairly close to $2$. Our nonlinear solutions are constructed in $S_m^{s,\sigma}$, while the weaker $V^{a'}$ based space $S_{w,m}^{-s,-\sigma}$ is essentially the dual to $S_m^{s,\sigma}$ (after applying $\p_t^{-1}$).

For later use,  we note that the disposability of Fourier multipliers in Lemma \ref{lem:Up basic prop} implies that
    \begin{equation}\label{eqn:hom decomp of Sm}
        \| \psi \|_{S^{s,\sigma}_m}^2 \approx \sum_{\substack{ \lambda \in 2^\ZZ }} |\lambda|^{2\sigma} \lr{\lambda}_m^{2s-2\sigma}\| \mc{U}_m(-t) \psi_\lambda \|_{U^a}^2.
    \end{equation}
An identical comment applies to the weak counterpart $S^{s,\sigma}_{w, m}$. Similarly, a straight forward computation using \eqref{eqn:lin dirac flow} and Lemma \ref{lem:Up basic prop} gives
    \begin{equation}\label{eqn:S decomp into pm}
        \begin{split}
             \| \mc{U}_m(-t) \psi \|_{U^a} &\approx \| e^{-it\lr{\nabla}_m} \Pi_+ \psi \|_{U^a} + \| e^{it\lr{\nabla}_m} \Pi_- \psi \|_{U^a},\\
              \| \mc{U}_m(-t) \varphi \|_{V^{a'}} &\approx \| e^{-it\lr{\nabla}_m} \Pi_+ \psi \|_{V^{a'}} + \| e^{it\lr{\nabla}_m} \Pi_- \psi \|_{V^{a'}}.
        \end{split}
    \end{equation}
 Finally, as we often deal with frequency localised estimates with fixed $m\in \{0,1\}$, for  ease of notation we frequently suppress the explicit dependence on the mass $m$, and simply write
        $$ S = S_m^{0,0} \qquad \text{and} \qquad S_{w} = S^{0,0}_{w,m}. $$

The solution spaces $S$ and $S_w$ are built up in terms of $U^p$ and $V^p$, and hence they inherit many of the same properties contained in Lemma \ref{lem:Up basic prop}. In particular, after noting that $\mc{U}(-t) C_{\g \beta} \psi$ has temporal Fourier support in the region $\{\tau \gtrsim \beta\}$, an application of Lemma \ref{lem:Up basic prop} immediately implies the following.

\begin{lemma}[Basic Properties of $S$ and $S_w$]\label{lem:prop of S}
Let $p\g a>1$ and $q\g a'$. Then we have the continuous embeddings $S\subset S_w \subset L^\infty_t L^2_x$, and for any $\beta, \lambda \in 2^\ZZ$ we have modulation bounds
        \begin{equation}\label{eqn:Xsb control}
		\| C_{\g \beta}  \psi_{\lambda} \|_{L^p_t L^2_x} \lesa \beta^{-\frac{1}{p}} \|  \psi \|_{S}, \qquad \|C_{\g \beta}  \psi_\lambda \|_{L^{q}_t L^2_x} \lesa \beta^{-\frac{1}{q}}\| \psi \|_{S_w} 	
        \end{equation}
and disposability bounds
        \begin{equation}\label{eqn:disposability}
            \| C_{\les \beta} \psi \|_{S} \lesa \| \psi \|_S, \qquad \| C_{\les \beta} \psi \|_{S_w}\lesa \| \psi \|_{S_w}.
        \end{equation}
Moreover, if $\psi \in S_w$, then there exists a scattering state $f_\infty\in L^2$ such that
        $$ \lim_{t\to \infty} \| \mc{U}(-t) \psi - f_\infty \|_{L^2_x} = 0. $$
\end{lemma}
\begin{proof}
The only property that does not immediately follow from Lemma \ref{lem:Up basic prop} is the disposability bounds \eqref{eqn:disposability}. As $C_{\les \beta} = \mc{U}(t) P^{(t)}_{\les \beta} \mc{U}(-t)$ where $P^{(t)}_{\les \beta}$ is a smooth temporal Fourier multiplier to the region $\{|\tau| \les \beta\}$, after unpacking the definitions, it suffices to prove that
        $$ \| P^{(t)}_{\les \beta} u \|_{U^a} \lesa \| u\|_{U^a},  \qquad \| P^{(t)}_{\les \beta} v \|_{V^{a'}} \lesa \| v \|_{V^{a'}}.$$
But this follows from the boundedness of temporal convolution operators on $U^p$ and $V^p$, which in turn can be obtained from the dual characterisation (iv) in Lemma \ref{lem:Up basic prop} (see for instance \cite{Koch2014} for details).
\end{proof}

A slightly more involved property is the energy inequality.

\begin{lemma}[Energy Inequality]\label{lem:energy ineq}
Let $s, \sigma\in \RR$. Let $f\in L^2_x$ and assume that for all $\lambda \in 2^\ZZ$ we have $ F_\lambda \in L^1_{t, loc} L^2_x$, and the bound
	$$ \sup_{\substack{\phi \in C_0^\infty(\RR^{1+d}) \\ \| \phi \|_{S^{-s,-\sigma}_{w, m}} \les 1}} \Big| \int_0^\infty \int_{\RR^d} \overline{\phi} F dt dx \Big| < \infty. $$
If we let
		$$\psi(t) = \ind_{[0, \infty)}(t) \mc{U}_m(t) f +  \ind_{[0, \infty)}(t) \int_0^t \mc{U}_m(t-t') \gamma^0 F(t') dt'$$
then $ \psi - \ind_{[0, \infty)}(t) \mc{U}_m(t) f  \in C(\RR; H^{s,\sigma}_m) \cap S^{s,\sigma}_m $ and moreover
        $$ \| \psi \|_{S^{s,\sigma}_m} \lesa \| f \|_{H^{s,\sigma}_m} +  \sup_{\substack{\phi \in C_0^\infty(\RR^{1+d}) \\ \| \phi \|_{S^{-s,-\sigma}_{w, m}} \les 1}} \Big| \int_0^\infty \int_{\RR^d} \overline{\phi} F dt dx \Big|.$$
\end{lemma}
\begin{proof}
In view of the bound
        $$ \| \ind_{[0, \infty)} \mc{U}(t) f \|_{S} \approx \| \ind_{[0, \infty)} f \|_{\ell^2 U^a} \lesa \| f\|_{L^2}$$
it suffices to consider the case $f=0$. For every $\lambda \in 2^\ZZ$ the assumption $F_\lambda \in L^1_{t,loc}L^2_x$ immediately implies that $\psi_\lambda$, and hence $\mc{U}_m(-t) \psi_\lambda$, is continuous (as a map into $L^2$) on any finite time interval. On the other hand, for any $v\in C^\infty_0(\RR^{1+d})$ we have
    \begin{align*}
        \Big| \int_0^\infty \int_{\RR^d} (\p_t v)^\dagger \mc{U}_m(-t) \psi_\lambda \, dx dt\Big|
                 &= \Big| \int_0^\infty \int_{\RR^d} \overline{ \mc{U}_m(t) v}  F_\lambda \, dx dt\Big| \\
                 &\les \| \mc{U}_m(t) v_\lambda \|_{S^{-s,-\sigma}_{w, m}} \sup_{\substack{\phi \in C_0^\infty(\RR^{1+d}) \\ \| \phi \|_{S^{-s,-\sigma}_{w, m}} \les 1}} \Big| \int_0^\infty \int_{\RR^d} \overline{\phi} F dt dx \Big| \\
                 &\lesa |\lambda|^{-\sigma}\lr{\lambda}^{\sigma-s}_m \| v \|_{V^{a'}}  \sup_{\substack{\phi \in C_0^\infty(\RR^{1+d}) \\ \| \phi \|_{S^{-s,-\sigma}_{w, m}} \les 1}} \Big| \int_0^\infty \int_{\RR^d} \overline{\phi} F dt dx \Big| < \infty.
    \end{align*}
Hence (iv) in  Lemma \ref{lem:Up basic prop} implies that $\psi_\lambda \in U^a$. Moreover, the above computation together with a standard duality argument gives
    \begin{align*}
      \| \psi \|_{S^{s,\sigma}_m}
      &\approx \sup_{ \substack{\|(a_\lambda)_{\lambda \in 2^\ZZ}\|_{\ell^{2}}<1\\ \sup_{\lambda \in 2^\ZZ}  \| v^{(\lambda)} \|_{V^{a'}} \les 1}} \Big| \sum_{\lambda \in 2^\ZZ}\int_{\RR^{1+d}} \Big( |\lambda|^\sigma \lr{\lambda}^{s-\sigma}_m a_\lambda \p_t v^{(\lambda)}\Big)^\dagger \mc{U}_m(-t) \psi_\lambda dt dx \Big| \\
      &\approx \sup_{ \substack{\|(a_\lambda)_{\lambda \in 2^\ZZ}\|_{\ell^{2}}<1\\ \sup_{\lambda \in 2^\ZZ}  \| v^{(\lambda)} \|_{V^{a'}} \les 1}} \Big| \sum_{\lambda \in 2^\ZZ} \int_0^\infty \int_{\RR^d} \overline{ \sum_{\lambda \in 2^\ZZ} |\lambda|^\sigma \lr{\lambda}_m^{s-\sigma} a_\lambda \mc{U}_m(t) P_\lambda v^{(\lambda)}} F dx dt \Big|
    \end{align*}
and thus the required bound follows after observing that if $\| (a_\lambda) \|_{\ell^2}< 1$ and $\sup_{\lambda \in 2^\ZZ} \| v^{(\lambda)}\|_{V^{a'}}\les 1$ then
    $$ \Big\| \sum_{\lambda \in 2^\ZZ}|\lambda|^\sigma \lr{\lambda}_m^{s-\sigma} a_\lambda \mc{U}_m(t) P_\lambda v^{(\lambda)} \Big\|_{S^{-s,-\sigma}_{w, m}} \lesa \Big( \sum_{\lambda \in 2^\ZZ} |a_\lambda|^2 \| v^{(\lambda)}\|_{V^{a'}}^2 \Big)^\frac{1}{2} \les 1. $$
Finally the claimed continuity in $t$ follows from the bound
            $$ \Big( \sum_{\lambda \in 2^\ZZ} \| \psi_\lambda \|_{L^\infty_t H^{s,\sigma}_m}^2 \Big)^\frac{1}{2} \lesa \| \psi \|_{S^{s,\sigma}_m} $$
           and  the continuity of $\psi_\lambda$.
\end{proof}

The norms $\| \cdot \|_{S^{s,\sigma}_m}$ have the same scaling as the Sobolev norms $\|\cdot \|_{H^{s,\sigma}_m}$. More precisely, a simple computation gives

\begin{lemma}[Mass rescaling]\label{lem:mass rescaling}
Let $m, s,\sigma \in \RR$ and $\alpha > 0$. Let $\psi \in S^{s,\sigma}_m$ and define
            $$ \varphi(t,x) = \alpha^\frac{1}{2} \psi( \alpha t, \alpha x). $$
Then $\varphi \in S_{\alpha m}^{s,\sigma} $ and $\|\varphi\|_{S_{\alpha m }^{s,\sigma}} \approx \alpha^{s-\frac{d-1}{2}} \| \psi \|_{S_m^{s,\sigma}}$. Moreover, the same holds with $S_m^{s,\sigma}$ replaced with $S_{w,m}^{s,\sigma}$.
\end{lemma}

\section{Strichartz estimates and consequences}\label{sec:str}
Functions in our solution spaces $S$ and $S_w$ also satisfy useful Strichartz estimates.

\begin{lemma}[Wave admissible Strichartz estimates]\label{lem:wave stri}
Let $m \in \RR$ and  $(q,r)$ be wave-admissible, i.e.\
$$
2\les q \les \infty, \quad 2\les r<\infty, \quad \frac{1}{q} + \frac{d-1}{2r} = \frac{d-1}{4}.
$$
and $s=\frac{d+1}{d-1}\frac{1}{q}$ and $\sigma= \frac{2}{d-1}\frac{1}{q}$. Then  we have
    $$
		\| \psi \|_{L^q_t L^r_x} + \sup_{\alpha \in 2^{-\NN}}\Big\| \Big( \sum_{\kappa \in \mc{C}_\alpha} \| \psi_{\kappa}\|_{L^r_x}^2 \Big)^{\frac{1}{2}}\Big\|_{L^q_t}
    \lesa \| \psi \|_{S_m^{s , \sigma}}.
	$$
	If we additionally assume that $q>a'$, then
    $$
		  \| \psi \|_{L^{q}_t L^{r}_x} + \sup_{\alpha \in 2^{-\NN}}\Big\| \Big( \sum_{\kappa \in \mc{C}_\alpha} \| \psi_{ \kappa}\|_{L^{r}_x}^2 \Big)^{\frac{1}{2}}\Big\|_{L^{q}_t}  \lesa  \| \psi \|_{S_{w,m}^{s ,\sigma}}.
	$$
\end{lemma}
\begin{proof} If $|m|=1$, the classical Strichartz estimate for the Klein-Gordon equation for wave admissible pairs $(q,r)$ (see e.g. \cite[(2.3) for $\theta=0$]{Machihara2003})  implies
\begin{equation}\label{eq:free-wave}
\|e^{\pm it \lr{\nabla}}f_\lambda \|_{L^{q}_t L^{r}_x}\lesa \lr{\lambda}^{s} \|f_\lambda\|_{L^2}, \text{ for all }\lambda\in 2^{\Z}.
\end{equation}
For a Schr\"{o}dinger admissible pair $(q,p)$ with $\frac{1}{p}=\frac12-\frac{2}{dq}$ (see e.g. \cite[(2.3) for $\theta=1$]{Machihara2003})   we have
\begin{equation}\label{eq:schr-free}
\|e^{\pm it \lr{\nabla}}f_\lambda \|_{L^{q}_t L^{p}_x}\lesa  \lr{\lambda}^{\frac{d+2}{d}\frac{1}{q}} \|f_\lambda\|_{L^2}, \text{ for all }\lambda \in 2^{\Z}.
\end{equation}
For small frequencies, i.e.\ $\lambda \in 2^{-\N}$, we can use \eqref{eq:schr-free} and the Bernstein inequality $\|P_\lambda g\|_{L^r_x}\lesa \lambda^{\frac{2}{d-1}\frac{1}{q}} \|P_\lambda g\|_{L^p_x}$ to improve \eqref{eq:free-wave} to
$$\|e^{\pm it \lr{\nabla}}f_\lambda \|_{L^{q}_t L^{r}_x}\lesa \lambda^{\sigma} \lr{\lambda}^{s-\sigma} \|f_\lambda\|_{L^2}, \text{ for all }\lambda\in 2^{\Z}.$$
For general $m \in \R$, by rescaling from $m\ne 0$ to $|m|=1$ or in case $m=0$ by the classical estimate for the wave equation (see e.g.\ \cite[Prop.\ 3.1]{Ginibre1995}), we obtain
$$\|e^{\pm it \lr{\nabla}_m} f_\lambda \|_{L^{q}_t L^{r}_x}\lesa \lambda^{\sigma} \lr{\lambda}_m^{s-\sigma} \|f_\lambda\|_{L^2}, \text{ for all }\lambda\in 2^{\Z}.$$
The Minkowski inequality implies
 $$\sup_{\alpha \in 2^{-\NN}}\Big\| \Big( \sum_{\kappa \in \mc{C}_\alpha} \| e^{\pm it \lr{\nabla}_m} f_{ \kappa,\lambda }\|_{L^{r}_x}^2 \Big)^{\frac{1}{2}}\Big\|_{L^{q}_t} \lesa \lambda^{\sigma} \lr{\lambda}_m^{s-\sigma} \|f_\lambda\|_{L^2}, \text{ for all }\lambda\in 2^{\Z}.$$
 Next, we extend these two estimates to $S_{m}^{s ,\sigma}$ and $S_{w,m}^{s ,\sigma}$.
 To this end, consider a $e^{\pm it \lr{\nabla}_m} U^q$ atom $u_\lambda (t)=\sum_{j=1}^N \ind_{[t_j,t_{j+1})}(t)e^{\pm it \lr{\nabla}_m}P_\lambda f_j$.
 Then,
 \begin{align*}
 \|u_\lambda \|_{L^{q}_t L^{r}_x} =& \Big( \sum_{j=1}^N \Big\|\ind_{[t_j,t_{j+1})}(t)e^{\pm it \lr{\nabla}_m} P_\lambda f_{j} \Big\|_{L^{q}_t L^{r}_x }^q\Big)^{\frac1q}
\lesa{} \Big( \sum_{j=1}^N  \lambda^{q\sigma} \lr{\lambda}_m^{q(s-\sigma)}  \|P_\lambda f_j\|_{L^2}^q\Big)^{\frac1q}\lesa  \lambda^{\sigma} \lr{\lambda}_m^{s-\sigma}.
 \end{align*}
 Similarly, since the square sum is \emph{inside} the $L^q_t$ norm, we have
  \begin{align*}
  \Big\| \Big( \sum_{\kappa \in \mc{C}_\alpha} \|u_{\lambda,\kappa} \|_{L^r_x}^2 \Big)^{\frac{1}{2}}\Big\|_{L^q_t}
    \lesa{} & \Big(\sum_{j=1}^N \Big\| \Big( \sum_{\kappa \in \mc{C}_\alpha} \|e^{\pm it \lr{\nabla}_m} P_{\lambda,\kappa} f_{j} \|_{L^r_x}^2 \Big)^{\frac{1}{2}}\Big\|_{L^q_t}^q\Big)^{\frac1q}
\lesa{}   \lambda^{\sigma} \lr{\lambda}_m^{s-\sigma}.
    \end{align*}
 Therefore, both estimates extend to general functions in the atomic space $e^{\pm it \lr{\nabla}_m} U^q$.
 By the dyadic structure of the spaces $S_{m}^{s ,\sigma}$ and $S_{w,m}^{s ,\sigma}$, the Littlewood-Paley square function estimate, and the continuous embeddings $U^a \subset U^q$ and in case $q>a'$ also $V^{a'} \subset U^{q}$, we obtain
the claimed estimates.
\end{proof}

In view of scaling, we are only able to use wave-admissible pairs, with the exception of a specific non-transverse interaction (which breaks the scaling) which only occurs if $m\ne0$, where we  also require Schr\"{o}dinger admissible pairs.

\begin{lemma}[Schr\"{o}dinger admissible Strichartz estimates]\label{lem:schr stri}
Let $m \ne 0$ and  $(q,p)$ be Schr\"{o}dinger-admissible, i.e.\
$$
2\les q \les \infty, \quad 2\les p<\infty, \quad \frac{1}{q} + \frac{d}{2p} = \frac{d}{4}.
$$
and $s'=\frac{d+2}{d}\frac{1}{q}$. Then  we have
    $$
		\| \psi \|_{L^q_t L^p_x} + \sup_{\alpha \in 2^{-\NN}}\Big\| \Big( \sum_{\kappa \in \mc{C}_\alpha} \| \psi_{\kappa}\|_{L^p_x}^2 \Big)^{\frac{1}{2}}\Big\|_{L^q_t}
    \lesa |m|^{-\frac{2}{dq}}\| \psi \|_{S_m^{s' , 0}}.
	$$
	If we additionally assume that $q>a'$, then
    $$
		  \| \psi \|_{L^{q}_t L^{p}_x} + \sup_{\alpha \in 2^{-\NN}}\Big\| \Big( \sum_{\kappa \in \mc{C}_\alpha} \| \psi_{ \kappa}\|_{L^{p}_x}^2 \Big)^{\frac{1}{2}}\Big\|_{L^{q}_t}  \lesa |m|^{-\frac{2}{dq}} \| \psi \|_{S_{w,m}^{s' ,0}}.
	$$
\end{lemma}
\begin{proof}
Rescaling to $|m|=1$, the estimate \eqref{eq:schr-free}, and the Minkowski inequality imply
\begin{equation*}
\|e^{\pm it \lr{\nabla}_m }f_\lambda \|_{L^{q}_t L^{p}_x}
 + \Big\| \Big( \sum_{\kappa \in \mc{C}_\alpha} \|e^{\pm it \lr{\nabla}_m }f_{\lambda,\kappa} \|_{L^p_x}^2 \Big)^{\frac{1}{2}}\Big\|_{L^q_t}
\lesa |m|^{-\frac{2}{dq}} \lr{\lambda}_m^{s'}\|f_\lambda\|_{L^2},
\end{equation*}
for all $\lambda\in 2^\Z$, $\alpha\in 2^{-\N}$. Now, the same transference argument as in the proof of Lemma \ref{lem:wave stri} shows that the estimate extends to $S_m^{s' , 0}$ resp.\ $S_{m,w}^{s' , 0}$.
\end{proof}

As a first application of the results of Section \ref{sec:str} we prove a simple bilinear bound which exploits localisation to caps to obtain a null form gain in the case where one function has high modulation.

\begin{theorem}[High modulation case]\label{thm:bi high mod}
Let $d\g 2$, $\frac{8}{5}<a<2$ and $a\les p \les a'$. Then for any $R, \mu_1, \mu_2, \beta \in 2^\ZZ$ and $\alpha \in 2^{-\NN}$ we have
		\begin{equation}\label{eq:bi high mod 1}
\bigg\|  \sum_{\substack{ \kappa_1, \kappa_2 \in \mc{C}_\alpha  \\ \ma(\kappa_1, \kappa_2) + \frac{|m|}{\lr{\mu}_m} \approx \alpha}} P_{\les R}(\overline{C_{\gtrsim \beta} \varphi}_{\mu_1, \kappa_1} \psi_{\mu_2, \kappa_2}) \bigg\|_{L^p_t L^2_{x}}
		\lesa \alpha \beta^{- \frac{1}{a'}}  \mb{m}_{R, \mu}^{\frac{1}{2} - \frac{2}{d-1}(\frac{1}{p} - \frac{1}{a'})}  \big( \mu_2^{\frac{2}{d-1}}\lr{\mu_2}_m \big)^{\frac{1}{p} -\frac{1}{a'} }\| \varphi_{\mu_1} \|_{S_w} \| \psi_{\mu_2} \|_{S_w}
	    \end{equation}
and
	\begin{equation}\label{eq:bi high mod 2}
\bigg\|  \sum_{\substack{ \kappa_1, \kappa_2 \in \mc{C}_\alpha  \\ \ma(\kappa_1, \kappa_2) + \frac{|m|}{\lr{\mu}_m} \approx \alpha}} P_{\les R}(\overline{C_{\gtrsim \beta} \varphi}_{\mu_1, \kappa_1} \psi_{\mu_2, \kappa_2}) \bigg\|_{L^p_t L^2_{x}}\\
		\lesa \alpha \mb{m}_{R, \mu}^{\frac{1}{2}}  \beta^{-\frac{1}{p}} \| \varphi_{\mu_1} \|_{S} \| \psi_{\mu_2} \|_{S_w}
\end{equation}	
where we take $\mb{m}_{R, \mu} :=\min\{R, \mu\} \big( \min\{R, \alpha \mu\}\big)^{d-1}$ and $\mu = \min\{\mu_1, \mu_2\}$.
\end{theorem}
\begin{proof}
We start with a simple bound for spatial functions (spinors) and $2\les r <\infty$: For any $\kappa_1, \kappa_2 \in \mc{C}_\alpha$ with $\angle(\kappa_1, \kappa_2)  + \frac{|m|}{\lr{\mu}_m} \approx \alpha$, and $f \in L^2$,  $g \in L^r$, such that
    $$\supp \widehat{\Pi_\pm f} \subset \{ |\xi| \approx \mu_1, \pm \tfrac{\xi}{|\xi|} \in \kappa_1\}, \qquad \supp \widehat{\Pi_\pm g} \subset \{ |\xi| \approx \mu_2, \pm\tfrac{\xi}{|\xi|}\in \kappa_2\},$$
we have the bound
    \begin{equation}\label{eqn:spatial v2} \| P_{\les R} ( \overline{f} g)  \|_{L^2_x} \lesa \alpha (\textbf{m}_{R, \mu})^\frac{1}{r} \| f \|_{L^2_x} \| g \|_{L^r_x}.\end{equation}

Since the product has Fourier support in a ball of radius $\lesa \mu_1+\mu_2$, we may as well suppose that $R\lesa \mu_1 + \mu_2$. Let $\xi_{\kappa_j} \in \RR^d$ be the center of the angular sector $\{ |\xi| \approx \mu_j, \frac{\xi}{|\xi|}\in \kappa_j\}$ and introduce the short hand $\Pi_{\pm_j}^{\kappa_j} = \Pi_{\pm_j}(\xi_{\kappa_j})$. Note that $\Pi_{\pm_j}^{\kappa_j}$ is then simply a complex valued matrix. Since $\alpha \approx \angle(\kappa_1, \kappa_2) + \frac{|m|}{\lr{\mu}_m}$, an application of \eqref{eqn:Pi null struc} gives the null form bounds
$$ \| (\Pi_{\pm_1} - \Pi_{\pm_1}^{\kappa_1})f \|_{L^r_x} \lesa \alpha \| f\|_{L^r_x}, \qquad \| \gamma^0(\Pi_{\pm_2} - \Pi_{\pm_2}^{\kappa_2})g \|_{L^r_x}\lesa \alpha \| g\|_{L^r_x},$$
and
$$ \| \Pi_{\pm_1}^{\kappa_1} \gamma^0 \Pi_{\pm_2}^{\kappa_2} g\|_{L^r_x} \lesa \alpha \|  g \|_{L^r_x},$$
see e.g. \cite{Candy2018a} or \cite[Lemma 3.3]{Bejenaru2017}. Hence, after exploiting the decomposition
\begin{equation}\label{eqn:thm bi trans:null decomp}
	\begin{split}
		\overline{\Pi_{\pm_1} f} \Pi_{\pm_2} g&= \big[(\Pi_{\pm_1} - \Pi_{\pm_1}^{\kappa_1})f\big]^\dagger \big[\gamma^0 \Pi_{\pm_2} g\big]\\
		&\qquad +\big[\Pi_{\pm_1}^{\kappa_1}f\big]^\dagger \big[\gamma^0(\Pi_{\pm_2} - \Pi_{\pm_2}^{\kappa_2})g\big]
		+\big[f\big]^\dagger \big[\Pi_{\pm_1}^{\kappa_1} \gamma^0 \Pi_{\pm_2}^{\kappa_2} g\big]
	\end{split}
\end{equation}
\eqref{eqn:spatial v2} reduces to proving that
    $$\| P_{\les R} (f^\dagger g) \|_{L^2_x} \lesa (\mb{m}_{R, \mu})^{\frac{1}{r}} \|f\|_{L^2_x} \| g \|_{L^r_x}.$$
If $R\lesa \mu_1\sim \mu_2$, then $ f^\dagger g$ has Fourier support in a parallelepiped of dimensions $R \times \min\{R,\alpha\mu\}\times \ldots \times \min\{R ,\alpha\mu\}$, therefore the Fourier support has measure $R \min\{R,\alpha\mu\}^{d-1}$. If  $\frac1\varrho=\frac12+\frac1r$, Bernstein's inequality and H\"{o}lder's inequality imply
	  	\begin{align*}
	  		\big\| P_{\les R}( f^\dagger g) \big\|_{L^2_x}
	  		\lesa{} \big[ R \min\{R,\alpha\mu\}^{d-1}\big]^{\frac{1}{r}} \big\| f^\dagger g\big\|_{L^\varrho_x}
	  		\lesa{} (\textbf{m}_{R, \mu})^{\frac{1}{r}} \|f\|_{L^2_x} \| g \|_{L^r_x}.
	  	\end{align*}
Similarly, if $ \mu_1\ll \mu_2\sim R$ and $\frac12=\frac1\rho+\frac1r$, we obtain
	  	\begin{align*}
	  		\big\| P_{\les R} (  f^\dagger  g) \big\|_{L^2_x}
	  		\lesa{}\|f\|_{L^\rho_x} \| g \|_{L^r_x}
	  		\lesa{} \big[ \mu (\alpha \mu)^{d-1} \big]^\frac{1}{r}  \|f\|_{L^2_x} \| g \|_{L^r_x}.
	  	\end{align*}
And again, if $ \mu_2\ll \mu_1\sim R$, we obtain
        $$\big\| P_{\les R} (  f^\dagger  g) \big\|_{L^2_x}
	  		\lesa{}\|f\|_{L^2_x} \| g \|_{L^\infty_x}
	  		\lesa{} \big[ \mu (\alpha \mu)^{d-1} \big]^\frac{1}{r}  \|f\|_{L^2_x} \| g \|_{L^r_x}.$$
This concludes the proof of \eqref{eqn:spatial v2}.
  	
To prove \eqref{eq:bi high mod 1}, let  $\frac{1}{q}=\frac{1}{p}-\frac{1}{a'}$ and $\frac{1}{q}+\frac{d-1}{2r}=\frac{d-1}{4}$.
The pair $(q,r)$ is wave Strichartz admissible (i.e. satisfies the conditions of Lemma \ref{lem:wave stri}) and the assumption $\frac{8}{5} < a< 2$ ensures that $q> a'$. An application of H\"older's inequality, \eqref{eqn:spatial v2}, and a trivial cap summation give
       \begin{align*}
  	     &\bigg\|  \sum_{\substack{ \kappa_1, \kappa_2 \in \mc{C}_\alpha  \\ \ma(\kappa_1, \kappa_2) + \frac{m}{\lr{\mu}_m} \approx \alpha}} P_{\les R}(\overline{C_{\gtrsim \beta} \varphi}_{\mu_1, \kappa_1} \psi_{\mu_2, \kappa_2}) \bigg\|_{L^p_t L^2_{x}}\\
  	     \lesa{}&\alpha  (\mb{m}_{R, \mu})^{\frac{1}{r}}  \big\|C_{\gtrsim \beta}  \varphi_{\mu_1}\big\|_{L^{a'}_tL^2_x}
  \bigg\|  \Big(\sum_{\kappa_2 \in \mc{C}_\alpha}  \|\psi_{\mu_2, \kappa_2}\|_{L^r_x}^2\Big)^{\frac12}\bigg\|_{L^q_t}\\
  \lesa{}&\alpha  (\mb{m}_{R, \mu})^{\frac{1}{r}} \beta^{-\frac{1}{a'}} \mu_2^{\frac{2}{d-1}\frac{1}{q}}  \lr{\mu_2}_m^{\frac{1}{q}}\| \varphi_{\mu_1} \|_{S_w} \| \psi_{\mu_2} \|_{S_w}
  \end{align*}
where we have used Lemma \ref{lem:prop of S} and Lemma \ref{lem:wave stri} in the last step.

Similarly, applying \eqref{eqn:spatial v2} with $r=2$ gives
	\begin{align*}
		&\bigg\|  \sum_{\substack{ \kappa_1, \kappa_2 \in \mc{C}_\alpha  \\ \ma(\kappa_1, \kappa_2) + \frac{m}{\lr{\mu_2}_m} \approx \alpha}} P_{\les R}(\overline{C_{\gtrsim \beta} \varphi}_{\mu_1, \kappa_1} \psi_{\mu_2, \kappa_2}) \bigg\|_{L^p_t L^2_{x}}\\
		&\lesa \alpha  (\mb{m}_{R, \mu})^{\frac{1}{2}} \big\|C_{\gtrsim \beta} \varphi_{\mu_1}\big\|_{L^{p}_tL^2_x}\bigg\|  \Big(\sum_{\kappa_2 \in \mc{C}_\alpha}  \|\psi_{\mu_2, \kappa_2}\|_{L^2_x}^2\Big)^{\frac12}\bigg\|_{L^\infty_t}
	\end{align*}
and hence \eqref{eq:bi high mod 2} again follows from Lemma \ref{lem:prop of S} and Lemma \ref{lem:wave stri}.
\end{proof}

As a second application of the Strichartz estimates, we prove a bilinear estimate in a non-transverse interaction arising only if $m\ne 0$.
In this setting, when $\alpha$ is very small, the angle might be zero. In this case, we have access to the Schr\"{o}dinger admissible Strichartz estimates, see Lemma \ref{lem:schr stri}. Together with the null structure, this gives a large gain due to very small $\alpha$.

\begin{theorem}[Non-transverse case]\label{thm:bi nontrans}
Let $d\g2$ and $\frac{1}{2}\les \frac{1}{a} < \frac{3}{4}$. For all $0<R \lesa \mu_1 \approx \mu_2$, and $\alpha \approx \frac{|m|}{\lr{\mu_2}_m}>0$, we have
         $$ \Big\| \sum_{\substack{\kappa_1, \kappa_2 \in \mc{C}_\alpha \\ \ma(\kappa_1, \kappa_2) \lesa \alpha}} P_{\les R}(\overline{\varphi_{\mu_1, \kappa_1}} \psi_{\mu_2, \kappa_2}) \Big\|_{L^2_{t,x}} \lesa  \alpha^{1-\frac{1}{d}} R^{\frac{1}{2} - \frac{1}{d}} \big( \min\{R, \alpha \mu\}\big)^{\frac{d-1}{2} - \frac{d-1}{d}} \lr{\mu}_m^\frac{1}{2} \| \varphi_{\mu_1} \|_{S_w} \| \psi_{\mu_2} \|_{S_w}.$$
\end{theorem}
\begin{proof}
Similarly to the purely spatial bound \eqref{eqn:spatial v2},
using the null form bound \eqref{eqn:Pi null struc}, Bernstein's and H\"{o}lder's inequalities, we have
$$
 \Big\| \sum_{\substack{\kappa_1, \kappa_2 \in \mc{C}_\alpha \\ \ma(\kappa_1, \kappa_2) \lesa \alpha}} P_{\les R}(\overline{\varphi_{\mu_1, \kappa_1}} \psi_{\mu_2, \kappa_2}) \Big\|_{L^2_{t,x}}
         \lesa    \alpha (\mb{m}_{R, \mu_2})^{\frac{2}{p}-\frac12}  \Big\| \sum_{\substack{\kappa_1, \kappa_2 \in \mc{C}_\alpha \\ \ma(\kappa_1, \kappa_2) \lesa \alpha}}
         \|\varphi_{\mu_1, \kappa_1}\|_{L^p_x} \| \psi_{\mu_2, \kappa_2}\|_{L^p_x} \Big\|_{L^2_{t}}
$$
for any $2\les p\les 4$, where $\mb{m}_{R, \mu_2} =\min\{R, \mu_2\} \big( \min\{R, \alpha \mu_2\}\big)^{d-1}$. Choosing $\frac{1}{p} = \frac{1}{2} - \frac{1}{2d}$,
another application of H\"{o}lder's inequality implies
\begin{align*}
&\Big\| \sum_{\substack{\kappa_1, \kappa_2 \in \mc{C}_\alpha \\ \ma(\kappa_1, \kappa_2) \lesa \alpha}} P_{\les R}(\overline{\varphi_{\mu_1, \kappa_1}} \psi_{\mu_2, \kappa_2}) \Big\|_{L^2_{t,x}}
\\
\lesa{}&  \alpha (\mb{m}_{R, \mu_2})^{\frac{1}{2}-\frac1d} \Big\| \Big( \sum_{\kappa_1 \in \mc{C}_\alpha} \|\psi_{\mu_1, \kappa_1}\|_{L^p_x}^2\Big)^\frac{1}{2} \Big\|_{L^4_t}
          \Big\| \Big( \sum_{\kappa_2 \in \mc{C}_\alpha} \|\psi_{\mu_2, \kappa_2}\|_{L^p_x}^2\Big)^\frac{1}{2} \Big\|_{L^4_t}\\
          &\lesa \alpha  (\mb{m}_{R, \mu_2})^{\frac{1}{2}-\frac1d}   |m|^{-\frac1d} \lr{\mu_2}_m^{\frac{1}{2}+ \frac{1}{d}}\| \varphi_{\mu_1} \|_{S_w} \| \psi_{\mu_2} \|_{S_w}
           \end{align*}
      where we used  Lemma \ref{lem:schr stri} (note that $m>0$ and $a'<4$) in the last step.
The claim follows because $|m| \approx \alpha \lr{\mu_2}_m$.
\end{proof}

\section{Bilinear Fourier Restriction estimates and consequences}\label{sec:be2}
Our goal is to now consider the more difficult case where both $\varphi$ and $\psi$ are supported close to the hyperbola $\tau^2 = |\xi|^2 + m^2$. In particular, we can no longer exploit the large modulation estimates as in the previous section, and instead have to exploit transversality. Namely, the improved decay of the product $\overline{\varphi}\psi$ when the waves propagate in different directions/speeds. To illustrate this gain, we begin by considering a bilinear $L^2_{t,x}$ estimate for free Klein-Gordon waves.

\begin{lemma}\label{lem:bi L2 free KG}
Let $\alpha, \mu, \lambda \in 2^\ZZ$ and $\kappa_1, \kappa_2 \in \mc{C}_\alpha$ with $0<\alpha \les 1$, $0<\mu\les \lambda$, and
    \begin{equation}\label{eqn:lem bi L2 free KG:cond}
        \frac{\lambda}{\mu} + \frac{\alpha \lr{\mu}_m}{|m|} \gg 1, \qquad \angle(\kappa_1, \kappa_2) + \frac{|m|}{\lr{\mu}_m} \approx \alpha.
    \end{equation}
If $\supp \widehat{f} \subset \{ |\xi| \approx \lambda, \frac{\xi}{|\xi|} \in \kappa_1\}$ and $\supp \widehat{g} \subset \{|\xi| \approx \mu, \frac{\pm \xi}{|\xi|} \in \kappa_2\}$ then for any $R>0$ we have
        $$ \big\| e^{it\lr{\nabla}_m} f e^{\pm i t\lr{\nabla}_m} g \big\|_{L^2_{t,x}} \lesa \alpha^{-\frac{1}{2}} \big(\min\{R, \mu\}\big)^{\frac{1}{2}} \big( \min\{R, \alpha\mu\}\big)^{\frac{d-2}{2}}\Big( \frac{\lr{\lambda}_m}{\lambda}\Big)^{\frac{1}{2}} \| f \|_{L^2} \| g \|_{L^2}. $$
\end{lemma}
\begin{proof}
The left hand side only depends on $|m|$, so it is enough to consider $m\geq 0$.
The key observation is that, after restricting to the support of $\widehat{f}$ and $\widehat{g}$, the normal directions to the hyperboloids $\tau = \lr{\xi}_m$ and $\tau = \pm \lr{\xi}_m$ are transverse. More precisely, a short computation shows that for any $\xi\in \supp \widehat{f}$ and $\eta \in \supp \widehat{g}$ we have
         \begin{equation}\label{eqn:lem bi L2 free KG:trans}
            \Big| \frac{\xi}{\lr{\xi}_m} - \frac{\pm \eta}{\lr{\eta}_m}\Big| \approx \frac{m^2 ||\xi| - |\eta||}{\lr{\mu}_m^2 \lr{\lambda}_m} + \Big(\frac{\lambda \mu}{\lr{\lambda}_m \lr{\mu}_m}\Big)^\frac{1}{2} \ma(\kappa_1, \kappa_2).
         \end{equation}
Thus the required bound now follows from \eqref{eqn:lem bi L2 free KG:cond} together with a standard argument using Plancherel and H\"older's inequality, see for instance \cite[Theorem 5.2]{Candy2019a} or (the proof of ) \cite[Lemma 2.6]{Candy2018a}.

In slightly more detail, we consider the two cases $\alpha \gg \frac{m}{\lr{\mu}_m}$ and $\alpha \approx \frac{m}{\lr{\mu}_m}$. In the former case, let $\mc{Q}_R$ denote a finitely overlapping collection of cubes $q\subset \RR^d$ of diameter $R>0$, and let $f_q$ denote a smooth Fourier cutoff of the Fourier support of $f$ to the cube $q\in \mc{Q}_R$. Since $\alpha \gg \frac{m}{\lr{\mu}_m}$, we have $\mu \gtrsim m$ and, in view of the transversality condition \eqref{eqn:lem bi L2 free KG:cond}, $\ma(\kappa_1, \kappa_2) \approx \alpha$. In particular following the argument in
\cite[Lemma 2.6]{Candy2018a} (using a change of variables in the direction orthogonal to the center of the cap $\kappa_1$) we obtain for any $q, q'\in \mc{Q}_R$
        $$ \big\| e^{it\lr{\nabla}_m} f_q e^{\pm i t\lr{\nabla}_m} g_{q'} \big\|_{L^2_{t,x}} \lesa \alpha^{-\frac{1}{2}} \big(\min\{R, \mu\}\big)^{\frac{1}{2}} \big( \min\{R, \alpha\mu\}\big)^{\frac{d-2}{2}} \|f_q \|_{L^2} \| g_{q'}\|_{L^2}. $$
Hence decomposing into cubes $q, q' \in \mc{Q}_R$ we see that
        \begin{align*}
       \big\| P_{\les R} \big( e^{it\lr{\nabla}_m} f e^{\pm i t\lr{\nabla}_m} g\big) \big\|_{L^2_{t,x}}
                    &\lesa \sum_{\substack{q, q' \in \mc{Q}_R \\ \dist(q, q') \lesa R}} \| e^{ i t\lr{\nabla}_m} f_{q} e^{\pm it \lr{\nabla}_m} g_{q'}  \|_{L^2_{t,x}} \\
                    &\lesa \alpha^{-\frac{1}{2}} \big(\min\{R, \mu\}\big)^{\frac{1}{2}} \big( \min\{R, \alpha\mu\}\big)^{\frac{d-2}{2}}  \sum_{\substack{q, q' \in \mc{Q}_R \\ \dist(q, q') \lesa R}}  \|f_q \|_{L^2} \| g_{q'}\|_{L^2} \\
                    &\lesa \alpha^{-\frac{1}{2}} \big(\min\{R, \mu\}\big)^{\frac{1}{2}} \big( \min\{R, \alpha\mu\}\big)^{\frac{d-2}{2}} \|f \|_{L^2} \| g\|_{L^2}
        \end{align*}
which suffices since $\lambda \gtrsim \mu \gtrsim m$ in this case. On the other hand, if $\alpha \approx \frac{m}{\lr{\mu}_m}$ then \eqref{eqn:lem bi L2 free KG:cond} implies that $\lambda \gg \mu$. Consequently the left hand side is vanishing unless $R \gtrsim \lambda$, and thus we can safely assume that $R\gtrsim \mu$ and discard the outer $P_{\les R}$ projection. The transversality assumption \eqref{eqn:lem bi L2 free KG:cond} together with \eqref{eqn:lem bi L2 free KG:trans} gives for any $\xi \in \supp \widehat{f}$ and $\eta \in \supp \widehat{g}$
         $$\Big| \frac{\xi}{\lr{\xi}_m} - \frac{\pm \eta}{\lr{\eta}_m}\Big| \gtrsim \alpha^2 \frac{\lambda}{\lr{\lambda}_m}.$$
Hence following the argument in \cite[Lemma 2.6]{Candy2018a} and applying a change of variables in the radial direction, we conclude that
        $$ \big\| e^{it\lr{\nabla}_m} f e^{\pm i t\lr{\nabla}_m} g \big\|_{L^2_{t,x}} \lesa \Big( \frac{\alpha^2 \lambda}{\lr{\lambda}_m}\Big)^{-\frac{1}{2}} (\alpha \mu)^{\frac{d-1}{2}} \|f \|_{L^2} \| g \|_{L^2}. $$
which suffices as $R\gtrsim \mu$ in this regime.
\end{proof}

After decomposing into atoms, the bilinear $L^2_{t,x}$ estimate in Lemma \ref{lem:bi L2 free KG} immediately implies a corresponding $U^2$ version. In particular, it also holds with the free waves replaced by elements of $S\times S$ (under the same support assumptions). On the other hand, the proof of Theorem \ref{thm:gwp} requires the bilinear $L^2_{t,x}$ estimate in Lemma \ref{lem:bi L2 free KG} to hold for functions in $S_w \times S$ (the high frequency term is placed in $S_w$). Unpacking definitions and using the embedding $V^{a'} \subset U^b$ for $a'<b$, we require a $U^2 \times U^b$ version of Lemma \ref{lem:bi L2 free KG} with $b>2$. Interpolating with a weaker Strichartz estimate can give such an estimate, but causes a high-low frequency loss in the wave regime $\lambda \gg \mu \gtrsim |m|$. In particular as it stands, Lemma \ref{lem:bi L2 free KG} is far from sufficient to conclude Theorem \ref{thm:gwp}.

There are two approaches in the literature to resolve this issue. The first, following Tataru \cite{Tataru2001}, is to work with null frames and develop a suitable functional framework in which a version of Lemma \ref{lem:bi L2 free KG} can be extended to perturbations of free waves. An alternative approach, first observed in \cite{Candy2018b} in the case of the wave maps equation, is to develop a suitable perturbed version of Lemma \ref{lem:bi L2 free KG} by arguing via the atomic Fourier restriction estimates in \cite{Candy2019a}. The second approach has the advantage that it simultaneously deals with both the massive and massless cases in the same underlying function spaces (up to the pull back via the free solution operator). As this flexibility is crucial in the proof of massless and non-relativistic limits, we follow the second approach and obtain a suitable generalisation of Lemma \ref{lem:bi L2 free KG} by arguing via an atomic bilinear restriction estimate for the hyperboloid.

\begin{theorem}[Atomic Bilinear Restriction for Klein-Gordon \cite{Candy2019a}]\label{thm:bilinear small scale KG}
Let $d\g 2$ and $1\les q \les 2\les b_2 \les b_1$ with
    $$\frac{1}{q} < \frac{d+1}{4}, \qquad  \frac{1}{b_1}>\frac{ 2}{(d+1)q}, \qquad \frac{1}{b_1} + \frac{1}{b_2} > \frac{1}{q}.$$
Let $\alpha, \mu, \lambda \in 2^\ZZ$ and $\kappa_1, \kappa_2 \in \mc{C}_\alpha$ with $0<\alpha \les 1$, $0<\mu\les \lambda$, and
    \begin{equation}\label{eq:thm bilinear small scale KG:cond}
        \frac{\lambda}{\mu} + \frac{\alpha \lr{\mu}_m}{|m|} \gg 1, \qquad \angle(\kappa_1, \pm \kappa_2) + \frac{|m|}{\lr{\mu}_m} \approx \alpha.
    \end{equation}
If $\supp \widehat{u} \subset \{ |\xi| \approx \lambda, \frac{\xi}{|\xi|} \in \kappa\}$ and $\supp \widehat{v} \subset \{|\xi| \approx \mu, \frac{\xi}{|\xi|} \in \kappa_2\}$ then
        $$ \big\| u v \big\|_{L^q_tL^2_x} \lesa \alpha^{\frac{d-1}{2} - \frac{2}{q}} \mu^{\frac{d}{2}-\frac{2}{q}} \lr{\mu}_m^{\frac{1}{q}} \Big( \frac{\lr{\lambda}_m}{\lr{\mu}_m}\Big)^{\frac{1}{q} - \frac{1}{2} + (d+1)(\frac{1}{2} - \frac{1}{b_2})} \Big( \frac{\mu \lr{\lambda}_m}{\lambda \lr{\mu}_m}\Big)^{\frac{1}{b_1}} \| e^{-it\lr{\nabla}_m} u \|_{U^{b_1}} \| e^{\mp it \lr{\nabla}_m} v \|_{U^{b_2}}.$$
\end{theorem}
\begin{proof}
Again, it is enough to consider $m\g 0$.
This is a special case of \cite[Theorem 1.7]{Candy2019a}, and the above statement follows by arguing as in \cite[Theorem 1.10]{Candy2019a}. More precisely, the case $\mu \gg m$ follows directly by arguing as in the proof of \cite[Theorem 1.10]{Candy2019a}. When $\mu \lesa m$ we define the phases
            $$ \Phi_1(\xi) = \Phi_2(\xi) = \lr{\xi}_m$$
and domains $\Lambda_1 = \{|\xi| \approx \lambda\}$, $\lambda_2 = \{ |\eta| \approx \mu\}$. The condition \eqref{eq:thm bilinear small scale KG:cond} implies that $\alpha \approx 1$ and $\lambda \gg \mu$. Hence result follows by directly applying Theorem \cite[Theorem 1.7]{Candy2019a} (in the fully elliptic regime $\lambda \lesa m$ this is immediate, in the mixed case $\lambda\gg m$ this requires a short argument as in the proof of \cite[Theorem 1.10]{Candy2019a}).
\end{proof}

The proof of global well-posedness in Theorem \ref{thm:gwp} only requires the case $q=2$. The case $q<2$ is only needed in the proof of Theorem \ref{thm:limit} via Theorem \ref{thm:rest+disp} below. Moreover, in the special case $b_1 = b_2 =2$,  we recover the frequency dependence in Lemma \ref{lem:bi L2 free KG} (provided $R\gtrsim \mu$).

The key importance of the Theorem \ref{thm:bilinear small scale KG} is that it completely resolves the loss of regularity issue discussed below the statement of Lemma \ref{lem:bi L2 free KG}. Namely we can place the high frequency wave in the weaker space $U^{b_1}$ for some $b_1>2$ \emph{without} losing a high-low factor $\lambda/\mu$ provided only that the low frequency wave remains in $U^2$. This is the crucial observation which allows us to treat the cubic Dirac equation via adapted function spaces. In the case of the wave equation, the analogous theorem has been successfully applied to the wave maps equation in \cite{Candy2018b}, where a direct proof of a special case of Theorem \ref{thm:bilinear small scale KG} was given relying on a wave table construction of Tao \cite{Tao2001b}. The proof of the more general version contained in Theorem \ref{thm:bilinear small scale KG} follows from a similar argument after adapting the wave table construction to the Klein-Gordon equation.

We now adapt the previous results to spinors lying in our solution spaces $S$ and $S_w$. Note that the following bilinear estimates again exploit the fact that the product $\overline{\varphi}\psi$ is a null form to gain an additional factor of $\alpha$ when compared to the dependence in Lemma \ref{lem:bi L2 free KG} and Theorem \ref{thm:bi trans}.

\begin{theorem}[Transverse case]\label{thm:bi trans}
Let $d\g 2$ and $\frac{1}{2} \les \frac{1}{a} \les  \frac{1}{2} + \frac{1}{12}$. Let $\alpha, \mu, \lambda, R \in 2^\ZZ$ with $0<\mu \les \lambda$ and $0<\alpha \les 1$. Assume that the transversality condition $\frac{\lambda}{\mu} + \frac{\alpha \lr{\mu}_m}{|m|} \gg 1$ hold, and define the constant
        $$ \mb{C} = \mb{C}(\alpha, R, \mu, \lambda) = \alpha^{\frac{1}{2}} (\min\{\mu, R\})^{\frac{1}{2}} (\min\{ \alpha \mu, R\})^{\frac{d-2}{2}} \Big( \frac{\lr{\lambda}_m}{\lambda}\Big)^\frac{1}{2}.$$
Then we have
        $$ \sum_{\substack{ \kappa_1, \kappa_2 \in \mc{C}_\alpha  \\ \ma(\kappa_1, \kappa_2) + \frac{|m|}{\lr{\mu}_m} \approx \alpha}} \big\| P_{\les R}(\overline{\varphi}_{\lambda, \kappa_1} \psi_{\mu, \kappa_2}) \big\|_{L^2_{t,x}} \lesa \Big( \frac{ \lambda \lr{\mu}_m}{\alpha \min\{R, \mu\} \lr{\lambda}_m}\Big)^{ 6d (\frac{1}{a} - \frac{1}{2})} \mb{C} \| \varphi_{\lambda} \|_{S_w} \| \psi_{\mu} \|_{S}$$
and
        $$ \sum_{\substack{ \kappa_1, \kappa_2 \in \mc{C}_\alpha  \\ \ma(\kappa_1, \kappa_2) + \frac{|m|}{\lr{\mu}_m} \approx \alpha}} \big\| P_{\les R}( \overline{\varphi}_{\lambda, \kappa_1} \psi_{\mu, \kappa_2}) \big\|_{L^2_{t,x}} \lesa  \Big( \frac{\lambda}{\alpha \min\{\mu, R\}} \Big)^{6d(\frac{1}{a} - \frac{1}{2})}\mb{C} \| \varphi_{\lambda} \|_{S_w} \| \psi_{\mu} \|_{S_w}.$$
\end{theorem}
\begin{proof}
For ease of notation, in the following we write $ \psi^{\pm_1} = \Pi_{\pm_1} \psi$ and $U^b_{\pm} = e^{\pm i t\lr{\nabla}_m} U^b$ and define the constant
        $$ \mb{C}^*(b) = \alpha^{\frac{d-3}{2}} \mu^{\frac{d-1}{2}} \Big( \frac{\lr{\lambda}_m}{\lambda} \Big)^{\frac{1}{2}} \Big( \frac{\lr{\lambda}_m}{\lr{\mu}_m}\Big)^{(d+1)(\frac{1}{2} - \frac{1}{b})} \Big( \frac{\lambda \lr{\mu}_m}{\mu \lr{\lambda}_m}\Big)^{\frac{1}{12}}. $$
An application of Theorem \ref{thm:bilinear small scale KG} (discarding the outer $P_{\les R}$ multiplier) implies that for any choice of signs $\pm_1$ and $\pm_2$, and any $\frac{1}{2} - \frac{1}{12} \les \frac{1}{b} \les \frac{1}{2}$ we have
  \begin{equation}\label{eq:thm bi trans:restric}
    \big\| P_{\les R}\big( ( \varphi^{\pm_1}_{\lambda, \kappa_1})^\dagger  \psi^{\pm_2}_{\mu, \kappa_2}\big) \big\|_{L^2_{t,x}}
        \lesa \mb{C}^*(b)  \| \varphi^{\pm_1}_{\lambda, \kappa_1} \|_{U^{\frac{12}{5}}_{\pm_1}} \| \psi^{\pm_2}_{\mu, \kappa_2} \|_{U^{b}_{\pm_2}}.
  \end{equation}
We now improve this bound for the null form $\overline{\varphi}\psi$. Namely, the null form bound \eqref{eqn:Pi null struc} together with Lemma \ref{lem:Up basic prop} gives (adapting the notation in Theorem \ref{thm:bi high mod})
    $$ \| ( \Pi_{\pm_1} - \Pi^{\kappa_1}_{\pm_1}) \varphi^{\pm_1}_{\lambda, \kappa_1} \|_{U^{\frac{12}{5}}_{\pm_1}} \lesa \alpha \| \varphi^{\pm_1}_{\lambda, \kappa_1} \|_{U^{\frac{12}{5}}_{\pm_1}}, \qquad \| \Pi^{\kappa_1}_{\pm_1} \varphi^{\pm_1}_{\lambda, \kappa_1} \|_{U^{\frac{12}{5}}_{\pm_1}} \lesa \| \varphi^{\pm_1}_{\lambda, \kappa_1} \|_{U^{\frac{12}{5}}_{\pm_1}}$$
and
    $$ \| \gamma^0 ( \Pi_{\pm_2} - \Pi^{\kappa_2}_{\pm_2}) \psi^{\pm_2}_{\mu, \kappa_2} \|_{U^{b}_{\pm_2}} + \| \Pi_{\pm_1}^{\kappa_1} \gamma^0 \Pi^{\kappa_2}_{\pm_2} \psi^{\pm_2}_{\mu, \kappa_2} \|_{U^{b}_{\pm_2}} \lesa \alpha \| \psi^{\pm_2}_{\mu, \kappa_2} \|_{U^{b}_{\pm_2}}, \qquad \| \gamma^0 \psi^{\pm_2}_{\mu, \kappa_2} \|_{U^{b}_{\pm_2}} \lesa \| \psi^{\pm_2}_{\mu, \kappa_2} \|_{U^{b}_{\pm_2}}. $$
Consequently, as
      $$( \varphi^{\pm_1}_{\lambda, \kappa_1})^\dagger  \psi^{\pm_2}_{\mu, \kappa_2} = ( \Pi_{\pm_1} \varphi^{\pm_1}_{\lambda, \kappa_1})^\dagger  \Pi_{\pm_2} \psi^{\pm_2}_{\mu, \kappa_2}$$
the decomposition \ref{eqn:thm bi trans:null decomp} together with the square sum estimate in Lemma \ref{lem:Up basic prop} implies that we can improve the bilinear $L^2_{t,x}$ bound \eqref{eq:thm bi trans:restric} to
  \begin{equation}\label{eq:thm bi trans:restric+null}
    \sum_{\substack{\kappa_1, \kappa_2 \in \mc{C}_\alpha \\ \ma(\kappa_1, \kappa_2) + \frac{|m|}{\lr{\mu}_m}\approx \alpha }}\big\| P_{\les R}\big( \overline{ \varphi^{\pm_1}_{\lambda, \kappa_1}}  \psi^{\pm_2}_{\mu, \kappa_2}\big) \big\|_{L^2_{t,x}}
        \lesa  \mb{C}^*(b) \alpha^{1 -\frac{d-1}{6}} \| \varphi^{\pm_1}_{\lambda} \|_{U^{\frac{12}{5}}_{\pm_1}} \| \psi^{\pm_2}_{\mu} \|_{U^{b}_{\pm_2}}
  \end{equation}
On the other hand, after decomposing into $U^2_{\pm}$ atoms (as in the proof of Lemma \ref{lem:wave stri}) Lemma \ref{lem:bi L2 free KG} implies that
        $$\big\| P_{\les R} \big( ( \varphi^{\pm_1}_{\lambda, \kappa_1})^\dagger  \psi^{\pm_2}_{\mu, \kappa_2} \big) \big\|_{L^2_{t,x}}
        \lesa \alpha^{-1} \mb{C} \| \varphi_{\lambda, \kappa_1}^{\pm_1} \|_{U^2_{\pm_1}} \|  \psi_{\mu, \kappa_2}^{\pm_2} \|_{U^2_{\pm_2}}$$
where $\mb{C}$ is as in the statement of the theorem. Hence arguing as above, we obtain the improved estimate
    \begin{equation}\label{eq:thm bi trans:biL2 in U2}
    \sum_{\substack{\kappa_1, \kappa_2 \in \mc{C}_\alpha \\ \ma(\kappa_1, \kappa_2) + \frac{|m|}{\lr{\mu}_m}\approx \alpha }} \big\| P_{\les R} \big( \overline{\varphi^{\pm_1}_{\lambda, \kappa_1}}  \psi^{\pm_2}_{\mu, \kappa_2} \big) \big\|_{L^2_{t,x}}
        \lesa  \mb{C} \| \varphi_{\lambda}^{\pm_1} \|_{U^2_{\pm_1}} \|  \psi_{\mu}^{\pm_2} \|_{U^2_{\pm_2}}.
    \end{equation}
Interpolating (using Lemma \ref{lem:bil-inter}) between \eqref{eq:thm bi trans:restric+null} and \eqref{eq:thm bi trans:biL2 in U2} we conclude that for any $0\les \theta \les 1$ we have
    \begin{equation}\label{eq:thm bi trans:biL2 interp}
    \sum_{\substack{\kappa_1, \kappa_2 \in \mc{C}_\alpha \\ \ma(\kappa_1, \kappa_2) + \frac{|m|}{\lr{\mu}_m}\approx \alpha }} \big\| P_{\les R} \big( \overline{\varphi^{\pm_1}_{\lambda, \kappa_1}}  \psi^{\pm_2}_{\mu, \kappa_2} \big) \big\|_{L^2_t L^2_x}
        \lesa  \mb{C} \Big( \frac{ \mb{C}^*(b) \alpha^{1-\frac{d-1}{6}}}{ \mb{C}}\Big)^\theta \| \varphi_{\lambda}^{\pm_1} \|_{U^{b_1}_{\pm_1}} \| \psi_{\mu}^{\pm_2} \|_{U^{b_2}_{\pm_2}}
    \end{equation}
where $\frac{1}{b_1} = \frac{1-\theta}{2} + \frac{5\theta}{12}$ and $\frac{1}{b_2} = \frac{1-\theta}{2} + \frac{\theta}{b}$.
Taking $b_1 = a'$ and $b=2$ (which implies that $\theta = 12(\frac{1}{a} - \frac{1}{2})$), and noting that since $\min\{ R, \alpha \mu\} \g \alpha \min\{R, \mu\}$ we have
        $$ \frac{ \mb{C}^*(2) \alpha^{1-\frac{d-1}{6}}}{\mb{C}} \les \Big( \frac{\mu}{\min\{R, \mu\}}\Big)^{\frac{d-1}{2}} \Big( \frac{\lambda \lr{\mu}_m}{\mu \lr{\lambda}_m}\Big)^{\frac{1}{12}} \alpha^{-\frac{d-1}{6}} \les \Big( \frac{ \lambda \lr{\mu}_m}{\alpha \min\{R, \mu\} \lr{\lambda}_m}\Big)^{\frac{d-1}{2}}$$
the $S_w \times S$ estimate follows. Similarly, choosing $b_1 = b_2 = a'$ and $b=\frac{12}{5}$, and noting that
        $$ \frac{ \mb{C}^*(12/5) \alpha^{1-\frac{d-1}{6}}}{\mb{C}} \les \Big( \frac{\mu}{\min\{R, \mu\}}\Big)^{\frac{d-1}{2}} \Big( \frac{\lr{\lambda}_m}{\lr{\mu}_m}\Big)^{\frac{d+1}{12}} \Big( \frac{\lambda \lr{\mu}_m}{\mu \lr{\lambda}_m}\Big)^{\frac{1}{12}} \alpha^{-\frac{d-1}{6}} \les \Big( \frac{ \lambda }{\alpha \min\{R, \mu\}}\Big)^{\frac{d-1}{2}}$$
gives the claimed $S_w\times S_w$ bound.
\end{proof}

Theorem \ref{thm:bi trans} gives very good control over the transverse interactions where $\frac{\lambda}{\mu} + \frac{\alpha \lr{\mu}_m}{|m|} \gg 1$. On the other hand,  the remaining (non-transverse) frequency interactions, namely the high-high and small angle case $\lambda \approx \mu$ and $\alpha \approx \frac{|m|}{\mu}$, are covered by Theorem \ref{thm:bi nontrans}. To simplify the arguments in the following section, it is useful to have an estimate that combines both these cases. This combined estimate involves a non-sharp dependence on the frequency parameters, but still suffices for our purposes.

\begin{corollary}\label{cor:bi unified}
Let $d\g 2$, $2\les q \les \infty$, and $0<\frac{1}{a} - \frac{1}{2} \les \frac{1}{100d}$. If  $0<\mu \les \lambda$, $\frac{|m|}{\lr{\mu}_m} \lesa  \alpha \les 1$, and $R>0$,  then
        \begin{align*}
            \bigg\| \sum_{\substack{ \kappa_1, \kappa_2 \in \mc{C}_\alpha  \\ \ma(\kappa_1, \kappa_2) + \frac{|m|}{\lr{\mu}_m} \approx \alpha}} P_{\les R}(\overline{\varphi}_{\lambda, \kappa_1} \psi_{\mu, \kappa_2}) \bigg\|_{L^q_t L^2_x}
            &\lesa \big( \alpha \lr{\mu}_m\big)^{\frac{1}{4}} \lr{\rho}_m^{\frac{1}{4}} \rho^{\frac{d-1}{2} - \frac{1}{q}} \Big( \frac{\mu \lr{\lambda}_m}{\lambda \lr{\mu}_m}\Big)^{\frac{1}{4}} \| \varphi_{\lambda} \|_{S_w} \| \psi_{\mu} \|_{S}
        \end{align*}
and
        $$ \bigg\| \sum_{\substack{ \kappa_1, \kappa_2 \in \mc{C}_\alpha  \\ \ma(\kappa_1, \kappa_2) + \frac{|m|}{\lr{\mu}_m} \approx \alpha}} P_{\les R}(\overline{\varphi}_{\lambda, \kappa_1} \psi_{\mu, \kappa_2}) \bigg\|_{L^q_t L^2_x} \lesa  \big( \alpha \lr{\mu}_m\big)^{\frac{1}{4}} \lr{\rho}_m^{\frac{1}{4}} \rho^{\frac{d-1}{2} - \frac{1}{q}} \Big( \frac{\mu \lr{\lambda}_m}{\lambda \lr{\mu}_m}\Big)^{\frac{1}{4}} \Big( \frac{\lr{\lambda}_m}{\lr{\mu}_m}\Big)^{6d(\frac{1}{a}-\frac{1}{2})} \| \varphi_{\lambda} \|_{S_w} \| \psi_{\mu} \|_{S_w}$$
where we let $\rho = \min\{\mu, R\}$.
\end{corollary}
\begin{proof}
We simply observe that since $\min\{\alpha \mu, R\} \les \rho$, letting $\delta = 6(d-1) (\frac{1}{a} - \frac{1}{2})$ we can bound the constant in Theorem \ref{thm:bi trans} by
	\begin{align*}
       \Big( \frac{ \lambda \lr{\mu}_m}{\alpha \rho \lr{\lambda}_m}\Big)^{\delta}  \alpha^{\frac{1}{2}} \rho^{\frac{1}{2}} (\min\{ \alpha \mu, R\})^{\frac{d-2}{2}} \Big( \frac{\lr{\lambda}_m}{\lambda}\Big)^\frac{1}{2}
       &\lesa \big( \alpha \lr{\mu}\big)^{\frac{1}{2} - \delta} \lr{\rho}^\delta \rho^{\frac{d-2}{2}} \Big(\frac{\mu \lr{\lambda}_m}{\lambda \lr{\mu}_m}\Big)^{\frac{1}{2} - \delta} \Big( \frac{\rho\lr{\mu}_m}{ \mu \lr{\rho}_m}\Big)^\delta \Big( \frac{\rho}{\mu}\Big)^{\frac{1}{2} - 2 \delta}\\
       &\lesa \big( \alpha \lr{\mu}\big)^{\frac{1}{2} - \delta} \lr{\rho}^\delta \rho^{\frac{d-2}{2}} \Big(\frac{\mu \lr{\lambda}_m}{\lambda \lr{\mu}_m}\Big)^{\frac{1}{2} - \delta}.
    \end{align*}
Similarly, in the nontransverse setting of Theorem \ref{thm:bi nontrans}, we have $\mu \approx \lambda$ and $\alpha \approx \frac{|m|}{\lr{\mu}_m} \lesa \frac{\lr{\rho}_m}{\lr{\mu}_m}$. Hence we can bound the constant given by Theorem \ref{thm:bi nontrans} by
	\begin{align*}
		\alpha^{1-\frac{1}{d}} \rho^{\frac{1}{2} - \frac{1}{d}} \big( \min\{\alpha \mu, R\}\big)^{\frac{d-1}{2} - 1 + \frac{1}{d}} \lr{\mu}_m^\frac{1}{2}
                    &\lesa \alpha^{\frac{1}{2}} \rho^{\frac{d-2}{2}} \lr{\mu}_m^{\frac{1}{2}} \\
                    &\lesa (\alpha \lr{\mu}_m)^{\frac{1}{2} - \delta} \lr{\rho}_m^\delta \rho^{\frac{d-2}{2}} \approx  (\alpha \lr{\mu}_m)^{\frac{1}{2} - \delta} \lr{\rho}_m^\delta \rho^{\frac{d-2}{2}} \Big( \frac{\mu \lr{\lambda}_m}{\lambda \lr{\mu}_m}\Big)^{\frac{1}{2} - \delta}.
	\end{align*}
Consequently, for all frequency interactions, we obtain the bilinear $L^2_{t,x}$ estimate
    $$\bigg\| \sum_{\substack{ \kappa_1, \kappa_2 \in \mc{C}_\alpha  \\ \ma(\kappa_1, \kappa_2) + \frac{|m|}{\lr{\mu}_m} \approx \alpha}} P_{\les R}(\overline{\varphi}_{\lambda, \kappa_1} \psi_{\mu, \kappa_2}) \bigg\|_{L^2_{t,x}}
            \lesa (\alpha \lr{\mu}_m)^{\frac{1}{2} - \delta} \lr{\rho}_m^\delta \rho^{\frac{d-2}{2}} \Big( \frac{\mu \lr{\lambda}_m}{\lambda \lr{\mu}_m}\Big)^{\frac{1}{2} - \delta} \| \varphi_{\lambda} \|_{S_w} \| \psi_{\mu} \|_{S}. $$
On the other hand, since
        $$ \alpha \rho^{\frac{1}{2}} \big(\min\{\alpha \mu, R\}\big)^{\frac{d-1}{2}}\lesa \alpha^{\frac{1}{2}-\delta} \rho^{\frac{d}{2}} \lesa (\alpha \lr{\mu}_m)^{\frac{1}{2} - \delta} \lr{\rho}_m^\delta \rho^{\frac{d-1}{2}} \Big( \frac{\mu \lr{\lambda}_m}{\lambda \lr{\mu}_m}\Big)^{\frac{1}{2} - \delta}$$
an application of \eqref{eqn:spatial v2} (with $r=2$) gives
    \begin{align*}
      \bigg\| \sum_{\substack{ \kappa_1, \kappa_2 \in \mc{C}_\alpha  \\ \ma(\kappa_1, \kappa_2) + \frac{|m|}{\lr{\mu}_m} \approx \alpha}} P_{\les R}(\overline{\varphi}_{\lambda, \kappa_1} \psi_{\mu, \kappa_2}) \bigg\|_{L^\infty_t L^2_x}
        &\lesa  \alpha \rho^{\frac{1}{2}} \big(\min\{\alpha \mu, R\}\big)^{\frac{d-1}{2}}  \| \varphi_{\lambda} \|_{L^\infty_t L^2_x} \| \psi_{\mu} \|_{L^\infty_t L^2_x} \\
        &\lesa (\alpha \lr{\mu}_m)^{\frac{1}{2} - \delta} \lr{\rho}_m^\delta \rho^{\frac{d-1}{2}} \Big( \frac{\mu \lr{\lambda}_m}{\lambda \lr{\mu}_m}\Big)^{\frac{1}{2} - \delta} \| \varphi_{\lambda} \|_{S_w} \| \psi_{\mu} \|_{S_w}.
    \end{align*}
Letting $\delta = \frac{1}{4}$, the $S_w \times S$ bound now follows by convexity. The $S_w \times S_w$ bound is similar, the only difference is the slight high frequency loss coming from Theorem \ref{thm:bi trans}.
\end{proof}

The final estimate we state is required in the proof of Theorem \ref{thm:limit}, and combines the bilinear restriction estimate in Theorem \ref{thm:bilinear small scale KG} with the dispersive estimate for the Klein-Gordon equation.

\begin{theorem}\label{thm:rest+disp}
Let $d\g 2$, $\frac{1}{2}\les \frac{1}{q} < \frac{d+1}{4}$ and $R>2$. Suppose that $f, g \in L^1_x$ have Fourier support in the ball $\{ |\xi| \les R\}$. Then
        $$ \sup_{|m| \les 1} \| \overline{\mc{U}_m(t) f} \mc{U}_m(t)  g \|_{L^q_t L^2_x} \lesa R^{\frac{3d}{2} - \frac{1}{q}} \| f\|_{L^1_x} \| g \|_{L^1_x}. $$
\end{theorem}
\begin{proof}
Clearly, it is enough to consider $0\les m\les 1$ and let us fix $m\in \{0,1\}$ for the moment. For $\lambda \g \mu >0$ and $0<\alpha \les 1$, under the transversality condition $\frac{\lambda}{\mu} + \frac{\alpha \lr{\mu}_m}{m} \gg 1$ (trivial if $m=0$) Theorem \ref{thm:bilinear small scale KG} implies that we have the bilinear restriction estimate
        \begin{align}
            \sum_{\substack{\kappa_1, \kappa_2\in \mc{C}_{\alpha} \\ \angle(\kappa_1, \kappa_2) + \frac{m}{\lr{\mu}_m} \approx \alpha}} \| (e^{\pm_1 it \lr{\nabla}_m} f_{\lambda, \kappa_1})^\dagger e^{\pm_2 i t \lr{\nabla}_m} g_{\mu, \kappa_2} \|_{L^q_t L^2_x}
            &\lesa  \alpha^{\frac{d-1}{2}-\frac{2}{q}}  \mu^{\frac{d}{2}-\frac{2}{q}} \Big( \frac{\mu}{\lambda}\Big)^{\frac{1}{2}}\lr{\lambda}_m^{\frac1q}\| f_{\lambda} \|_{L^2} \| g_{\mu} \|_{L^2} \notag \\
            &\lesa   \alpha^{\frac{d-1}{2}-\frac{2}{q}}  \mu^{d+\frac12-\frac{2}{q}} \lambda^{\frac{d-1}{2}}\lr{\lambda}_m^{\frac1q}
          \| f_{\lambda} \|_{L^1} \| g_{\mu} \|_{L^1}. \label{eqn:thm rest+disp:bilinear rest}
        \end{align}
On the other hand, if $m=1$, standard stationary phase estimates give for any $0\les \theta \les 1$ the interpolated dispersive bound
        \begin{equation}\label{eqn:KG disp}
            \| e^{\pm_2 it \lr{\nabla}} f_\mu \|_{L^\infty_x} \lesa |t|^{-\frac{d-1 + \theta}{2}} \mu^{(1-\theta)\frac{d+1}{2}} \lr{\mu}^{\theta\frac{d+2}{2}} \|f_\mu \|_{L^1_x}
        \end{equation}
see for instance \cite{Brenner1985}. In particular, provided that $\frac{1}{q} < \frac{d-1+\theta}{2}$ we have for any $ 0<\alpha \les 1$ and $T>0$
        \begin{align*}
        &  \Big(\sum_{\kappa\in \mc{C}_{\alpha} } \| e^{\pm_2it \lr{\nabla}} g_{\mu, \kappa} \|_{L^q_t L^\infty_x}^2\Big)^\frac{1}{2}\\
                \lesa{}&  T^{\frac{1}{q}} \alpha^{\frac{d-1}{2}} \mu^{\frac{d}{2}} \| g_{\mu}\|_{L^2} + (\# \mc{C}_\alpha)^{\frac{1}{2}} \lr{\mu}^{\frac{d+1+\theta}{2}}\Big(\frac{\mu}{\lr{\mu}}\Big)^{(1-\theta)\frac{d+1}{2}} \| g_\mu\|_{L^1_x} \| |t|^{-\frac{d-1+\theta}{2}} \|_{L^q_t(|t|>T)} \\
                \lesa{}& \Big(T^{\frac{1}{q}} \alpha^{\frac{d-1}{2}} \mu^{d} + T^{\frac{1}{q} - \frac{d-1+\theta}{2}} \alpha^{-\frac{d-1}{2}}  \lr{\mu}^{\frac{d+1+\theta}{2}}\Big(\frac{\mu}{\lr{\mu}}\Big)^{(1-\theta)\frac{d+1}{2}}\Big)\|g_\mu \|_{L^1_x}.
        \end{align*}
For simplicity, we use $\mu\les \lr{\mu}$. Optimising in $T$,  we conclude that for any $\frac{1}{q} < \frac{d-1+\theta}{2}$  we have
        $$ \Big(\sum_{\kappa\in \mc{C}_{\alpha}} \| e^{\pm_2 it \lr{\nabla}} g_{\mu, \kappa} \|_{L^q_t L^\infty_x}^2\Big)^\frac{1}{2} \lesa \alpha^{\frac{d-1}{2} - \frac{2}{q}} \lr{\mu}^{d-\frac{1}{q}} (\alpha \lr{\mu})^{\frac{2\theta}{q(d-1+\theta)}} \| g_\mu \|_{L^1_x}.$$
Therefore, for $\lambda\sim \mu$, via H\"older's and Bernstein's inequality we have
        \begin{align}
            &\sum_{\substack{\kappa_1, \kappa_2\in \mc{C}_{\alpha} \\ \angle(\kappa_1, \kappa_2) \lesa \alpha }} \| (e^{\pm_1 it \lr{\nabla}} f_{\lambda, \kappa_1})^\dagger e^{\pm_2 i t \lr{\nabla}} g_{\mu, \kappa_2} \|_{L^q_t L^2_x}\notag \\
                        &\lesa \Big( \sum_{\substack{\kappa_1, \kappa_2\in \mc{C}_{\alpha} \\ \angle(\kappa_1, \kappa_2) \lesa \alpha}} \| e^{\pm_1 it \lr{\nabla}} f_{\lambda, \kappa_1}\|_{L^\infty_t L^2_x}^2 \Big)^\frac{1}{2} \Big( \sum_{\substack{\kappa_1, \kappa_2\in \mc{C}_{\alpha} \\ \angle(\kappa_1, \kappa_2) \lesa \alpha}} \| e^{\pm_2 i t \lr{\nabla}} g_{\mu, \kappa_2} \|_{L^q_t L^\infty_x}^2 \Big)^\frac{1}{2} \notag \\
                        &\lesa \alpha^{\frac{d-1}{2} - \frac{2}{q}} \lr{\mu}^{d-\frac{1}{q}} \mu^{\frac{d}{2}} (\alpha \lr{\mu})^{\frac{2\theta}{q(d-1+\theta)}} \| f_\lambda \|_{L^1} \| g_\mu \|_{L^1_x} \label{eqn:thm rest+disp:bilinear via disp}
        \end{align}
Choosing $\theta>\frac{3-d}{2}$ covers the range $\frac{1}{2}\les \frac{1}{q} < \frac{d+1}{4}$.
Consequently, applying the bound \eqref{eqn:thm rest+disp:bilinear rest} in the transverse case $\frac{\lambda}{\mu} + \frac{\alpha \lr{\mu}_m}{m} \gg 1$, and \eqref{eqn:thm rest+disp:bilinear via disp} in the non-tranverse case $\frac{\lambda}{\mu} + \frac{\alpha \lr{\mu}_m}{m} \approx  1$ (note that \eqref{eqn:thm rest+disp:bilinear via disp} is only needed when $m\not= 0$, i.e. $m=1$ and $\alpha \lr{\mu} \approx 1$) we conclude that for any $0<\alpha\les 1$ we have
        $$   \sum_{\substack{\kappa_1, \kappa_2\in \mc{C}_{\alpha} \\ \angle(\kappa_1, \kappa_2) + \frac{m}{\lr{\mu}_m} \approx \alpha}} \| (e^{\pm_1 it \lr{\nabla}_m} f_{\lambda, \kappa_1})^\dagger e^{\pm_2 i t \lr{\nabla}_m} g_{\mu, \kappa_2} \|_{L^q_t L^2_x}
            \lesa  \alpha^{\frac{d-1}{2}-\frac{2}{q}} \mu^{\frac{d}{2}} \lr{\mu}_m^{d-\frac{1}{q}}  \Big( \frac{\lr{\lambda}_m}{\lr{\mu}_m}\Big)^{\frac{d-1}{2} +\frac{1}{q}}\| f_{\lambda} \|_{L^1} \| g_{\mu} \|_{L^1} $$
Arguing as in the proof of \eqref{eqn:spatial v2}, the null structure bound \eqref{eqn:Pi null struc} then gives the improved bound
        $$   \sum_{\substack{\kappa_1, \kappa_2\in \mc{C}_{\alpha} \\ \angle(\kappa_1, \kappa_2) + \frac{m}{\lr{\mu}_m} \approx \alpha}} \| \overline{ \mc{U}_m(t)f_{\lambda, \kappa_1}} \mc{U}_m(t) g_{\mu, \kappa_2} \|_{L^q_t L^2_x}
            \lesa  \alpha^{\frac{d+1}{2}-\frac{2}{q}}
             \mu^{\frac{d}{2}} \lr{\mu}_m^{d-\frac{1}{q}}  \Big( \frac{\lr{\lambda}_m}{\lr{\mu}_m}\Big)^{\frac{d-1}{2} +\frac{1}{q}}
             \| f_{\lambda} \|_{L^1} \| g_{\mu} \|_{L^1}. $$
Consequently, applying the decomposition \eqref{eqn:whitney decom} we see that for any $m\in \{0, 1\}$ we have
        $$   \| \overline{ \mc{U}_m(t)f_{\lambda}} \mc{U}_m(t) g_{\mu} \|_{L^q_t L^2_x}
            \lesa
              \mu^{\frac{d}{2}} \lr{\mu}_m^{d-\frac{1}{q}}  \Big( \frac{\lr{\lambda}_m}{\lr{\mu}_m}\Big)^{\frac{d-1}{2} +\frac{1}{q}}
           \| f\|_{L^1} \| g \|_{L^1}. $$
          Dyadically summing up $\mu\les \lambda\lesa R$ implies the claimed estimate for $m \in \{0,1\}$, and rescaling in general.\end{proof}

\section{Multilinear estimates}\label{sec:mult-est}
For ease of notation, throughout this section we write
    $$ \mb{D} = \Big( \frac{\lambda_1 \lambda_2 \lambda_3}{\lambda_0}\Big)^{\frac{d-2}{2}} \Big( \frac{\lr{\lambda_1}_m \lr{\lambda_2 }_m\lr{\lambda_3}_m}{\lr{\lambda_0}_m}\Big)^{\frac{1}{2}} $$
    and
    $$\lambda_{min} = \min\{\lambda_0, \lambda_1, \lambda_2, \lambda_2\}, \quad \lambda_{med} = \min(\{ \lambda_0, \lambda_1, \lambda_2, \lambda_3\} \setminus \{ \lambda_{min}\}).$$
 Our goal is to prove the following trilinear bound which gives control over the frequency localised nonlinearity in our iteration space $S$, uniformly in $m \in \R$.
\begin{theorem}[Control of frequency localised nonlinearity]\label{thm:nonlin}
Let $0<\frac{1}{a} - \frac{1}{2} \les \frac{1}{100d}$, $s_d=\frac{d-1}{2}$, $\sigma_d=\frac{d-2}{2}$. There exists $\epsilon>0$ such that for any $m \in \R$ and $\lambda_j \in 2^\ZZ$ ($j = 0, 1,2, 3$) we have
        \begin{align*}
            \Big| \int_{\RR^{1+d}} \overline{\varphi}_{\lambda_0} \psi^{[1]}_{\lambda_1} \overline{\psi}^{[2]}_{\lambda_2} \psi^{[3]}_{\lambda_3} dtdx \Big|
                        \lesa \Big( \frac{\lambda_{min}}{\lambda_{med}}\Big)^{\epsilon} \mb{D} \| \varphi_{\lambda_0} \|_{S_{w,m}} \| \psi^{[1]}_{\lambda_1} \|_{S_m} \| \psi^{[2]}_{\lambda_2} \|_{S_m} \| \psi^{[3]}_{\lambda_3} \|_{S_m},
        \end{align*}
where $\epsilon>\frac{d-2}{2}$ in the case $\lambda_0\ll  \min\{ \lambda_{med},|m|\}$.
\end{theorem}
\begin{proof}
It is enough to consider $m\g 0$.
For $ \alpha, \tilde{\alpha} \in 2^{-\NN_0}$ and $ \boldsymbol{\lambda} = (\lambda_0, \lambda_1, \lambda_2, \lambda_3) \in (2^\ZZ)^4$ we define the quadrilinear quantity
		$$ B(\alpha, \tilde{\alpha}, \boldsymbol{\lambda})
			= \Big| \int_{\RR^{1+d}} \Big(\sum_{\substack{\kappa_0, \kappa_1 \in \mc{C}_\alpha\\ \ma(\kappa_0, \kappa_1) + \frac{m}{\lr{\lambda_{0, 1}}_m}\approx \alpha}} \overline{\varphi}_{\lambda_0, \kappa_0} \psi_{\lambda_1, \kappa_1}\Big) \Big(\sum_{\substack{\kappa_2, \kappa_3 \in \mc{C}_{\tilde{\alpha}}\\ \ma(\kappa_2, \kappa_3) + \frac{m}{\lr{\lambda_{2, 3}}_m}\approx \tilde\alpha}}\overline{\psi}_{\lambda_2, \kappa_2} \psi_{\lambda_3, \kappa_3}\Big)  dx dt \Big|$$
			For ease of notation, we introduce the short hand
            $$\lambda_{i,j} = \min\{\lambda_i, \lambda_j\}, \qquad \psi_{\lambda_j} = \psi^{[j]}_{\lambda_j}.$$
Our goal is then to prove that
		\begin{equation}\label{eqn:thm nonlin:goal}
    \sum_{\alpha, \tilde{\alpha} \in 2^{-\NN_0}} B(\alpha, \tilde{\alpha}, \boldsymbol{\lambda}) \lesa  \Big( \frac{\lambda_{min}}{\lambda_{med}}\Big)^{\epsilon} \mb{D} \| \varphi_{\lambda_0} \|_{S_{w}} \| \psi_{\lambda_1} \|_{S} \| \psi_{\lambda_2} \|_S \| \psi_{\lambda_3} \|_S.
        \end{equation}
        again with the provision that $\epsilon>\frac{d-2}{2}$ in the case $\lambda_0\ll \min\{ \lambda_{med},m\}$. We consider separately the cases $\lambda_0 \gtrsim \lambda_1$ and $\lambda_0 \ll \lambda_1$. The former is the more difficult of the two, and we break it further into the four cases
	$$ \alpha \lr{\lambda_0}_m \lr{\lambda_1}_m \lesa \tilde{\alpha} \lr{\lambda_2}_m \lr{\lambda_3}_m, \qquad \text{$\overline{\varphi}_{\lambda_0} \psi_{\lambda_1}$ nonresonant}, \qquad \text{$\overline{\psi}_{\lambda_0} \psi_{\lambda_1}$ nonresonant}, \qquad \text{remaining cases}  $$
where the precise definition of ‘nonresonant’ is given below.\\

\noindent \underline{\textbf{Case 1:} $\lambda_0 \gtrsim \lambda_1$.}\\

Our goal is to prove that \eqref{eqn:thm nonlin:goal} holds when $\lambda_0 \gtrsim \lambda_1$. We may freely add the restriction  $\lambda_2\geq \lambda_3$ because the subcase where $\lambda_2 < \lambda_3$ would then follow by replacing waves with their complex conjugates. Moreover, if in addition we assume that $\lambda_1 \gtrsim \lambda_3$ and we can prove the bound \eqref{eqn:thm nonlin:goal} with $\psi_{\lambda_2} \in S_w$, then the remaining case $\lambda_1 \ll \lambda_3$ would then follow by exchanging the roles of the pairs $\overline{\varphi_{\lambda_0}}\psi_{\lambda_1}$ and $\overline{\psi_{\lambda_2}}\psi_{\lambda_3}$ and noting that $\lambda_1 \ll \lambda_3 \lesa \lambda_2$ implies that $B$ vanishes unless $\lambda_0 \lesa \lambda_2$.  Finally, as we now have $\lambda_0\g \lambda_1$ and $\lambda_1, \lambda_2 \g \lambda_3$, we may also assume that $\lambda_0 \approx \max\{\lambda_1, \lambda_2\}$ as otherwise $B$ vanishes. In particular, we may freely restrict our attention to the case $\lambda_1 \lambda_2 \approx \lambda_{1,2}\lambda_0$.  To sum up this discussion, we have reduced Case 1 to proving the slightly stronger estimate
        \begin{equation}\label{eqn:thm nonlin:goal case1}
         \sum_{\alpha, \tilde{\alpha} \in 2^{-\NN_0}} B(\alpha, \tilde{\alpha}, \boldsymbol{\lambda}) \lesa  \Big( \frac{\lambda_3}{\lambda_{1,2}}\Big)^{\epsilon} \mb{D}  \| \varphi_{\lambda_0} \|_{S_w} \| \psi_{\lambda_1} \|_S \| \psi_{\lambda_2} \|_{S_w} \| \psi_{\lambda_3} \|_S
        \end{equation}
under the stronger frequency constraints
        \begin{equation}\label{eq:constr-case1}
        \lambda_0 \gtrsim \lambda_1 \gtrsim \lambda_3,  \qquad  \lambda_2 \geq \lambda_3, \qquad \text{and} \qquad \lambda_1 \lambda_2 \approx \lambda_{1,2} \lambda_0
        \end{equation}
(note that $\lambda_{1,2} = \min\{\lambda_1, \lambda_2\} = \lambda_{med}$ if \eqref{eq:constr-case1} holds). In other words, in each pair we place the high frequency wave into the weaker space $S_w$. We now consider separately the four subcases
    $$ \alpha^2 \lr{\lambda_0}_m \lr{\lambda_1}_m \lesa \tilde{\alpha}^2 \lr{\lambda_2}_m \lr{\lambda_3}_m, \qquad \varphi_{\lambda_0}=C_{\gtrsim \alpha^2\lr{\lambda_1}_m}\varphi_{\lambda_0}, \qquad \psi_{\lambda_1}=C_{\gtrsim \alpha^2\lr{\lambda_1}_m}\psi_{\lambda_1}, \qquad \text{remaining case} $$
under the assumption that \eqref{eq:constr-case1} holds. \\

\noindent \underline{\textit{Case 1a)} :  $\alpha^2 \lr{\lambda_0}_m \lr{\lambda_1}_m \lesa \tilde{\alpha}^2 \lr{\lambda_2}_m \lr{\lambda_3}_m$.}

The idea here is to exploit the additional restriction in the sum over $\alpha$ and $\tilde{\alpha}$ to obtain the crucial high-low frequency gain. An application of Cauchy-Schwarz and
Corollary \ref{cor:bi unified} implies that
\begin{align*}
&B(\alpha, \tilde{\alpha}, \boldsymbol{\lambda})\\
&\lesa \bigg\| \sum_{\substack{\kappa_0, \kappa_1 \in \mc{C}_\alpha\\ \ma(\kappa_0, \kappa_1) + \frac{m}{\lr{\lambda_1}_m}\approx \alpha }} P_{\lesa \lambda_2} \big(\overline{\varphi}_{\lambda_0, \kappa_0} \psi_{\lambda_1, \kappa_1}\big)\bigg\|_{L^2_{t,x}} \bigg\|\sum_{\substack{\kappa_2, \kappa_3 \in \mc{C}_{\tilde{\alpha}}\\ \ma(\kappa_2, \kappa_3) + \frac{m}{\lr{\lambda_{3}}_m} \approx \tilde{\alpha} }} \overline{\psi}_{\lambda_2, \kappa_2} \psi_{\lambda_3, \kappa_3}\bigg\|_{L^2_{t,x}}\\
&\lesa \Big[ \big( \alpha \lr{\lambda_1}_m\big)^{\frac{1}{4}} \lr{\lambda_{1,2}}_m^{\frac{1}{4}} \lambda_{1,2}^{\frac{d-2}{2}} \Big( \frac{\lambda_1 \lr{\lambda_0}_m}{\lambda_0 \lr{\lambda_1}_m}\Big)^{\frac{1}{4}} \| \varphi_{\lambda_0}\|_{S_w} \| \psi_{\lambda_1} \|_{S} \Big] \times \Big[
\alpha^{\frac{1}{4}} \lr{\lambda_3}^{\frac{1}{2}} \lambda_3^{\frac{d-2}{2}} \Big( \frac{\lambda_3 \lr{\lambda_2}_m}{\lambda_2 \lr{\lambda_3}_m}\Big)^{\frac{1}{4}}\| \psi_{\lambda_2}\|_{S_w} \| \psi_{\lambda_3} \|_{S} \Big].
\end{align*}
The constraints \eqref{eq:constr-case1} imply that $\lr{\lambda_0}_m \lr{\lambda_{1,2}}_m \approx \lr{\lambda_1}_m \lr{\lambda_2}_m$. Hence summing up in $\alpha^2 \lesa \tilde{\alpha}^2 \frac{\lr{\lambda_2}_m \lr{\lambda_3}_m}{\lr{\lambda_0}_m\lr{\lambda_1}_m}$ and then in $\tilde\alpha\les 1$ yields
\begin{align*}
\sum_{\substack{\alpha, \tilde{\alpha} \in 2^{-\NN_0} \\ \alpha^2 \lesa \tilde{\alpha}^2 \frac{\lr{\lambda_2}_m \lr{\lambda_3}_m}{\lr{\lambda_0}_m\lr{\lambda_1}_m}} } B(\alpha, \tilde{\alpha}, \boldsymbol{\lambda})&\lesa \Big( \frac{\lr{\lambda_0}_m \lr{\lambda_3}_m}{\lr{\lambda_2}_m \lr{\lambda_1}_m}\Big)^{\frac{1}{8}} \Big( \frac{\lambda_1 \lr{\lambda_0}_m}{\lambda_0 \lr{\lambda_1}_m}\Big)^{\frac{1}{4}} \Big( \frac{\lambda_3 \lr{\lambda_2}_m}{\lambda_2 \lr{\lambda_2}_m}\Big)^{\frac{1}{4}} \mb{D}
 \| \varphi_{\lambda_0}\|_{S_w} \| \psi_{\lambda_1} \|_{S} \| \psi_{\lambda_2}\|_{S_w} \| \psi_{\lambda_3} \|_{S} \\
 &\lesa \Big( \frac{\lambda_3}{\lambda_{1,2}}\Big)^{\frac{1}{8}} \mb{D} \| \varphi_{\lambda_0}\|_{S_w} \| \psi_{\lambda_1} \|_{S} \| \psi_{\lambda_2}\|_{S_w} \| \psi_{\lambda_3} \|_{S}
 \end{align*}
 where we again used the constraints \eqref{eq:constr-case1}. Hence \eqref{eqn:thm nonlin:goal case1} follows in this case. \\

%
%
%
%
%

\noindent\underline{\textit{Case 1b)} : $\varphi_{\lambda_0}=C_{\gtrsim \alpha^2\lr{\lambda_1}_m}\varphi_{\lambda_0}$.}

This is the case where the pair $\overline{\varphi}_{\lambda_0} \psi_{\lambda_1}$ is `nonresonant', as the high frequency term is away from the cone/hyperboloid, and thus we get a large gain from the embedding
            $$ \| C_{\gtrsim \beta} \psi_{\lambda_1} \|_{L^{a'}_t L^2_x} \lesa \beta^{\frac{1}{a}  - 1} \| \psi_{\lambda_1} \|_{S_w}$$
via Theorem \ref{thm:bi high mod}. It is in this case where the freedom to take $a<2$ becomes crucial to gain a high-low gain. More precisely, as $\min\{\alpha \lambda_1, \lambda_{1,2}\} \les \lambda_{1,2}$, unpacking the constant in Theorem \ref{thm:bi high mod} gives
    \begin{align*}
        \alpha \big( \alpha^2 \lr{\lambda_1}_m \big)^{\frac{1}{a} - 1} \lambda_{1,2}^{\frac{d}{2} - \frac{2d}{d-1}(\frac{2}{a} - 1)} \big( \lambda_1^{\frac{2}{d-1}} \lr{\lambda_1}_m \big)^{\frac{2}{a} -1 }
               &= \alpha^{2(\frac{1}{a} - \frac{1}{2})} \lambda_{1,2}^{\frac{d-1}{2}} \lambda_1^{\frac{1}{2} - \frac{1}{a}} \Big( \frac{\lambda_{1, 2}}{\lambda_1}\Big)^{\frac{1}{2} - \frac{4d}{d-1}(\frac{1}{a} - \frac{1}{2})} \Big( \frac{\lambda_1}{\lr{\lambda_1}_m}\Big)^{\frac{1}{2} - 3 (\frac{1}{a} - \frac{1}{2})} \\
               &\lesa \alpha^{2(\frac{1}{a} - \frac{1}{2})} \lambda_{1,2}^{\frac{d-1}{2}} \lambda_1^{\frac{1}{2} - \frac{1}{a}}.
    \end{align*}
Thus recalling the constraints \eqref{eq:constr-case1}, an application of Theorem \ref{thm:bi high mod} and Corollary \ref{cor:bi unified} implies that
	\begin{align}
	&B(\alpha, \tilde{\alpha}, \boldsymbol{\lambda})\notag \\
	   &\lesa  \bigg\| \sum_{\substack{\kappa_0, \kappa_1 \in \mc{C}_\alpha\\ \ma(\kappa_0, \kappa_1) + \frac{m}{\lr{\lambda_1}_m}\approx \alpha}} P_{\lesa \lambda_2} \big( \overline{C_{\gtrsim \alpha^2 \lr{\lambda_1}_m}\varphi}_{\lambda_0, \kappa_0} \psi_{\lambda_1, \kappa_1}\big) \bigg\|_{L^a_t L^2_x} \bigg\| \sum_{\substack{\kappa_2, \kappa_3 \in \mc{C}_{\tilde{\alpha}}\\ \ma(\kappa_2, \kappa_3) + \frac{m}{\lr{\lambda_3}_m}\approx \tilde\alpha}}\overline{\psi}_{\lambda_2, \kappa_2} \psi_{\lambda_3, \kappa_3} \bigg\|_{L^{a'}_t L^2_x} \notag \\
	   &\lesa \Big[ \alpha^{2(\frac{1}{a} - \frac{1}{2})} \lambda_{1,2}^{\frac{d-1}{2}} \lambda_1^{\frac{1}{2} - \frac{1}{a}}   \| \varphi_{\lambda_0} \|_{S_w} \| \psi_{\lambda_1} \|_{S_w} \Big] \Big[ \tilde{\alpha}^{\frac{1}{4}} \lr{\lambda_3}_m^{\frac{1}{2}} \lambda_3^{\frac{d-2}{2} + \frac{1}{a} - \frac{1}{2}} \Big(\frac{\lambda_3 \lr{\lambda_2}_m}{\lambda_2 \lr{\lambda_3}_m}\Big)^{\frac{1}{4}} \| \psi_{\lambda_2} \|_{S_w} \| \psi_{\lambda_3}\|_{S}\Big] \notag \\
        &\lesa \alpha^{2(\frac{1}{a} - \frac{1}{2})} \tilde{\alpha}^{\frac{1}{4}} \Big( \frac{\lambda_3}{\lambda_{1}}\Big)^{\frac{1}{a} - \frac{1}{2}} \mb{D} \| \varphi_{\lambda_0} \|_{S_w} \| \psi_{\lambda_1} \|_{S_w} \| \psi_{\lambda_2} \|_{S_w} \| \psi_{\lambda_3}\|_{S}. \label{eqn:case1b est}
    \end{align}
Hence summing up over $\alpha, \tilde{\alpha} \in 2^{-\NN_0}$, we again obtain \eqref{eqn:thm nonlin:goal case1}. \\

\noindent\underline{\textit{Case 1c)} : $\psi_{\lambda_1}=C_{\gtrsim \alpha^2\lr{\lambda_1}_m}\psi_{\lambda_1}$.}

This case is similar to 1b) as the pair $\overline{\varphi}\psi$ is again non-resonant, and thus we can apply the high-modulation gain given by Theorem \ref{thm:bi high mod}. In more detail, recalling the constraints \eqref{eq:constr-case1} and noting that $d\g 2$, another application of Theorem \ref{thm:bi high mod} (using $(\min\{\alpha \lambda_1, \lambda_{1,2}\})^{\frac{d-1}{2}}\les (\alpha \lambda_1)^{\frac{1}{2}} \lambda_{1,2}^{\frac{d-2}{2}}$) and Corollary \ref{cor:bi unified} gives
	\begin{align}
	&B(\alpha, \tilde{\alpha}, \boldsymbol{\lambda})\notag \\
	&\lesa  \bigg\| \sum_{\substack{\kappa_0, \kappa_1 \in \mc{C}_\alpha\\ \ma(\kappa_0, \kappa_1) + \frac{m}{\lr{\lambda_1}_m}\approx \alpha}} P_{\lesa \lambda_2} \big( \overline{\varphi}_{\lambda_0, \kappa_0} C_{\gtrsim \alpha^2 \lr{\lambda_1}_m}\psi_{\lambda_1, \kappa_1}\big) \bigg\|_{L^a_t L^2_x} \bigg\| \sum_{\substack{\kappa_2, \kappa_3 \in \mc{C}_{\tilde{\alpha}}\\ \ma(\kappa_2, \kappa_3) + \frac{m}{\lr{\lambda_3}_m}\approx \tilde\alpha}}\overline{\psi}_{\lambda_2, \kappa_2} \psi_{\lambda_3, \kappa_3} \bigg\|_{L^{a'}_t L^2_x} \notag\\
		&\lesa \Big[ \alpha^{\frac{1}{2} - 2(\frac{1}{a} - \frac{1}{2})} \lambda_{1,2}^{\frac{d-1}{2}} \lambda_1^{\frac{1}{2}-\frac{1}{a}} \Big( \frac{\lambda_1}{\lr{\lambda_1}_m}\Big)^{\frac{1}{a} - \frac{1}{2}} \| \varphi_{\lambda_0} \|_{S_w} \| \psi_{\lambda_1} \|_{S} \Big] \Big[ \tilde{\alpha}^{\frac{1}{4}}\lambda_3^{\frac{d-1}{2} + \frac{1}{a} -\frac{1}{2}} \Big( \frac{\lambda_3 \lr{\lambda_2}_m}{\lambda_2 \lr{\lambda_3}_m}\Big)^{\frac{1}{4}} \| \psi_{\lambda_2} \|_{S_w} \| \psi_{\lambda_3}\|_{S}\Big]. \notag
    \end{align}
Hence after summing up over $\alpha, \tilde{\alpha} \in 2^{-\NN_0}$ the constraints \eqref{eq:constr-case1} imply that
    \begin{equation}
        \sum_{\alpha, \tilde{\alpha} \in 2^{-\NN_0}  } B(\alpha, \tilde{\alpha}, \boldsymbol{\lambda})
		\lesa  \Big( \frac{\lambda_3}{\lambda_1}\Big)^{\frac{1}{a} - \frac{1}{2}} \mb{D} \| \varphi_{\lambda_0} \|_{S_w} \| \psi_{\lambda_1} \|_{S} \| \psi_{\lambda_2} \|_{S_w} \| \psi_{\lambda_3}\|_{S} \label{eqn:case1c est}
    \end{equation}
which again clearly suffices to give \eqref{eqn:thm nonlin:goal case1}. \\

\noindent \underline{\textit{Case 1d)} : $\alpha^2 \lr{\lambda_0}_m \lr{\lambda_1}_m \gg \tilde{\alpha}^2 \lr{\lambda_2}_m \lr{\lambda_3}_m$ and $\varphi_{\lambda_0}=C_{\ll \alpha^2\lr{\lambda_1}_m}\varphi_{\lambda_0}$, $\psi_{\lambda_1}=C_{\ll \alpha^2\lr{\lambda_1}_m}\psi_{\lambda_1}$.}

This is the only remaining subcase of Case 1, where the product $\overline{\varphi}_{\lambda_0} \psi_{\lambda_1}$ may be fully resonant (thus $\varphi_{\lambda_0}$ and $\psi_{\lambda_1}$ have Fourier support close to the cone/hyperboloid). The key point here is that by considering the interaction of the various Fourier supports, we can show that $B$ vanishes unless at least one of $\psi_{\lambda_2}$ or $\psi_{\lambda_3}$ has Fourier support away from the cone/hyperboloid. In other words, the assumptions $\alpha^2 \lr{\lambda_0}_m \lr{\lambda_1}_m \gg \tilde{\alpha}^2 \lr{\lambda_2}_m \lr{\lambda_3}_m$ and $\overline{\varphi}_{\lambda_0}\psi_{\lambda_1}$ is resonant, implies that the product $\overline{\psi}_{\lambda_2} \psi_{\lambda_3}$ must be nonresonant. More precisely, an application of Lemma \ref{lem:Fourier supp} implies that
		$$ \supp \mc{F}_{t,x}\big[ \overline{\varphi}_{\lambda_0, \kappa_0} \psi_{\lambda_1, \kappa_1}] \subset \big\{ \big| |\tau|^2 - \lr{\xi}_m^2 \big| \approx \alpha^2 \lr{\lambda_0}_m \lr{\lambda_1}_m \}$$
and for any $\beta\gtrsim \tilde{\alpha}^2 \lr{\lambda_3}_m$
		$$ \supp \mc{F}_{t,x}\big[ \overline{C_{\les \beta}\psi}_{\lambda_2, \kappa_2} C_{\les \beta} \psi_{\lambda_3, \kappa_3}\big] \subset \big\{ \big| |\tau|^2 - \lr{\xi}_m^2\big| \lesa (\lr{\lambda_2}_m + \beta) \beta ) \big\}.   $$
In particular, taking
		$$ \beta = (\alpha^2 \lr{ \lambda_0}_m \lr{\lambda_1}_m )^\frac{1}{2} \min\Big\{1, \frac{(\alpha^2 \lr{\lambda_0}_m \lr{\lambda_1}_m)^\frac{1}{2}}{\lr{\lambda_2}_m}\Big\}$$
and noting that our assumption $\alpha^2 \lr{\lambda_0}_m \lr{\lambda_1}_m \gg \tilde{\alpha}^2 \lr{\lambda_2}_m \lr{\lambda_3}_m$ together with the constraint \eqref{eq:constr-case1} implies that $\beta \gtrsim \tilde{\alpha}^2 \lr{\lambda_3}_m$, we see that at least one of $\psi_{\lambda_2}$ or $\psi_{\lambda_3}$ must be at distance $\beta$ from the hyperboloid/cone as otherwise we trivially have $B=0$. Therefore, an application of Theorem \ref{thm:bi high mod}, Corollary \ref{cor:bi unified}, and the disposability bounds \eqref{eqn:disposability}, gives
\begin{align}
B(\alpha, \tilde{\alpha}, \boldsymbol{\lambda})
    &\lesa \Big\|\sum_{\substack{\kappa_0, \kappa_1 \in \mc{C}_\alpha\\ \ma(\kappa_0, \kappa_1) + \frac{m}{\lr{\lambda_1}_m}\approx \alpha}}
        P_{\lesa \lambda_2}( \overline{\varphi}_{\lambda_0, \kappa_0} \psi_{\lambda_1, \kappa_1})\Big\|_{L^2_{t,x}} \Bigg( \Big\| \sum_{\substack{\kappa_2, \kappa_3 \in \mc{C}_{\tilde{\alpha}}\\ \ma(\kappa_2, \kappa_3) + \frac{m}{\lr{\lambda_3}_m}\approx \tilde\alpha}}\overline{C_{\gtrsim \beta}\psi}_{\lambda_2, \kappa_2} \psi_{\lambda_3, \kappa_3}\Big\|_{L^2_{t,x}}  \notag \\
    &\qquad \qquad \qquad \qquad + \Big\| \sum_{\substack{\kappa_2, \kappa_3 \in \mc{C}_{\tilde{\alpha}}\\ \ma(\kappa_2, \kappa_3) + \frac{m}{\lr{\lambda_3}_m}\approx \tilde\alpha}}
        \overline{C_{\ll \beta}\psi}_{\lambda_2, \kappa_2} C_{\gtrsim \beta}\psi_{\lambda_3, \kappa_3}\Big\|_{L^2_{t,x}} \Bigg)\notag\\
    &\lesa (\alpha \lr{\lambda_1}_m)^{\frac{1}{4}} \lr{\lambda_{1,2}}_m^{\frac{1}{4}} \lambda_{1,2}^{\frac{d-2}{2}} \Big( \frac{\lambda_1 \lr{\lambda_0}_m}{\lambda_0 \lr{\lambda_1}_m}\Big)^{\frac{1}{4}} \tilde{\alpha}^{\frac{d-1}{2}} \lambda_3^{\frac{d-2}{2}} \lr{\lambda_3}_m^{\frac{1}{2}}\notag \\
     &\qquad  \times \Big[ \Big( \frac{\tilde{\alpha}^2 \lr{\lambda_3}_m}{\beta}\Big)^{1-\frac{1}{a}} \Big( \frac{\lambda_3}{\lr{\lambda_3}_m}\Big)^{2(1-\frac{1}{a})} +\Big( \frac{\tilde{\alpha}^2 \lr{\lambda_3}_m}{\beta}\Big)^{\frac{1}{2}} \frac{\lambda_3}{\lr{\lambda_3}_m} \Big]  \| \varphi_{\lambda_0}\|_{S_w} \| \psi_{\lambda_1} \|_{S} \|\psi_{\lambda_2}\|_{S_w} \|\psi_{\lambda_3}\|_{S} \notag \\
    &\lesa \sigma_{\alpha, \tilde{\alpha}, \boldsymbol{\lambda}} \mb{D} \| \varphi_{\lambda_0}\|_{S_w} \| \psi_{\lambda_1} \|_{S} \|\psi_{\lambda_2}\|_{S_w} \|\psi_{\lambda_3}\|_{S}
    \label{eqn:thm nonlinear:case1d}
\end{align}
where using \eqref{eq:constr-case1} and $\beta \gtrsim \tilde{\alpha}^2 \lr{\lambda_3}_m$ we can take
	$$\sigma_{\alpha, \tilde{\alpha}, \boldsymbol{\lambda}} = \tilde{\alpha}^{\frac{1}{2}} \Big( \frac{ \alpha\lr{\lambda_1}_m}{\lr{\lambda_{1,2}}_m}\Big)^{\frac{1}{4}}\Big( \frac{ \lambda_1 \lr{\lambda_0}_m}{\lambda_0 \lr{\lambda_1}_m}\Big)^{\frac{1}{4}} \Big( \frac{\tilde{\alpha}^2 \lr{\lambda_3}}{\beta}\Big)^{1-\frac{1}{a}} \Big( \frac{\lambda_3}{\lr{\lambda_3}_m}\Big)^{2(1-\frac{1}{a})}.  $$
We now split the sum over $\alpha \in 2^{-\NN_0}$ into the regions $\frac{\lr{\lambda_{1,2}}_m}{\lr{\lambda_1}_m} \ll \alpha \les 1$ and $\alpha \lesa \frac{\lr{\lambda_{1,2}}_m}{\lr{\lambda_1}_m}$. In the former region, the constraint \eqref{eq:constr-case1} forces $\beta \approx \alpha \lr{\lambda_1}_m$ and $\lambda_1 \gg \lr{\lambda_2}_m$ hence
                $$\sigma_{\alpha, \tilde{\alpha}, \boldsymbol{\lambda}} \approx   \tilde{\alpha}^{\frac{5}{2}- \frac{2}{a}} \Big( \frac{ \alpha \lr{\lambda_1}_m}{\lr{\lambda_{1,2}}_m}\Big)^{\frac{1}{a} - \frac{3}{4}} \Big( \frac{ \lr{\lambda_3}}{\lr{\lambda_{1,2}}_m}\Big)^{1-\frac{1}{a}} \Big( \frac{\lambda_3}{\lr{\lambda_3}_m}\Big)^{2(1-\frac{1}{a})} \lesa \tilde{\alpha}^{\frac{5}{2}- \frac{2}{a}} \Big( \frac{ \alpha \lr{\lambda_1}_m}{\lr{\lambda_{1,2}}_m}\Big)^{\frac{1}{a} - \frac{3}{4}} \Big( \frac{\lambda_3}{\lr{\lambda_{1,2}}_m}\Big)^{1-\frac{1}{a}} . $$
Consequently, provided $\frac{1}{a} - \frac{3}{4} < 0$, summing up \eqref{eqn:thm nonlinear:case1d} over $\alpha, \tilde{\alpha} \in 2^{\NN_0}$ with $\frac{\lr{\lambda_{1,2}}_m}{\lr{\lambda_1}_m} \ll \alpha \les 1$ we obtain \eqref{eqn:thm nonlin:goal case1}. For the remaining region $\alpha \lesa \frac{\lr{\lambda_{1,2}}_m}{\lr{\lambda_2}_m}$, in view of the constraints \eqref{eq:constr-case1} and the assumption $\alpha^2 \lr{\lambda_0}_m \lr{\lambda_1}_m \gtrsim \tilde{\alpha}^2 \lr{\lambda_2}_m \lr{\lambda_3}_m$ we have
            $$\beta \approx \frac{\alpha^2 \lr{\lambda_0}_m \lr{\lambda_1}_m}{\lr{\lambda_2}_m} \approx \frac{\alpha^2 \lr{\lambda_1}_m^2}{\lr{\lambda_{1,2}}_m} \gtrsim \tilde{\alpha}^2 \lr{\lambda_3}. $$
In particular, the sum over $\alpha \in \NN_0$ is restricted to the range $\tilde{\alpha} (\frac{\lr{\lambda_3}_m \lr{\lambda_{1,2}}_m}{\lr{\lambda_1}_m^2})^\frac{1}{2} \lesa \alpha \lesa \frac{\lr{\lambda_{1,2}}_m}{\lr{\lambda_1}_m}$, and we have the upper bound
        \begin{align*}
            \sigma_{\alpha, \tilde{\alpha}, \boldsymbol{\lambda}}
            &\approx \Big( \frac{\tilde{\alpha}^2 \lr{\lambda_3}_m \lr{\lambda_{1,2}}_m}{\alpha^2 \lr{\lambda_1}^2} \Big)^{\frac{7}{8} - \frac{1}{a}} \tilde{\alpha}^{\frac{3}{4}} \Big( \frac{ \lambda_1 \lr{\lambda_0}_m}{\lambda_0 \lr{\lambda_1}_m}\Big)^{\frac{1}{4}}  \Big( \frac{\lr{\lambda_3}_m}{\lr{\lambda_{1,2}}_m}\Big)^{\frac{1}{8}} \Big( \frac{\lambda_3}{\lr{\lambda_3}_m}\Big)^{2(1-\frac{1}{a})} \\
            &\lesa \Big( \frac{\tilde{\alpha}^2 \lr{\lambda_3}_m \lr{\lambda_{1,2}}_m}{\alpha^2 \lr{\lambda_1}^2} \Big)^{\frac{7}{8} - \frac{1}{a}} \tilde{\alpha}^{\frac{3}{4}}  \Big( \frac{\lambda_3}{\lr{\lambda_{1,2}}_m}\Big)^{\frac{1}{8}}
        \end{align*}
again provided $\frac{1}{a} - \frac{1}{2} \ll 1$. Therefore summing up \eqref{eqn:thm nonlinear:case1d} over $\alpha, \tilde{\alpha} \in 2^{\NN_0}$ with $\tilde{\alpha} (\frac{\lr{\lambda_3}_m \lr{\lambda_{1,2}}_m}{\lr{\lambda_1}_m^2})^\frac{1}{2} \lesa \alpha \lesa \frac{\lr{\lambda_{1,2}}_m}{\lr{\lambda_1}_m}$ we again obtain \eqref{eqn:thm nonlin:goal case1}.

\bigskip

\noindent \underline{\textbf{Case 2:} $\lambda_0 \ll \lambda_1$.}\\

This case follows by essentially repeating the argument in Case 1, thus we shall be somewhat brief. The key difference is now the low frequency term $\varphi_{\lambda_0}$ must be placed in $S_w$, which causes a $(\frac{\lr{\lambda_1}_m}{\lr{\lambda_0}_m})^\epsilon$ loss. However this is easily compensated by the large gain from the $(\frac{\lambda_0}{\lambda_1})^{\frac{1}{2}}$ factor arising from the inhomogeneous derivative term $(\frac{\lr{\lambda_1}_m \lr{\lambda_2}_m \lr{\lambda_3}_m}{\lr{\lambda_0}_m})^{\frac{1}{2}}$. We now turn to the details. As in Case 1, by conjugation symmetry, we may assume that $\lambda_2 \g \lambda_3$. Moreover, as the product $\overline{\varphi}_{\lambda_0}\psi_{\lambda_1}$ has spatial Fourier support concentrated near $|\xi| \approx \lambda_1$, we see that the quadrilinear term $B$ vanishes unless $\lambda_2 \gtrsim \lambda_1$. In particular, we may now assume the frequency restrictions
		\begin{equation}\label{eq:constr-case2}
            \lambda_0 \ll \lambda_1, \qquad \lambda_3 \les \lambda_2, \qquad \lambda_1 \lesa  \lambda_2.
        \end{equation}
To deal with the easy cases, we observe that applying H\"older's inequality together with  Corollary \ref{cor:bi unified} gives
	\begin{align*}
		B(\alpha, \tilde{\alpha}, \boldsymbol{\lambda})
					&\lesa \bigg\| \sum_{\substack{\kappa_0, \kappa_1 \in \mc{C}_\alpha\\ \ma(\kappa_0, \kappa_1) + \frac{m}{\lr{\lambda_0}_m}\approx \alpha }} \overline{\varphi}_{\lambda_0, \kappa_0} \psi_{\lambda_1, \kappa_1}\bigg\|_{L^2_{t,x}} \bigg\|\sum_{\substack{\kappa_2, \kappa_3 \in \mc{C}_{\tilde{\alpha}}\\ \ma(\kappa_2, \kappa_3) + \frac{m}{\lr{\lambda_{3}}_m} \approx \tilde{\alpha} }} \big(\overline{\psi}_{\lambda_2, \kappa_2} \psi_{\lambda_3, \kappa_3}\big)\bigg\|_{L^2_{t,x}}  \\
					&\lesa 	\Big[ \alpha^\frac{1}{4} \lambda_0^{\frac{d-2}{2} } \lr{\lambda_0}_m^{\frac{1}{2}} \Big( \frac{ \lambda_0 \lr{\lambda_1}_m}{\lambda_1 \lr{\lambda_0}_m}\Big)^{\frac{1}{4}}  \Big(\frac{\lr{\lambda_1}_m}{\lr{\lambda_0}_m}\Big)^{6d(\frac{1}{a} - \frac{1}{2})} \| \varphi_{\lambda_0}\|_{S_w} \| \psi_{\lambda_1} \|_{S_w} \Big] \\
&\qquad \qquad \times \Big[ \tilde{\alpha}^\frac{1}{4} \lambda_3^{\frac{d-2}{2} } \lr{\lambda_3}_m^{\frac{1}{2}} \Big( \frac{ \lambda_3 \lr{\lambda_2}_m}{\lambda_2 \lr{\lambda_3}_m}\Big)^{\frac{1}{4}}\| \psi_{\lambda_2} \|_S \| \psi_{\lambda_3} \|_{S} \Big] \\
				  &\lesa (\alpha \tilde{\alpha})^{\frac{1}{4}} \Big( \frac{\lr{\lambda_0}_m}{\lr{\lambda_1}_m}\Big)^{\frac{1}{4}} \Big( \frac{\lr{\lambda_0}_m}{\lr{\lambda_2}_m}\Big)^{\frac{1}{2}} \Big( \frac{\lambda_0^2}{\lambda_1 \lambda_2} \Big)^{\frac{d-2}{2}}  \Big( \frac{\lambda_0\lambda_3 \lr{\lambda_1}_m \lr{\lambda_2}_m}{\lambda_1 \lambda_2\lr{\lambda_0}_m \lr{\lambda_3}_m}\Big)^{\frac{1}{4}} \mb{D}  \| \varphi_{\lambda_0}\|_{S_w} \| \psi_{\lambda_1}\|_{S} \| \psi_{\lambda_2}\|_{S} \| \psi_{\lambda_3}\|_{S}
	\end{align*}
provided $\frac{1}{a} - \frac{1}{2} \ll 1$. In particular, in view of the constraints \eqref{eq:constr-case2}, if $\lambda_3 \gtrsim \lambda_0$ then a short computation gives
    $$ \sum_{\alpha, \tilde{\alpha} \in \NN_0} B(\alpha, \tilde{\alpha}, \boldsymbol{\lambda}) \lesa \Big( \frac{\lambda_0}{\lambda_{med}}\Big)^{d-2} \mb{D} \| \varphi_{\lambda_0}\|_{S_w} \| \psi_{\lambda_1}\|_{S} \| \psi_{\lambda_2}\|_{S} \| \psi_{\lambda_3}\|_{S} $$
which clearly suffices. On the other hand, if $\lambda_3 \ll \lambda_0$, then \eqref{eq:constr-case2} implies that $B$ vanishes unless $\lambda_1 \approx \lambda_2 \gg \lambda_0 \gg \lambda_3$. Moreover, as in Case 1a), we can exploit the sum over $\alpha$ to get a low-high gain if we in addition have $\alpha^2 \lr{\lambda_0}_m  \lesa \tilde{\alpha}^2 \lr{\lambda_3}_m$.  Consequently we have reduced matters to considering the high-low $\times$ high-low case
        \begin{equation}\label{eq:constr-case2mod}
            \lambda_3 \ll \lambda_0 \ll \lambda_1\approx \lambda_2, \qquad \alpha^2 \lr{\lambda_0}_m \gg \tilde{\alpha}^2 \lr{\lambda_3}_m.
        \end{equation}
We now consider separately the high-modulation cases, and the remaining small-modulation cases under the assumption that \eqref{eq:constr-case2mod} holds. \\

\noindent\underline{\textit{Case 2a)} :  $\varphi_{\lambda_0}=C_{\gtrsim \alpha^2\lr{\lambda_0}_m}\varphi_{\lambda_0}$ or $\psi_{\lambda_1} = C_{\gtrsim \alpha^2 \lr{\lambda_0}_m} \psi_{\lambda_1}$.}

This case is essentially identical to Case 1b) and 1c). We begin noting that
     \begin{align*}
        \alpha (\alpha^2 \lr{\lambda_0})^{\frac{1}{a} -1} \big( \alpha^{d-1} &\lambda_0^d\big)^{\frac{1}{2} - \frac{4}{d-1}(\frac{1}{a} - \frac{1}{2})} \lr{\lambda_1}^{\frac{2(d+1)}{d-1}(\frac{1}{a}-\frac{1}{2})} \\
        &= \alpha^{\frac{d-1}{2} - 2(\frac{1}{a} - \frac{1}{2})} \Big( \frac{\lr{\lambda_1}_m}{\lr{\lambda_0}_m}\Big)^{\frac{2(d+1)}{d-1}(\frac{1}{a} - \frac{1}{2})} \Big( \frac{\lambda_0}{\lr{\lambda_0}_m}\Big)^{1 - \frac{3d +1}{d-1}(\frac{1}{a} - \frac{1}{2})} \lambda_0^{\frac{d-1}{2} - \frac{1}{a}} \lr{\lambda_0}^{\frac{1}{2}}.
     \end{align*}
Hence, provided $0<\frac{1}{a} - \frac{1}{2} \ll 1$,  as either $\varphi_{\lambda_0}$ or $\psi_{\lambda_1}$ has modulation at least $\alpha^2 \lr{\lambda_0}_m$, an application of Theorem \ref{thm:bi high mod} (placing both $\varphi_{\lambda_0}$ and $\psi_{\lambda_1}$ in  $S_w$ and recalling the constraints \eqref{eq:constr-case2mod}) together with  Corollary \ref{cor:bi unified} gives
	\begin{align*}
	B(\alpha, \tilde{\alpha}, \boldsymbol{\lambda})
	&\lesa  \bigg\| \sum_{\substack{\kappa_0, \kappa_1 \in \mc{C}_\alpha\\ \ma(\kappa_0, \kappa_1) + \frac{m}{\lr{\lambda_0}_m}\approx \alpha}}  \overline{\varphi}_{\lambda_0, \kappa_0} \psi_{\lambda_1, \kappa_1} \bigg\|_{L^a_t L^2_x} \bigg\| \sum_{\substack{\kappa_2, \kappa_3 \in \mc{C}_{\tilde{\alpha}}\\ \ma(\kappa_2, \kappa_3) + \frac{m}{\lr{\lambda_3}_m}\approx \tilde\alpha}}\overline{\psi}_{\lambda_2, \kappa_2} \psi_{\lambda_3, \kappa_3} \bigg\|_{L^{a'}_t L^2_x} \notag \\
		&\lesa \Big[ \alpha^{\frac{1}{4}} \Big(\frac{\lr{\lambda_1}_m}{\lr{\lambda_0}_m}\Big)^{\frac{1}{4}} \lambda_0^{\frac{d}{2} - \frac{1}{a}} \| \varphi_{\lambda_0} \|_{S_w}
 \| \psi_{\lambda_1} \|_{S_w} \Big] \Big[ \tilde{\alpha}^{\frac{1}{4}} \lambda_3^{\frac{d-1}{2} +\frac{1}{a} - \frac{1}{2}}  \lr{\lambda_3}_m^{\frac{1}{2}} \| \psi_{\lambda_2} \|_{S} \| \psi_{\lambda_3}\|_{S}\Big] \notag \\
		&\lesa (\alpha \tilde{\alpha})^{\frac{1}{4}}  \Big( \frac{\lambda_3}{\lambda_0}\Big)^{\frac{1}{a} - \frac{1}{2}} \Big( \frac{\lr{\lambda_0}_m}{\lr{\lambda_1}_m}\Big)^{\frac{3}{4}} \Big( \frac{\lambda_0}{\lambda_1}\Big)^{d-2} \Big( \frac{\lambda_0}{\lr{\lambda_0}_m}\Big)^{\frac{1}{2}} \mb{D}  \| \varphi_{\lambda_0} \|_{S_w} \| \psi_{\lambda_1} \|_{S_w} \| \psi_{\lambda_2} \|_{S_w} \| \psi_{\lambda_3}\|_{S}.
	\end{align*}
Summing up over $\alpha, \tilde{\alpha} \in 2^{-\NN_0}$ we clearly obtain \eqref{eqn:thm nonlin:goal} in this case.\\

\noindent \underline{\textit{Case 2b)} :  $\alpha^2 \lr{\lambda_0}_m \gg \tilde{\alpha}^2 \lr{\lambda_3}_m$,  and $\varphi_{\lambda_0}=C_{\ll \alpha^2\lr{\lambda_0}_m}\varphi_{\lambda_0}$, $\psi_{\lambda_1}=C_{\ll \alpha^2\lr{\lambda_0}_m}\psi_{\lambda_1}$.}

This case is similar to Case 1d), and contains the only remaining subcase of Case 2, where the product $\overline{\varphi}\psi$ may be fully resonant. As in Case 1d), the point is to exploit the fact that Lemma \ref{lem:Fourier supp} implies that the quadrilinear product $B$ vanishes unless one of $\psi_{\lambda_2}$ or $\psi_{\lambda_3}$ has large modulation. The only difference is that as we are now placing the low frequency term into the weak space $S_w$ we have a slight high-low loss. However, as above, this loss is easily compensated by the large gain coming from the derivative factors. In more detail,  an application of Lemma \ref{lem:Fourier supp} implies that
		$$ \supp \mc{F}_{t,x}\big[ \overline{\varphi}_{\lambda_0, \kappa_0} \psi_{\lambda_1, \kappa_1}] \subset \big\{ \big| |\tau|^2 - \lr{\xi}_m^2 \big| \approx \alpha^2 \lr{\lambda_0}_m \lr{\lambda_1}_m \}$$
and (since $\alpha^2 \lr{\lambda_0}_m \gtrsim \tilde{\alpha}^2 \lr{\lambda_3}_m$ and $\alpha^2 \lr{\lambda_0}_m  \lesa \lr{\lambda_1}_m \approx \lr{\lambda_2}_m$)
		$$ \supp \mc{F}_{t,x}\big[ \overline{C_{\ll \alpha^2 \lr{\lambda_0}_m}\psi}_{\lambda_2, \kappa_2} C_{\ll \alpha^2 \lr{\lambda_0}_m} \psi_{\lambda_3, \kappa_3}\big] \subset \big\{ \big| |\tau|^2 - \lr{\xi}_m^2\big| \ll \alpha^2 \lr{\lambda_0}_m \lr{\lambda_1}_m \big\}.   $$
and hence we see that at least one of $\psi_{\lambda_2}$ or $\psi_{\lambda_3}$ must be at distance $\alpha^2\lr{\lambda_0}_m$ from the hyperboloid/cone as otherwise we trivially have $B=0$. Therefore an application of Theorem \ref{thm:bi high mod} and Corollary \ref{cor:bi unified} gives
\begin{align*}
&B(\alpha, \tilde{\alpha}, \boldsymbol{\lambda})\\
    &\lesa \Big\|\sum_{\substack{\kappa_0, \kappa_1 \in \mc{C}_\alpha\\ \ma(\kappa_0, \kappa_1) + \frac{m}{\lr{\lambda_0}_m}\approx \alpha}}
        \overline{\varphi}_{\lambda_0, \kappa_0} \psi_{\lambda_1, \kappa_1}\Big\|_{L^2_{t,x}} \Bigg( \Big\| \sum_{\substack{\kappa_2, \kappa_3 \in \mc{C}_{\tilde{\alpha}}\\ \ma(\kappa_2, \kappa_3) + \frac{m}{\lr{\lambda_3}_m}\approx \tilde\alpha}}\overline{C_{\gtrsim \alpha^2 \lr{\lambda_0}_m}\psi}_{\lambda_2, \kappa_2} \psi_{\lambda_3, \kappa_3}\Big\|_{L^2_{t,x}}  \\
    &\qquad \qquad \qquad \qquad + \Big\| \sum_{\substack{\kappa_2, \kappa_3 \in \mc{C}_{\tilde{\alpha}}\\ \ma(\kappa_2, \kappa_3) + \frac{m}{\lr{\lambda_3}_m}\approx \tilde\alpha}}
        \overline{C_{\ll \alpha^2 \lr{\lambda_0}_m}\psi}_{\lambda_2, \kappa_2} C_{\gtrsim \alpha^2 \lr{\lambda_0}_m}\psi_{\lambda_3, \kappa_3}\Big\|_{L^2_{t,x}} \Bigg)\\
    &\lesa \Big[ \alpha^\frac{1}{4} \lr{\lambda_0}_m^{\frac{1}{2}} \lambda_0^{\frac{d-2}{2}} \Big( \frac{\lambda_0 \lr{\lambda_1}_m}{\lambda_1 \lr{\lambda_0}_m}\Big)^{\frac{1}{4}}  \Big(\frac{\lr{\lambda_1}_m}{\lr{\lambda_0}_m}\Big)^{6d(\frac{1}{a} - \frac{1}{2})} \| \varphi_{\lambda_0}\|_{S_w} \| \psi_{\lambda_1} \|_{S_w} \Big]
        \times\Big[ \tilde{\alpha}^{\frac{d+1}{2}} \lambda_3^d (\alpha^2 \lr{\lambda_0}_m)^{-\frac{1}{2}} \|\psi_{\lambda_2}\|_{S} \|\psi_{\lambda_3}\|_{S}\Big]\\
    &\approx \Big( \frac{\tilde{\alpha}^2 \lr{\lambda_3}_m}{\alpha^2 \lr{\lambda_0}_m}\Big)^{\frac{3}{8}} \Big( \frac{\lr{\lambda_0}_m}{\lr{\lambda_1}_m}\Big)^{\frac{3}{4} - 6d(\frac{1}{a} - \frac{1}{2})} \Big( \frac{\lambda_0}{\lambda_1} \Big)^{d-2+\frac{1}{4}} \Big( \frac{\lambda_3}{\lr{\lambda_0}_m}\Big)^{\frac{1}{8}} \Big( \frac{\lambda_3}{\lr{\lambda_3}_m}\Big)^{\frac{7}{8}} \mb{D} \| \varphi_{\lambda_0}\|_{S_w} \| \psi_{\lambda_1} \|_{S} \|\psi_{\lambda_2}\|_{S} \|\psi_{\lambda_3}\|_{S}.
\end{align*}
Summing up over $\alpha, \tilde{\alpha} \in 2^{-\NN_0}$ with $\alpha \gtrsim (\frac{\lr{\lambda_3}_m}{\lr{\lambda_0}_m})^{\frac{1}{2}} \tilde{\alpha}$ then gives \eqref{eqn:thm nonlin:goal} in this case. This completes the proof of Case 2.
\end{proof}

\begin{remark}\label{rem:improving bound}
A natural question is if the conclusion of Theorem \ref{thm:nonlin} can be improved. For instance, motivated by the energy dispersed arguments in \cite{Sterbenz2010}, if we aim to prove a large data theory for the cubic Dirac equation, a first step would be to improve the bound in Theorem \ref{thm:nonlin} to something like
        \begin{equation}\label{eqn:tri disp gain}
        \begin{split}
            \Big| \int_{\RR^{1+d}} \overline{\varphi}_{\lambda_0} &\psi_{\lambda_1} \overline{\psi}_{\lambda_2} \psi_{\lambda_3} dtdx \Big|\\
                        &\lesa \Big( \frac{\lambda_{min}}{\lambda_{max}}\Big)^{\epsilon} \mb{D}\| \varphi_{\lambda_0} \|_{S_w} \Big( \| \psi_{\lambda_1} \|_S \| \psi_{\lambda_2} \|_S \| \psi_{\lambda_3} \|_S\Big)^{1-\theta} \Big( \|\psi_{\lambda_1}\|_D \|\psi_{\lambda_2}\|_D \|\psi_{\lambda_3}\|_D \Big)^{\theta}
        \end{split}
        \end{equation}
where $0<\theta<1$ and $\| \cdot \|_D$ is a dispersive norm (for instance a (non-endpoint) Strichartz norm $L^q_t L^r_x$, or a norm measuring energy dispersion like $L^\infty_t \dot{B}^{-\frac{d}{2}}_{\infty, \infty}$). A bound of the form \eqref{eqn:tri disp gain} would essentially be necessary in order to obtain a nonlinear profile decomposition. However such a bound seems extremely difficult to prove, due to the fact it would essentially require (say by taking $\varphi_{\lambda_0} = \psi_{\lambda_2}$, $\psi_{\lambda_1} = \psi_{\lambda_3}$, and $\lambda_0 \gg \lambda_1$) improving the bilinear $L^2_{t,x}$ estimate
            $$ \| \overline{\psi}_{\lambda} \psi_\mu \|_{L^2_{t,x}} \lesa \mu^{\frac{d-2}{2}} \lr{\mu}_m^{\frac{1}{2}} \|\psi_{\lambda} \|_{S} \| \psi_\mu \|_{S}^{1-\theta} \| \psi_\mu \|_D^{\theta}$$
when $\lambda \gg \mu$. Even in the case where $\psi_{\lambda}$ and $\psi_\mu$ are both free waves, it is not at all clear that such an improvement is possible.
\end{remark}
\section{Proof of Theorem  \ref{thm:gwp}}\label{sec:proof-gwp}
 We use the standard approach via the contraction mapping principle. First, we consider the Duhamel integral
   $$ \mc{N}(\psi_1, \psi_2, \psi_3)(t) = \ind_{[0, \infty)}(t) \int_0^t \mc{U}_m(t-t')\gamma^0F(\psi_1,\psi_2,\psi_3) (t') dt'.$$
As in Section \ref{sec:mult-est}, for the sake of the exposition, we consider the Soler  nonlinearity $$F(\psi_1,\psi_2,\psi_3)=\overline{\psi}_1\psi_2\psi_3,$$ but the results are valid for cubic nonlinearities satisfying the null-structure as discussed in Section \ref{sec:setup}.

$ \mc{N}$ is well-defined and satisfies the following crucial estimate. Recall that $s_d=\frac{d-1}{2}$ and $\sigma_d=\frac{d-2}{2}$, and also that the nesting property
$S^{s,\sigma}_m\subset S^{s_d,\sigma_d}_m $ for $s\g s_d $ and $\sigma \les \sigma_d$ holds  only  if $m\ne 0$.
   \begin{theorem}\label{thm:est-n}
    Let $s\g s_d $. There exists $C_s>0$, such that for any $0\les \sigma \les \sigma_d$, $m \in \R$ and  functions $\psi_j \in  S^{s,\sigma}_m\cap  S^{s_d,\sigma_d}_m $ we have $\mc{N}(\psi_1, \psi_2, \psi_3)\in S^{s,\sigma}_m\cap C([0,\infty),H^{s,\sigma}_m(\R^d))$ and
    \begin{equation}\label{eq:est-n}
      \begin{split}
    & \| \mc{N}(\psi_1, \psi_2, \psi_3)\|_{S^{s,\sigma}_m}
     \les C_s  \sum_{j=1}^3\| \psi_j \|_{S^{s,\sigma}_m} \prod_{\substack{k=1,2,3\\ k\ne j}}\| \psi_{k} \|_{S^{s_d,\sigma_d}_m} \end{split}
\end{equation}
  \end{theorem}
  \begin{proof} Within the following proof, we use the convention that implicit constants may depend on $s\g s_d$, but not on $\sigma$.
  Clearly, $P_{\lambda_0}F(\psi_1, \psi_2, \psi_3)\in L^1_{t,loc}L^2_x$ for any $\lambda_0\in 2^\Z$, so
 Lemma \ref{lem:energy ineq} implies that $\mc{N}(\psi_1, \psi_2, \psi_3)\in S^{s,\sigma}_m\cap C([0,\infty),H^{s,\sigma}_m(\R^d))$ and
        \begin{equation}\label{eqn:global-energy ineq}
            \| \mc{N}(\psi_1, \psi_2, \psi_3) \|_{S^{s,\sigma}_m} \approx \sup_{\substack{\varphi \in C_0^\infty(\RR^{1+d}) \\ \| \varphi \|_{S^{-s,-\sigma}_{w, m}} \les 1}}   \Big| \int_0^\infty \int_{\RR^d} \overline{\varphi} \psi_3 \overline{\psi}_1\psi_2 dtdx \Big|,
        \end{equation}
        provided that the right hand side is finite. Let $ \varphi \in S^{-s,-\sigma}_{w,m}$. With a slight abuse of notation we dyadically decompose
        $$\varphi=\sum_{\lambda_0 \in 2^\Z} \varphi_{\lambda_0}, \quad \psi_j=\sum_{\lambda_j \in 2^\Z} \psi_{\lambda_j}, \; (j=1,2,3).$$
        By symmetry it suffices to prove the bound
        \begin{equation}\label{eq:sum-bd}
        \begin{split}
         \sum_{\Lambda}  \Big| \int_0^\infty \int_{\RR^d} \overline{\varphi}_{\lambda_0} \psi_{\lambda_3} \overline{\psi}_{\lambda_1}\psi_{\lambda_2} dtdx \Big| \lesa
        \| \varphi \|_{S^{-s,-\sigma}_{w, m}}  \sum_{j=1}^3\| \psi_j \|_{S^{s,\sigma}_m} \prod_{\substack{k=1,2,3\\ k\ne j}}\| \psi_{k} \|_{S^{s_d,\sigma_d}_m} ,
        \end{split}
        \end{equation}
        where $\Lambda$ is the set of all $(\lambda_0,\lambda_1,\lambda_2,\lambda_3)\in (2^\Z)^4$ satisfying $\lambda_1\g\lambda_2\g\lambda_3$.
        By the Littlewood-Paley trichotomy there is a nontrivial contribution only in the subset $\Lambda_I$ where $\lambda_0\sim \lambda_1$ and the subset $\Lambda_{II}$ where $\lambda_0\ll \lambda_1\sim \lambda_2$.
        We further decompose $\Lambda_{II}$ into the subset $\Lambda_{IIa}$ where $|m|\lesa \lambda_0$, the subset $\Lambda_{IIb}$ where $\lambda_3\lesa \lambda_0\ll |m|$, and the subset $\Lambda_{IIc}$ where  $\lambda_0\ll \min\{|m|,\lambda_3\}$.

     \textit{   Contribution of $\Lambda_I$:} Using $\lambda_0\sim \lambda_1\g\lambda_2\g\lambda_3$, the frequency localized estimate of Theorem \ref{thm:nonlin} yields
     \begin{align*}
     \sum_{\Lambda_I}  \Big| \int_0^\infty \int_{\RR^d} \overline{\varphi}_{\lambda_0} \psi_{\lambda_3} \overline{\psi}_{\lambda_1}\psi_{\lambda_2} dtdx \Big| \lesa{}& \sum_{\Lambda_I}  \Big(\frac{\lambda_3}{\lambda_2}\Big)^\epsilon \| \varphi_{\lambda_0} \|_{S^{-s_d,-\sigma_d}_{w, m}}  \| \psi_{\lambda_1} \|_{S^{s_d,\sigma_d}_m} \| \psi_{\lambda_2} \|_{S^{s_d,\sigma_d}_m} \| \psi_{\lambda_3} \|_{S^{s_d,\sigma_d}_m}\\
     \approx{}& \sum_{\Lambda_I}  \Big(\frac{\lambda_3}{\lambda_2}\Big)^\epsilon \| \varphi_{\lambda_0} \|_{S^{-s,-\sigma}_{w, m}}  \| \psi_{\lambda_1} \|_{S^{s,\sigma}_m} \| \psi_{\lambda_2} \|_{S^{s_d,\sigma_d}_m} \| \psi_{\lambda_3} \|_{S^{s_d,\sigma_d}_m}\\
   \lesa{}&  \sum_{\lambda_0\sim \lambda_1} \| \varphi_{\lambda_0} \|_{S^{-s,-\sigma}_{w, m}}  \| \psi_{\lambda_1} \|_{S^{s,\sigma}_m}  \sum_{\lambda_3\les \lambda_2} \Big(\frac{\lambda_3}{\lambda_2}\Big)^\epsilon \| \psi_{\lambda_2} \|_{S^{s_d,\sigma_d}_m} \| \psi_{\lambda_3} \|_{S^{s_d,\sigma_d}_m}\\
 \lesa{} & \| \varphi \|_{S^{-s,-\sigma}_{w, m}} \| \psi_1 \|_{S^{s,\sigma}_m} \| \psi_{2} \|_{S^{s_d,\sigma_d}_m} \| \psi_{3} \|_{S^{s_d,\sigma_d}_m},
     \end{align*}
        where we used Cauchy-Schwarz to perform the dyadic summation in the last step.

  \textit{Contribution of $\Lambda_{IIa}$:}        In this case, since $|m|\lesa \lambda_0\ll \lambda_1$, we have $$\lambda_0^{\sigma-\sigma_d}\lr{\lambda_0}_m^{s-\sigma-(s_d-\sigma_d)}\lesa \lambda_1^{\sigma-\sigma_d}\lr{\lambda_1}^{s-\sigma-(s_d-\sigma_d)}_m.$$ Therefore,
    Theorem \ref{thm:nonlin} yields
     \begin{align*}
     \sum_{\Lambda_{IIa}}  \Big| \int_0^\infty \int_{\RR^d} \overline{\varphi}_{\lambda_0} \psi_{\lambda_3} \overline{\psi}_{\lambda_1}\psi_{\lambda_2} dtdx \Big| \lesa{}& \sum_{\Lambda_{IIa}}  \Big(\frac{\lambda_{min}}{\lambda_{med}}\Big)^\epsilon \| \varphi_{\lambda_0} \|_{S^{-s_d,-\sigma_d}_{w, m}}  \| \psi_{\lambda_1} \|_{S^{s_d,\sigma_d}_m} \| \psi_{\lambda_2} \|_{S^{s_d,\sigma_d}_m} \| \psi_{\lambda_3} \|_{S^{s_d,\sigma_d}_m}\\
     \lesa{}& \sum_{\Lambda_{IIa}}  \Big(\frac{\lambda_{min}}{\lambda_{med}}\Big)^\epsilon  \| \varphi_{\lambda_0} \|_{S^{-s,-\sigma}_{w, m}}  \| \psi_{\lambda_1} \|_{S^{s,\sigma}_m} \| \psi_{\lambda_2} \|_{S^{s_d,\sigma_d}_m} \| \psi_{\lambda_3} \|_{S^{s_d,\sigma_d}_m}\\
      \lesa{} & \| \varphi \|_{S^{-s,-\sigma}_{w, m}} \| \psi_1 \|_{S^{s,\sigma}_m} \| \psi_{2} \|_{S^{s_d,\sigma_d}_m} \| \psi_{3} \|_{S^{s_d,\sigma_d}_m},
      \end{align*}
 where again we used Cauchy-Schwarz to perform the dyadic summation in the last step.

         \textit{Contribution of $\Lambda_{IIb}$:} Here,  $\lambda_3\lesa \lambda_0\ll |m|$ implies that
                 $$\lambda_0^{\sigma-\sigma_d}\lr{\lambda_0}_m^{s-\sigma-(s_d-\sigma_d)}\lesa \lambda_3^{\sigma-\sigma_d}\lr{\lambda_3}_m^{s-\sigma-(s_d-\sigma_d)}  $$ and  Theorem \ref{thm:nonlin} yields
\begin{align*}
     \sum_{\Lambda_{IIb}}  \Big| \int_0^\infty \int_{\RR^d} \overline{\varphi}_{\lambda_0} \psi_{\lambda_3} \overline{\psi}_{\lambda_1}\psi_{\lambda_2} dtdx \Big| \lesa{}& \sum_{\Lambda_{IIb}}  \Big(\frac{\lambda_{3}}{\lambda_{0}}\Big)^\epsilon \| \varphi_{\lambda_0} \|_{S^{-s_d,-\sigma_d}_{w, m}}  \| \psi_{\lambda_1} \|_{S^{s_d,\sigma_d}_m} \| \psi_{\lambda_2} \|_{S^{s_d,\sigma_d}_m} \| \psi_{\lambda_3} \|_{S^{s_d,\sigma_d}_m}\\
     \lesa{}& \sum_{\Lambda_{IIb}}  \Big(\frac{\lambda_{3}}{\lambda_{0}}\Big)^\epsilon  \| \varphi_{\lambda_0} \|_{S^{-s,-\sigma}_{w, m}}  \| \psi_{\lambda_1} \|_{S^{s_d,\sigma_d}_m} \| \psi_{\lambda_2} \|_{S^{s_d,\sigma_d}_m} \| \psi_{\lambda_3} \|_{S^{s,\sigma}_m}\\
      \lesa{} & \| \varphi \|_{S^{-s,-\sigma}_{w, m}} \| \psi_1 \|_{S^{s_d,\sigma_d}_m} \| \psi_{2} \|_{S^{s_d,\sigma_d}_m} \| \psi_{3} \|_{S^{s,\sigma}_m}.
      \end{align*}

\textit{Contribution of $\Lambda_{IIc}$:} Here,  $\lambda_0\ll \min\{\lambda_3,|m|\}$, therefore  $\epsilon >\sigma_d$ in Theorem \ref{thm:nonlin}.
                 $$\lambda_0^{\sigma-\sigma_d}\lr{\lambda_0}_m^{s-\sigma-(s_d-\sigma_d)}\lesa \Big(\frac{\lambda_{3}}{\lambda_{0}}\Big)^{\sigma_d-\sigma}\lambda_3^{\sigma-\sigma_d}\lr{\lambda_3}_m^{s-\sigma-(s_d-\sigma_d)}  $$ and  Theorem \ref{thm:nonlin} yields
\begin{align*}
     \sum_{\Lambda_{IIc}}  \Big| \int_0^\infty \int_{\RR^d} \overline{\varphi}_{\lambda_0} \psi_{\lambda_3} \overline{\psi}_{\lambda_1}\psi_{\lambda_2} dtdx \Big| \lesa{}& \sum_{\Lambda_{IIc}}  \Big(\frac{\lambda_{0}}{\lambda_{3}}\Big)^\epsilon \| \varphi_{\lambda_0} \|_{S^{-s_d,-\sigma_d}_{w, m}}  \| \psi_{\lambda_1} \|_{S^{s_d,\sigma_d}_m} \| \psi_{\lambda_2} \|_{S^{s_d,\sigma_d}_m} \| \psi_{\lambda_3} \|_{S^{s_d,\sigma_d}_m}\\
     \lesa{}& \sum_{\Lambda_{IIc}}  \Big(\frac{\lambda_{0}}{\lambda_{3}}\Big)^{\epsilon-\sigma_d+\sigma}  \| \varphi_{\lambda_0} \|_{S^{-s,-\sigma}_{w, m}}  \| \psi_{\lambda_1} \|_{S^{s_d,\sigma_d}_m} \| \psi_{\lambda_2} \|_{S^{s_d,\sigma_d}_m} \| \psi_{\lambda_3} \|_{S^{s,\sigma}_m}\\
      \lesa{} & \| \varphi \|_{S^{-s,-\sigma}_{w, m}} \| \psi_1 \|_{S^{s_d,\sigma_d}_m} \| \psi_{2} \|_{S^{s_d,\sigma_d}_m} \| \psi_{3} \|_{S^{s,\sigma}_m}.
      \end{align*}
      This concludes the proof of  \eqref{eq:sum-bd} and Lemma \ref{lem:energy ineq} implies the claimed estimate \eqref{eqn:global-energy ineq}.
    \end{proof}

    Now, for $s\g s_d$ we conclude from Lemma \ref{lem:energy ineq} and Theorem \ref{thm:est-n} that there exists a constant $C_s>0$ such that we have the three estimates
        \begin{align*}
          \|\ind_{[0, \infty)} \mc{U}_m f\|_{S^{s,\sigma}_m}\les{}& C_s \|f \|_{H^{s,\sigma}_m}\\
          \|\mc{N}(\psi)\|_{S^{s,\sigma}_m}\les{}& C_s \|\psi\|_{S^{s_d,\sigma_d}_m}^2\|\psi\|_{S^{s,\sigma}_m}\\
          \|\mc{N}(\psi)-\mc{N}(\varphi)\|_{S^{s,\sigma}_m}\les{}& C_s \big(\|\psi\|_{S^{s,\sigma}_m}+|\varphi\|_{S^{s,\sigma}_m}\big)\big(\|\psi\|_{S^{s_d,\sigma_d}_m}+|\varphi\|_{S^{s_d,\sigma_d}_m}\big)\|\psi-\varphi\|_{S^{s_d,\sigma_d}_m}\\& {}+ C_s \big(\|\psi\|_{S^{s_d,\sigma_d}_m}+|\varphi\|_{S^{s_d,\sigma_d}_m}\big)^2\|\psi-\varphi\|_{S^{s,\sigma}_m},
        \end{align*}
        for any $\sigma_d \g \sigma \g 0$ and $m \in \R$. Set $C:=C_{s_d}$ and we may assume that $C_s\g C$.

        Now, let $f\in H^{s_d,\sigma_d}_m\cap H^{s,\sigma}_m$ with $\|f \|_{H^{s_d,\sigma_d}_m}\les \delta$ and $\|f \|_{H^{s,\sigma}_m}\les \delta_{s,\sigma}$.
    We will show that, for small enough $\delta>0$ (and arbitrary $\delta_{s,\sigma}$), there exists $\psi \in S^{s_d,\sigma_d}_m\cap S^{s,\sigma}_m$ such that
        $$ \psi= \ind_{[0, \infty)} \mc{U}_m f + i \mc{N}(\psi, \psi, \psi)=:\Omega(\psi).$$
        This will be our notion of solution of \eqref{eq:cd} on the time interval $[0,\infty)$.

       Consider the complete metric space
        $$X_r:=\{\psi \in S^{s_d,\sigma_d}_m\cap S^{s,\sigma}_m: d(\psi,0)\les r\},$$
        with metric $$d(\psi,\varphi)=\delta^{-1}\|\psi-\varphi\|_{S^{s_d,\sigma_d}_m}+\delta_{s,\sigma}^{-1}\|\psi-\varphi\|_{S^{s,\sigma}_m}.$$
        For $\psi \in X_r$ we have
        \begin{align*}
          d(\Omega(\psi),0)\les{}& C \delta^{-1}\|f \|_{H^{s_d,\sigma_d}_m}+ C_s\delta_{s,\sigma}^{-1} \|f \|_{H^{s,\sigma}_m}+C \delta^{-1}\|\psi\|_{S^{s_d,\sigma_d}_m}^3+C_s\delta_{s,\sigma}^{-1} \|\psi\|_{S^{s_d,\sigma_d}_m}^2\|\psi\|_{S^{s,\sigma}_m}\\
          \les{}& 2C_s +2C_s\delta^2 r^3\les r,
        \end{align*}
        provided that $r:=4C_s$ and $\delta^2\les 2^{-6}C_s^{-3}$, therefore $\Omega:X_r\to X_r$.
        In addition, for $\psi,\varphi \in X_r$,
         \begin{align*}
           d(\Omega(\psi),\Omega(\varphi))\les{}& 2C \delta^{-1}\big(\|\psi\|_{S^{s_d,\sigma_d}_m}+\|\varphi\|_{S^{s_d,\sigma_d}_m}\big)^2 \|\psi-\varphi\|_{S^{s_d,\sigma_d}_m}\\
           &{}
                   +C_s\delta_{s,\sigma}^{-1} \big(\|\psi\|_{S^{s,\sigma}_m}+\|\varphi\|_{S^{s,\sigma}_m}\big) \big(\|\psi\|_{S^{s_d,\sigma_d}_m}+\|\varphi\|_{S^{s_d,\sigma_d}_m}\big)  \|\psi-\varphi\|_{S^{s_d,\sigma_d}_m}\\
                                                &{}   +C_s\delta_{s,\sigma}^{-1} \big(\|\psi\|_{S^{s_d,\sigma_d}_m}+\|\varphi\|_{S^{s_d,\sigma_d}_m}\big)^2 \|\psi-\varphi\|_{S^{s,\sigma}_m}\\
           \les{}& 16C_s\delta^2 r^2d(\psi,\varphi).
         \end{align*}
         Therefore, choosing $\delta:=2^{-5}C_s^{-\frac{3}{2}}$, we have that  $\Omega:X_r\to X_r$ is a strict contraction with a unique fixed point $\psi \in X_r$. Since free solutions are continuous we conclude from Theorem \ref{thm:est-n} that $\psi=\Omega (\psi)\in  C([0,\infty),H^{s_d,\sigma_d}_m(\R^d)\cap H^{s,\sigma}_m(\R^d))$. Also, one can show that solutions are unique in the full space $S^{s_d,\sigma_d}_m$ by a local well-posedness theory for large initial data (e.g.\ analogous to \cite[Theorem 1.2]{Hadac2009}, we skip the details).

         If we consider initial data $f,g \in H^{s_d,\sigma_d}_m\cap H^{s,\sigma}_m $ satisfying $\|f\|_{H^{s_d,\sigma_d}_m},\|g\|_{H^{s_d,\sigma_d}_m}\les \delta$, and the two emanating solutions
         \begin{align*}
           X_r\ni \psi=& \ind_{[0, \infty)} \mc{U}_m f + i \mc{N}(\psi,\psi,\psi),\\
 X_r\ni \varphi=& \ind_{[0, \infty)} \mc{U}_m g + i \mc{N}(\varphi,\varphi,\varphi),
         \end{align*}
         the contraction estimate implies the local Lipschitz bound
         $$
\|\psi-\varphi\|_{S^{s_d,\sigma_d}_m}+\|\psi-\varphi\|_{S^{s,\sigma}_m}\lesa  \|f- g\|_{H^{s_d,\sigma_d}_m}+ \|f- g\|_{H^{s,\sigma}_m}.
         $$
Furthermore, from the implicit function theorem \cite[Corollary 15.1]{Deimling1985} it follows that the map $f\to \psi$ is $C^\infty$.
Finally the scattering claim follows immediately from Lemma \ref{lem:prop of S}.

By time reversibility, the same strategy can be applied for negative times $(-\infty,0]$.

\section{The massless limit}\label{sec:massless-limit}

We now turn to the proof of the convergence of the massive nonlinear Dirac equation to the massless equation as $m\to 0$. The argument we give is fairly general, and in fact only relies on three key properties of the nonlinear flow. Namely the fact we have uniform (in $m$) control over the nonlinear interactions, the convergence of the corresponding linear flows on compact time intervals, and a uniform in $m$ temporal decay property. To state the required properties more precisely, for $s_d = \frac{d-1}{2}$ and $0\les \sigma \les \sigma_d:=\frac{d-2}{2}$ we introduce the transfer operator
        \begin{equation}\label{eqn:R defn}
	\mc{R}_m^\sigma(t) =   \Big( \frac{|\nabla|}{\lr{\nabla}_m}\Big)^{s_d - \sigma}  \mc{U}_m(t)  \mc{U}_0(-t)
			\end{equation}
which is used to control the difference in the massive and massless flows. Note that by construction (and the definition of $S_m^{s_d, \sigma}$ via the pullback) we have the key identity
	\begin{equation}\label{eqn:R mapping}
		\| \mc{R}_m^\sigma(t) \psi \|_{S^{s_d, \sigma}_m} = \| \psi \|_{S^{s_d, \sigma}_0}.
	\end{equation}
Thus $\mc{R}^\sigma_m$ maps the massive to the massless solution spaces, globally in time. This somewhat simple identity is crucial in the proof of Theorem \ref{thm:limit}, and reflects the fact that our solution spaces are \emph{adapted} to the linear flow $\mc{U}_m(t)$, i.e. defined via the pullback. In fact, previous results in the literature (say on the non-relativistic limit) have generally only considered convergence on compact time intervals.

We now state precisely the key properties we require in the proof of Theorem \ref{thm:limit}. Note that in the following we fix $0\les \sigma \les \sigma_d$.

\begin{enumerate}[label= \textbf{(M\arabic*)}]
  \item \label{item:M1} (\emph{Uniform boundedness})
  There exists $\mb{C}\g 0$ such that for any $m\in \RR$ and any $\psi, \varphi \in S^{s_d, \sigma}_m$ we have
                $$ \Big\| \ind_{[0, \infty)}(t) \int_0^t \mc{U}_m(t-s) \gamma^0 \Big( F(\psi + \varphi) - F(\varphi) \Big) ds \Big\|_{S^{s_d, \sigma}_m} \les \mb{C} \| \psi \|_{S^{s_d, \sigma}_m} \big(\| \varphi \|_{S^{s_d, \sigma}_m} + \| \psi \|_{S^{s_d, \sigma}_m}\big)^2.$$\\

  \item \label{item:M2} (\emph{Convergence})
  For every $T>0$ and $\psi \in S_0^{s_d, \sigma}$ we have
                $$ \lim_{m \to 0} \Big\| \int_0^T \ind_{[0, T)}(s) \mc{U}_m(t-s) \Big[ \gamma^0 F\big( \mc{R}^\sigma_m(s) \psi(s) \big) - \mc{R}^\sigma_m(s) \gamma^0 F\big( \psi(s) \big) \Big] ds \Big\|_{S_m^{s_d, \sigma}} = 0. $$ \\

  \item  \label{item:M3} (\emph{Uniform decay})
  For any $\psi \in S_0^{s_d, \sigma}$ we have
                $$ \lim_{T\to \infty} \sup_{ |m|\les 1} \Big\| \int_0^t \ind_{[T, \infty)}(s) \mc{U}_m(t-s) \gamma^0 F\big( \mc{R}^{\sigma}_m(s) \psi(s) \big) ds \Big\|_{S^{s_d, \sigma}_m} = 0. $$  \\
\end{enumerate}

We leave the proof of \ref{item:M1}, \ref{item:M2}, and \ref{item:M3} till later in this section, and now turn to the proof of Theorem \ref{thm:limit}.

\begin{proof}[Proof of Theorem \ref{thm:limit}]
  We only prove convergence on the forward in time interval $[0, \infty)$, as the case of negative times then follows by symmetry.
  Let $\delta>0$ be as in Theorem \ref{thm:gwp}. For any given $m\in \RR$ and data $f^{(m)} \in H^{s_d, \sigma}$ with
        $$ \| f^{(m)} \|_{H^{s_d, \sigma_d}_m} \les \delta$$
we then take $\psi^{(m)} \in S^{s_d, \sigma}_m$ to be the solution to
$$ \psi^{(m)}(t) = \ind_{[0, \infty)}(t) \Big( \mc{U}_m(t) f^{(m)} + i \int_0^t \mc{U}_m(t-s) \gamma^0 F\big( \psi^{(m)}(s) \big) ds \Big).$$
We use the fact that the spaces are nested for fixed $s=s_d$, more precisely
  \[
 \| f^{(m)} \|_{H^{s_d, \sigma_d}_m}\lesa  \| f^{(m)} \|_{H^{s_d, \sigma}_m}, \; \text{ and } \| \psi^{(m)}\|_{S^{s_d, \sigma_d}_m}\lesa  \| \psi^{(m)}\|_{S^{s_d, \sigma}_m}.
  \]
Note that (see Section \ref{sec:proof-gwp}) we then have the smallness condition
            \begin{equation}\label{eqn:thm limit:psi small}
            	 \| \psi^{(m)}\|_{S^{s_d, \sigma}_m} \lesa \delta.
            	\end{equation}
The existence of such a solution can also (and essentially equivalently) be seen as a consequence of the uniform bound \ref{item:M1}. Define the modified solution operators
        $$ \mc{U}_m^\sigma(t) = \lr{\nabla}_m^{s_d - \sigma} |\nabla|^{\sigma} \mc{U}_m(t) $$
and take
			$$ u^{(m)}(t) =  \mc{U}_m^{\sigma}(-t) \psi^{(m)}(t) - \mc{U}^{\sigma}_0(-t) \psi^{(0)}(t). $$
Note that the smallness condition \eqref{eqn:thm limit:psi small} implies that
	\begin{equation}\label{eqn:thm limit:u small}			 \| u^{(m)} \|_{\ell^2 U^a} \les \| \psi^{(m)} \|_{S^{s_d, \sigma}_m} + \| \psi^{(0)} \|_{S^{s_d, \sigma}_0} \lesa \delta. 	
	\end{equation}
Our goal is to prove that
	\begin{equation}\label{eqn:thm limit:goal}
				 \limsup_{m\to 0} \| u^{(m)} \|_{\ell^2 U^a} \les 2 \limsup_{m\to 0} \| u^{(m)}(0)\|_{L^2_x}.
	\end{equation}
We begin by observing that for any $t\g 0$, the normalised difference $u^{(m)}$ satisfies the equation
        \begin{align*}
            u^{(m)}(t) -  u^{(m)}(0) &= i  \int_0^t  \mc{U}_m^{\sigma_d}(-s)  \Big[ \gamma^0 F\Big( \big(\mc{U}_m^{\sigma_d}(-s)\big)^{-1} u^{(m)}(s) + \mc{R}^{\sigma_d}_m(s) \psi^{(0)}(s) \Big) - R^{\sigma_d}_m(s) \gamma^0 F\big( \psi^{(0)}(s) \big)\Big] ds \\
                &= i\int_0^t  \mc{U}^{\sigma_d}_m(-s) \gamma^0 \Big[ F\Big( \big(\mc{U}_m^{\sigma_d}(-s)\big)^{-1} u^{(m)}(s) + \mc{R}^{\sigma_d}_m(s) \psi^{(0)}(s) \Big) - F\big( \mc{R}^{\sigma_d}_m(s)\psi^{(0)}(s) \big)\Big] ds \\
                &\qquad  + i\int_0^t  \mc{U}_m^{\sigma_d}(-s) \Big[ \gamma^0 F \Big( \mc{R}^{\sigma_d}_m(s) \psi^{(0)}(s) \Big) - R^{\sigma_d}_m(s) \gamma^0 F\big( \psi^{(0)}(s) \big)\Big] ds.
        \end{align*}
Applying the uniform bound \ref{item:M1}, the identity \eqref{eqn:R mapping}, and the smallness bounds \eqref{eqn:thm limit:psi small} and \eqref{eqn:thm limit:u small} we see that
	\begin{align*}
		\Big\| 	\int_0^t  &\mc{U}^{\sigma_d}_m(-s) \gamma^0 \Big[ F\Big( \big(\mc{U}_m^{\sigma_d}(-s)\big)^{-1} u^{(m)}(s) + \mc{R}^{\sigma_d}_m(s) \psi^{(0)}(s) \Big) - F\big( \mc{R}^{\sigma_d}_m(s)\psi^{(0)}(s) \big)\Big] ds\Big\|_{\ell^2 U^a} \\
			&=\Big\| 	\int_0^t  \mc{U}_m(t-s) \gamma^0 \Big[ F\Big( \big(\mc{U}_m^{\sigma_d}(-s)\big)^{-1} u^{(m)}(s) + \mc{R}^{\sigma_d}_m(s) \psi^{(0)}(s) \Big) - F\big( \mc{R}^{\sigma_d}_m(s)\psi^{(0)}(s) \big)\Big] ds\Big\|_{S^{s_d, \sigma}_m} \\
			&\lesa \big\|  \big(\mc{U}_m^{\sigma_d}(-t)\big)^{-1} u^{(m)} \big\|_{S^{s_d, \sigma}_m}
			\Big( \big\|  \big(\mc{U}_m^{\sigma_d}(-t)\big)^{-1} u^{(m)} \big\|_{S^{s_d, \sigma}_m} + \big\| \mc{R}_m^{\sigma}(t) \psi^{(0)} \big\|_{S^{s_d, \sigma}_m}\Big)^2 \\
			&\lesa \delta^2 \| u^{(m)} \|_{\ell^2 U^a} \les \frac{1}{2} \| u^{(m)} \|_{\ell^2 U^a},
	\end{align*}
        by choosing $\delta>0$ smaller if necessary.
Consequently we obtain the bound
	$$ \| u^{(m)} \|_{\ell^2 U^a} \les 2 \| u^{(m)}(0)\|_{L^2_x} + 2 \Big\| \int_0^t  \mc{U}_m^{\sigma_d}(-s) \Big[ \gamma^0 F \Big( \mc{R}^{\sigma_d}_m(s) \psi^{(0)}(s) \Big) - R^{\sigma_d}_m(s) \gamma^0 F\big( \psi^{(0)}(s) \big)\Big] ds \Big\|_{\ell^2 U^a}. $$	
To conclude the proof, we now observe that since $\mc{R}_0^\sigma(t) = 1$ and $\mc{U}^{\sigma}_m(-t)\mc{R}^\sigma_m(t) = \mc{U}_0^{\sigma}(-t)$ we have for any $|m| \les 1$
    \begin{align*}
      \Big\| \int_0^t \ind_{[T, \infty)}(s) \mc{U}_m^{\sigma_d}(-s) \Big[ \gamma^0 &F \Big( \mc{R}^{\sigma_d}_m(s) \psi^{(0)}(s) \Big) - R^{\sigma_d}_m(s) \gamma^0 F\big( \psi^{(0)}(s) \big)\Big] ds \Big\|_{\ell^2 U^a}\\
      &\les 2 \sup_{|m|\les 1} \Big\| \int_0^t \ind_{[T, \infty)}(s) \mc{U}_m^{\sigma_d}(-s)  \gamma^0 F \Big( \mc{R}^{\sigma_d}_m(s) \psi^{(0)}(s) \Big) ds \Big\|_{\ell^2 U^a}
    \end{align*}
and hence applying \ref{item:M2} and \ref{item:M3} (and recalling the definition of $\mc{U}^\sigma_m$) we obtain
    \begin{align*}
     &\limsup_{m\to 0} \Big\| \int_0^t  \mc{U}_m^{\sigma_d}(-s) \Big[ \gamma^0 F \Big( \mc{R}^{\sigma_d}_m(s) \psi^{(0)}(s) \Big) - R^{\sigma_d}_m(s) \gamma^0 F\big( \psi^{(0)}(s) \big)\Big] ds \Big\|_{\ell^2 U^a}\\
     &\lesa \sup_{T>0} \limsup_{m\to 0} \Big\| \int_0^t \ind_{[0, T)}(s) \mc{U}_m(t-s) \Big[ \gamma^0 F \Big( \mc{R}^{\sigma_d}_m(s) \psi^{(0)}(s) \Big) - R^{\sigma_d}_m(s) \gamma^0 F\big( \psi^{(0)}(s) \big)\Big] ds \Big\|_{S^{s_d, \sigma}_m}\\
     &\qquad + \limsup_{T\to \infty} \sup_{|m|\les 1} \Big\| \int_0^t \ind_{[T, \infty)}(s) \mc{U}_m(t-s)  \gamma^0 F \Big( \mc{R}^{\sigma_d}_m(s) \psi^{(0)}(s) \Big) ds \Big\|_{S^{s_d, \sigma}_m}\\
     &= 0.
    \end{align*}
Therefore we see that \eqref{eqn:thm limit:goal} holds.
\end{proof}

It only now remains to give the proof of the three key properties \ref{item:M1}, \ref{item:M2}, and \ref{item:M3}.

\subsection{Proof of (M1)} Because of
$\| \psi\|_{S^{s_d, \sigma_d}_m}\lesa  \| \psi\|_{S^{s_d, \sigma}_m}$
this is simply an application of the trilinear estimate contained in Theorem \ref{thm:est-n} together with the cubic structure of the nonlinearity $F$. Note that the constant in \eqref{eq:est-n} is independent of $m\in \RR$.

\subsection{Proof of (M2)} The convergence property is an application of Lemma \ref{lem:comp-norms-gen} together with the convergence of the symbols of the free solution operators $\mc{U}_m$ as $m\to 0$ on compact time intervals.  To simplify the argument to follow, we begin with a lemma containing a precise statement of this convergence in the solution spaces $S^{s_d, \sigma}_0$. Note that the convergence of $\mc{U}_m$ to $\mc{U}_0$ is straight forward in $L^\infty_t L^2_x$, the main novelty of the following lemma is that this also holds in our $U^a$ based spaces.

\begin{lemma}\label{lem:comp-norms II}
Let $s_d = \frac{d-1}{2}$ and $ 0\les \sigma \les \sigma_d $. Then for any $T, R>0$, $I=[0, T)$, and $m \in \RR $ we have
    $$\big\| \ind_I(t) \big( \mc{R}_m^\sigma -1 \big) P_{\g R} \psi \big\|_{S^{s_d,\sigma}_0}  \lesa{} |m| \big( R^{-1} + T\big) (1 + T|m|) \| \psi \|_{S^{s_d,\sigma}_{0}}, $$
 and
 		$$ \big\| \ind_I(t) (\mc{R}^{2s_d - \sigma}_m-1) P_{\g R} \varphi \big\|_{S^{-s_d, -\sigma}_{w, 0}} \lesa{}  |m| \big( R^{-1} + T\big) (1 + T|m|)\Big( \frac{\lr{R}_m}{R}\Big)^{2(s_d-\sigma)} \|\varphi \|_{S^{-s_d,-\sigma}_{w, 0}}.
 $$
\end{lemma}
\begin{proof}
  We may as well assume that $m\g 0$, as the case $m< 0$ is identical.
  Unpacking the definitions, the first claim reduces to showing that
        $$ \Big\| \ind_I(t) \Big[ \Big( \frac{|\nabla|}{\lr{\nabla}_m}\Big)^{s'} \mc{U}_0(-t) \mc{U}_m(t) - 1 \Big] P_{\g R} \phi \Big\|_{\ell^2 U^a} \lesa m^2 \big( R^{-1} + T\big)^2 \| \phi \|_{\ell^2 U^a}, $$
where $s'=s_d-\sigma \geq \frac12$.
After using \eqref{eqn:lin dirac flow} to write $\mc{U}_0(-t) \mc{U}_m(t) = \sum_{\pm_1, \pm_2} e^{it(\pm_1 \lr{\nabla}_m - \pm_2 |\nabla|)} \Pi^{(0)}_{\pm_1} \Pi^{(m)}_{\pm_2}$, the required bound follows by bounding the $\ell^2 U^a$ operator norms of
    $$A_1\phi(t):= \ind_I(t) \Big[ \Big( \frac{|\nabla|}{\lr{\nabla}_m}\Big)^{s'} e^{it(\lr{\nabla}_m - |\nabla|)} -1 \Big] P_{\g R} \Pi_+^{(0)} \Pi_+^{(m)} \phi(t)$$
and
    $$A_2\phi(t):= \ind_I(t) \Big[ \Big( \frac{|\nabla|}{\lr{\nabla}_m} \Big)^{s'} e^{it(\lr{\nabla}_m + |\nabla|)} -1 \Big] P_{\g R} \Pi_-^{(0)} \Pi_+^{(m)} \phi(t)$$
as the remaining cases can be estimated in the same way. To bound the operator norm of $A_1: \ell^2 U^{a}\to \ell^2 U^{a}$ we invoke Lemma \ref{lem:comp-norms-gen}. Note that for any $\xi \in \RR^d$ with $|\xi|\g R$ we have
the bounds
    $$
    \Big|\Big( \frac{|\xi|}{\lr{\xi}_m}\Big)^{s'} - 1\Big|\lesa  \frac{m^2}{R^2}\qquad \text{and} \qquad |e^{it(\lr{\xi}_m - |\xi|)}-1|\lesa  \frac{ tm^2}{\lr{\xi}_m}.
    $$
Therefore, using Remark \ref{eqn:rem Up operator bounds:Up} (with $q=1$), we get
\begin{align*}
  \|A_1\|_{U^a(\R,\mathcal{L}(L^2))}\lesa{}& \sup_{t \in I} \Big\| \Big(\Big( \frac{|\xi|}{\lr{\xi}_m}\Big)^{s'} e^{it(\lr{\xi}_m - |\xi|)} -1  \Big)p_{\g R}(\xi)  \Big\|_{L^\infty_{\xi}}\\ & +
                                                                                                                                                                                                       \int_I \Big\| \partial_t \Big( \Big( \frac{|\xi|}{\lr{\xi}_m}\Big)^{s'} e^{it(\lr{\xi}_m - |\xi|)} -1  \Big)p_{\g R}(\xi)  \Big\|_{L^\infty_{\xi}} dt\\
  \lesa{}& m^2 \big( R^{-1} + T\big)^2.
\end{align*}
Similarly for $A_2$, using the bound
    $$
    |\Pi_+^{(0)}(\xi)\Pi_-^{(m)}(\xi)|=|(\Pi_+^{(0)}(\xi)-\Pi_+^{(m)}(\xi))\Pi_-^{(m)}(\xi)|\lesa \frac{|m|}{\lr{\xi}_m}
    $$
gives
\begin{align*}
\|A_2\|_{U^a(\R,\mathcal{L}(L^2))}\lesa{}& \sup_{t \in I} \Big\|\Big( \Big( \frac{|\xi|}{\lr{\xi}_m} \Big)^{s'} e^{it(\lr{\xi}_m + |\xi|)} -1 \Big) \Pi_-^{(0)}(\xi) \Pi_+^{(m)}(\xi) p_{\g R}(\xi) \|_{L^\infty_\xi} \\
  &+\int_I \Big\| \partial_t  \Big( \Big( \frac{|\xi|}{\lr{\xi}_m} \Big)^{s'} e^{it(\lr{\xi}_m + |\xi|)} -1 \Big) \Pi_-^{(0)}(\xi) \Pi_+^{(m)}(\xi) p_{\g R}(\xi) \Big\|_{L^\infty_{\xi}} dt \\
  \lesa{}& m(R^{-1} + T).
\end{align*}
Hence the first claimed bound follows.

For the second claimed bound, we write
    $$ \mc{R}^{2s_d - \sigma}_m - 1 = \Big( \frac{\lr{\nabla}_m}{|\nabla|}\Big)^{2(s_d - \sigma)} \big( \mc{R}^\sigma_m -1\big) + \Big( \frac{\lr{\nabla}_m}{|\nabla|}\Big)^{2(s_d - \sigma)} - 1. $$
After observing that Lemma \ref{lem:comp-norms-gen} gives
    $$ \Big\| \Big( \frac{\lr{\nabla}_m}{|\nabla|}\Big)^{2(s_d - \sigma)} P_{\g R} \phi \Big\|_{S^{-s_d, -\sigma}_{w, 0}} \lesa \Big( \frac{\lr{R}_m}{R}\Big)^{2(s_d - \sigma)}  \| P_{\g R} \phi \|_{S^{-s_d, -\sigma}_{w, 0}}$$
and
$$ \Big\| \Big(\Big( \frac{\lr{\nabla}_m}{|\nabla|}\Big)^{2(s_d - \sigma)} -1 \Big)P_{\g R} \phi \Big\|_{S^{-s_d, -\sigma}_{w, 0}} \lesa \frac{m^2}{R^2} \Big( \frac{\lr{R}_m}{R}\Big)^{2(s_d - \sigma)} \| P_{\g R} \phi \|_{S^{-s_d, -\sigma}_{w, 0}},$$
the required estimate follows by adapting the multiplier bounds for $A_1$ and $A_2$ above.
 \end{proof}

We now return the proof of \ref{item:M2}. Let $\psi \in S^{s_d, \sigma}_0$. An application of \ref{item:M1} together with the definition of $\mc{R}_m^\sigma$ shows that
    \begin{align*}
      &\Big\| \int_0^T \ind_{[0, T)}(s) \mc{U}_0(t-s) \Big[ \big(\mc{R}^\sigma_m(s)\big)^{-1} \gamma^0 F\big( \mc{R}^\sigma_m(s) \psi(s) \big) - \gamma^0 F\big( \psi(s) \big) \Big] ds \Big\|_{S_0^{s_d, \sigma}}\\
      &= \Big\| \int_0^T \ind_{[0, T)}(s) \mc{U}_m(t-s) \Big[ \gamma^0 F\big( \mc{R}^\sigma_m(s) \psi(s) \big) - \mc{R}^\sigma_m(s) \gamma^0 F\big( \psi(s) \big) \Big] ds \Big\|_{S_m^{s_d, \sigma}}\\
      &\les \Big\| \int_0^T \ind_{[0, T)}(s) \mc{U}_m(t-s) \gamma^0 F\big( \mc{R}^\sigma_m(s) \psi(s) \big)ds \Big\|_{S_m^{s_d, \sigma}}+ \Big\| \int_0^T \ind_{[0, T)}(s) \mc{U}_0(t-s) \mc{R}^\sigma_m(s) \gamma^0 F\big( \psi(s) \big) \Big] ds \Big\|_{S_0^{s_d, \sigma}}\\
      &\lesa \| \mc{R}^\sigma_m(t) \psi \|_{S^{s_d, \sigma}_m}^3 + \| \psi \|_{S^{s_d, \sigma}_0}^3 \lesa  \|\psi \|_{S^{s_d, \sigma}_0}^3.
    \end{align*}
In particular, in view of the dyadic definition of $\ell^2 U^a$ (and hence $S^{s_d, \sigma}_0$), it suffices to prove that for any $T, R>0$ and any $\psi \in S^{s_d, \sigma}_0$ with $ \supp \widehat{\psi} \subset \{ R^{-1} \les |\xi| \les R\}$ we have
       \begin{equation}\label{eqn:M2:goal}
             \lim_{m \to 0} \Big\| \int_0^T \ind_{[0, T)}(s) \mc{U}_0(t-s) P_{\g R^{-1}}  \Big[ \big(\mc{R}^\sigma_m(s)\big)^{-1} \gamma^0 F\big( \mc{R}^\sigma_m(s) \psi(s) \big) - \gamma^0 F\big( \psi(s) \big) \Big] ds \Big\|_{S_0^{s_d, \sigma}} = 0.
       \end{equation}
Applying Lemma \ref{lem:energy ineq} together with \ref{item:M1} we see that for any $\varphi \in S^{-s_d, -\sigma}_{w,0}$ we have
    \begin{align*}
      \Big| \int_0^T \int_{\RR^d}& \overline{\varphi} P_{\g R^{-1}} \Big[ \gamma^0 \big(\mc{R}^\sigma_m\big)^{-1} \gamma^0 F\big( \mc{R}^\sigma_m \psi \big) -  F\big( \psi \big) \Big] dx dt\Big| \\
            &\les \Big| \int_0^T \int_{\RR^d} \overline{ \mc{R}^{2s_d - \sigma}_m P_{\g R^{-1}}\varphi}  F\big( \mc{R}^\sigma_m \psi \big) -  \overline{P_{\g R^{-1}} \varphi} F\big( \psi \big)  dx dt\Big|\\
            &\lesa  \| \ind_I(t) (\mc{R}_m^{2s_d - \sigma} -1) P_{\g R^{-1}} \varphi\|_{S^{-s_d, -\sigma}_{w,0}} \Big( \| \ind_I(t) (\mc{R}_m^{\sigma} - 1) \psi\|_{S^{s_d, \sigma}_0} + \| \psi \|_{S^{s_d, \sigma}_0}\Big)^3 \\
            &\qquad \qquad + \|\varphi\|_{S^{-s_d, -\sigma}_{w,0}}  \|\ind_I(t) (\mc{R}_m^{\sigma} - 1) \psi\|_{S^{s_d, \sigma}_0} \Big( \| \ind_I(t)(\mc{R}_m^{\sigma} - 1) \psi\|_{S^{s_d, \sigma}_0} + \| \psi \|_{S^{s_d, \sigma}_0}\Big)^2
    \end{align*}
Consequently Lemma \ref{lem:energy ineq} and Lemma \ref{lem:comp-norms II} imply that
    \begin{align*}
        \Big\| \int_0^T \ind_{[0, T)}(s) \mc{U}_0(t-s)& P_{\g R^{-1}}  \Big[ \big(\mc{R}^\sigma_m(s)\big)^{-1} \gamma^0 F\big( \mc{R}^\sigma_m(s) \psi(s) \big) - \gamma^0 F\big( \psi(s) \big) \Big] ds \Big\|_{S_0^{s_d, \sigma}} \\
        &\lesa |m| \big( R^{-1} + T\big) (1 + T|m|)^4 \Big( \frac{\lr{R}_m}{R}\Big)^{2(s_d-\sigma)} \| \psi \|_{S^{s_d, \sigma}_0}^3.
    \end{align*}
Letting $m\to 0$ we conclude that \eqref{eqn:M2:goal} holds.

\subsection{Proof of (M3)} As in the proof of \ref{item:M2}, in view of the uniform bound in \ref{item:M1}, it suffices to prove \ref{item:M3} under the assumption that $\supp \widehat{\psi} \subset \{R^{-1} \les |\xi| \les R\}$ for some $R>1$. Moreover, since we now have
             \begin{equation}\label{eqn:M3:Ua bdd by S}
             \| \mc{U}_0( - t) \psi \|_{U^a} \lesa \big( \log(R) \big)^{\frac{1}{2}}  \| \mc{U}_0(-t) \psi \|_{\ell^2 U^a} \lesa \big( R^{d-1} \log(R) \big)^{\frac{1}{2}} \| \psi \|_{S^{s_d, \sigma}_0},
             \end{equation}
we reduce to simply proving that \ref{item:M3} holds when $\mc{U}_0(-t) \psi \in U^a$. Again applying \ref{item:M1} and noting that we also have
           \begin{equation}\label{eqn:M3:S bdd by Ua}
           \| \psi \|_{S^{s_d, \sigma}_0} \lesa \big( R^{d-1} \log(R) \big)^{\frac{1}{2}} \| \mc{U}_0(-t) \psi \|_{U^a}
           \end{equation}
we see that it is enough to consider the case where $\mc{U}_0(-t) \psi$ is a $U^a$ atom with $\supp \widehat{\psi} \subset \{ R^{-1} \les |\xi| \les R\}$. In fact, by the definition of a $U^a$ atom, we can write $\mc{U}_0(-t) \psi = \sum_{j=1}^N \ind_{[T_j, T_{j+1})}(t) f_j$ with $T_{N+1} = \infty$. In particular, for sufficiently large $T$, $\ind_{[T, \infty)} \psi =  \ind_{[T, \infty)} \mc{U}_0(t) f_{N+1}$ and hence to prove \ref{item:M3}, it suffices to take $\psi = \mc{U}_0(t) f$ with $f\in L^2$ satisfying $\supp \widehat{f} \subset \{R^{-1} \les |\xi| \les R\}$. Moreover, as \ref{item:M1} gives
        $$\Big\|\int_0^t \ind_{[T, \infty)}(s) \mc{U}_m(t-s)  \gamma^0 F\big( \mc{R}^\sigma_m(s) \psi(s) \big) ds \Big\|_{S^{s_d, \sigma}_m} \lesa \| \psi \|_{S^{s_d, \sigma}_0}^3 \lesa R^{s_d}\| f \|_{L^2},  $$
it is enough to consider $f\in L^2\cap L^1$. Rescaling the dispersive estimate for the Klein-Gordon equation \eqref{eqn:KG disp} we see that
            $$ \sup_{m\in [0, 1]} \| \mc{U}_m(t) P_\lambda f \|_{L^\infty_x} \lesa |t|^{-\frac{d-1}{2}} \lambda^{\frac{d+1}{2}} \| f\|_{L^1}.$$
Hence applying Theorem \ref{thm:rest+disp} and letting $g = (\frac{|\nabla|}{\lr{\nabla}_m})^{s_d} f$ we have for any $\frac{4}{d+1} < q< 2$
    \begin{align*}
      \sup_{m\in [0, 1]} \Big\|\int_0^t \ind_{[T, \infty)}(s) \mc{U}_m(t-s)  &\gamma^0 F\big( \mc{R}^\sigma_m(s) \psi(s) \big)ds \Big\|_{S^{s_d, \sigma}_m}\\
                    &\lesa R^{s_d} \sup_{m\in [0, 1]} \| \ind_{[T, \infty)}(t) \big( \overline{\mc{U}_m(t) g} \mc{U}_m(t) g \big) \gamma^0 \mc{U}_m(t) g \big\|_{L^1_t L^2_x} \\
                    &\lesa R^{s_d} \sup_{m\in [0, 1]} \| \ind_{[T, \infty)}(t) \big( \overline{\mc{U}_m(t) g} \mc{U}_m(t) g \big)\big\|_{L^q_t L^2_x}  \| \ind_{[T, \infty)}(t) \gamma^0 \mc{U}_m(t) g \|_{L^{q'}_t L^\infty_x} \\
                    &\lesa R^{ \frac{5d}{2} - \frac{1}{q} }  T^{\frac{3-d}{2} - \frac{1}{q}} \| f \|_{L^1}^3.
    \end{align*}
Hence letting $T\to \infty$ we see that \ref{item:M3} follows. Note that we only require the bilinear restriction theory when $d=2$, as then we are forced to take $q<2$.

\subsection{Proof of Corollary \ref{cor:mlimit}}\label{subsec:proof-cor-mlimit}
Since $\mc{U}_m(t)$ is isometric in $\dot{H}^{s_d}$,
 \begin{align*}
   \big\| \ind_I(t)(\psi^{(m)}  -  \psi^{(0)})\big\|_{L^\infty_t \dot{H}^{s_d}} ={}&\big\| \ind_I(t)( \mc{U}_m(-t )\psi^{(m)}  -   \mc{U}_m(-t)\psi^{(0)})\big\|_{L^\infty_t \dot{H}^{s_d}} \\
   \les{}& \big\| \mc{U}_m(-t ) \psi^{(m)}  - \mc{U}_0(-t ) \psi^{(0)}\big\|_{L^\infty_t \dot{H}^{s_d}}+\big\| \ind_I(t)( \mc{U}_0(-t )\psi^{(0)}  -   \mc{U}_m(-t)\psi^{(0)})\big\|_{L^\infty_t \dot{H}^{s_d}}.
 \end{align*}
Under the hypothesis of Theorem \ref{thm:limit} we have the estimate
 \begin{align*} &\big\||\nabla|^{s_d} ( \mc{U}_m(-t ) \psi^{(m)}  -  \mc{U}_0(-t) \psi^{(0)})\big\|_{\ell^2 U^a}\\  \lesa{}& |m|^{s_d-\sigma} \delta + \big\| \lr{\nabla}_m^{s_d-\sigma}|\nabla|^\sigma \mc{U}_m(-t) \psi^{(m)}  -  |\nabla|^{s_d} \mc{U}_0(-t) \psi^{(0)}\big\|_{\ell^2U^a} \to 0 \text{ for }m \to 0,
 \end{align*}
and that, as $\ell^2U^{a}\subset L^\infty( \R,L^2 (\RR^d))$,
$$\lim_{m\to 0} \big\| \mc{U}_m(-t ) \psi^{(m)}  - \mc{U}_0(-t ) \psi^{(0)}\big\|_{L^\infty_t \dot{H}^{s_d}}= 0,$$
which shows that the first term vanishes in the limit.

Concerning the second term, since $\mc{U}_0(t)$ is isometric in $\dot{H}^{s_d}$,
\begin{align*}
  \big\| \ind_I(t)( \mc{U}_0(-t )\psi^{(0)}  -   \mc{U}_m(-t)\psi^{(0)})\big\|_{L^\infty_t \dot{H}^{s_d}}=&\|\ind_I(t)|\nabla|^{s_d}\Big( \mc{U}_0(t)\mc{U}_m(-t)-1\Big)\psi^{(0)}\|_{L^\infty_t \dot{H}^{s_d}}\\
\lesa{}& \Big(\sum_{\lambda \in 2^\Z}\lambda^{2s_d}\|\ind_I(t)\Big(\mc{U}_0(t) \mc{U}_m(-t)-1\Big)\psi^{(0)}_\lambda\|_{L^\infty_t L^2_x}^2\Big)^{\frac12}<\infty
\end{align*}
because $\psi^{(0)}\in S^{s_d,\sigma}_0$. Therefore, it is enough to prove that for each $\lambda \in 2^\Z$ we have
\[
\lim_{m \to 0}\|\ind_I(t)\Big(\mc{U}_0(t) \mc{U}_m(-t)-1\Big)\psi^{(0)}_\lambda\|_{L^\infty_t L^2_x}=0.
\]
But this is true, because from the proof of Lemma \ref{lem:comp-norms II} we know that
\[
\sup_{t \in I}\big\|\sum_{\pm_1, \pm_2} (e^{-it(\pm_1 \lr{\xi}_m - \pm_2 |\xi|)} -1)\Pi^{(0)}_{\pm_1}(\xi) \Pi^{(m)}_{\pm_2}(\xi)p_\lambda(\xi)\big\|_{L^\infty_\xi}\lesa \frac{\lr{T} m}{\lambda}.
\]
This concludes the proof of Corollary \ref{cor:mlimit}.

\section{The non-relativistic limit}\label{sec:nonrelativistic-limit}
In this section we give the proof of Theorem \ref{thm:limit c}. Similar to the proof of the massless limit $m\to 0$ given in Section \ref{sec:massless-limit}, the proof of the nonlinear relativistic limit relies on three key properties of the flow: uniform control over the nonlinear interactions as $c\to \infty$, the local in time convergence of the linear flows, and a uniform temporal decay property. To state these properties more precisely, we require some additional notation. We begin by noting that due to the change in scaling in \eqref{eq:dirac c}, the spaces $S^{s, \sigma}_m$ used to control solutions to the Dirac equation \ref{eq:cd} need to be adjusted.  To this end, we introduce the Schr\"odinger solution operator
        $$ \mc{V}_\infty(t) = e^{ i \frac{t}{2} \gamma^0 \Delta}$$
and given any $1\les c \les \infty$, we define the norms
        $$ \| \psi \|_{X_c} = \big\| |\nabla|^{\frac{d-2}{2}} \lr{c^{-1} \nabla}^{\frac{1}{2}} \mc{V}_c(-t) \psi \big\|_{\ell^2 U^a}, \qquad \| \psi \|_{Y_c} = \big\| |\nabla|^{-\frac{d-2}{2}} \lr{c^{-1} \nabla}^{-\frac{1}{2}} \mc{V}_c(-t) \psi \big\|_{\ell^2 V^{a'}}$$
and take $X_c = |\nabla|^{-\frac{d-2}{2}} \lr{c^{-1} \nabla}^{-\frac{1}{2}} \mc{V}_c(t) \ell^2 U^a$ and $Y_c = |\nabla|^{\frac{d-2}{2}} \lr{c^{-1} \nabla}^{\frac{1}{2}} \mc{V}_c(t) \ell^2 V^{a'}$. As we shall see below, these norms are simply rescaled versions of the norms $S^{s_d, \sigma_d}_m$ and $S^{-s_d, -\sigma_d}_w$ used earlier. Define the operator
          \begin{equation}\label{defn:rel R}
                \tilde{\mc{R}}_c(t) = \lr{c^{-1}\nabla}^{-\frac{1}{2}}  \mc{V}_c(t) \mc{V}_\infty(-t).
          \end{equation}
The operator $\tilde{R}_c$ plays similar role as that played by $\mc{R}_m^{\sigma}$ in the massless limit, namely it gives a way to measure the difference in the linear flows. For later use, we note that we again have the key mapping property
			\begin{equation}\label{eqn:rel R:mapping}
				\| \tilde{\mc{R}}_c(t) \psi \|_{X_c} = \| \psi \|_{X_\infty}			\end{equation}
We can now state the precise properties we require.\\

\begin{enumerate}[label= \textbf{(N\arabic*)}]
  \item \label{item:N1} (\emph{Uniform boundedness})
  There exists $\mb{C}\g 0$ such that for any $1\les c \les \infty$ and any $\psi, \varphi \in X_c$ we have
                $$ \Big\| \ind_{[0, \infty)}(t) \int_0^t \mc{V}_c(t-s) \gamma^0 \Big( F(\psi + \varphi) - F(\varphi) \Big) ds \Big\|_{X_c} \les \mb{C} \| \psi \|_{X_c} \big(\| \varphi \|_{X_c} + \| \psi \|_{X_c}\big)^2.$$\\

  \item \label{item:N2} (\emph{Convergence})
  For every $T>0$ and $\psi \in X_\infty$ we have
                $$ \lim_{c \to \infty} \Big\| \int_0^T \ind_{[0, T)}(s) \mc{V}_c(t-s) \Big[ \gamma^0 F\big( \tilde{\mc{R}}_c(s) \psi(s) \big) - \tilde{\mc{R}}_c(s) \gamma^0 F_1\big( \psi(s) \big) \Big] ds \Big\|_{X_c} = 0. $$ \\

  \item  \label{item:N3} (\emph{Uniform decay})
  For any $\psi \in X_\infty$ we have
                $$ \lim_{T\to \infty} \sup_{ 1\les c \les \infty} \Big\| \int_0^t \ind_{[T, \infty)}(s) \mc{V}_c(t-s) \gamma^0 F\big( \tilde{\mc{R}}_c(s) \psi(s) \big) ds \Big\|_{X_c} = 0. $$  \\
\end{enumerate}

We leave the proof of \ref{item:N1}, \ref{item:N2}, and \ref{item:N3} till later in this section, and turn to the proof of Theorem \ref{thm:limit c}.

\begin{proof}[Proof of Theorem \ref{thm:limit c}] The proof follows by repeating the proof of Theorem \ref{thm:limit} (note that the properties \ref{item:N1}, \ref{item:N2}, \ref{item:N3} exactly match the corresponding properties \ref{item:M1}, \ref{item:M2}, \ref{item:M3} used in the proof of Theorem \ref{thm:limit}). We only prove convergence on the forward in time interval $[0, \infty)$, as the case of negative times then follows by symmetry. Throughout the following argument, all implicit constants are independent of $1\les c \les \infty$.  A standard iteration argument (see for instance Section \ref{sec:proof-gwp}) together with the uniform nonlinear bound \ref{item:N1} shows that there exists $\delta>0$ such that for any $1\les c \les \infty$ and any data $f^{(c)}$ satisfying
                    $$ \| \lr{c^{-1} \nabla}^{\frac{1}{2}} f^{(c)} \|_{\dot{H}^{\frac{d-2}{2}}_x} \les \delta$$
there exists a (unique) solution $\psi^{(c)} \in X_c$ to
                    $$ \psi^{(c)}(t) = \ind_{[0, \infty)}(t) \Big( \mc{V}_c(t) f^{(c)} + i \int_0^t \mc{V}_c(t-s) \gamma^0 F\big( \psi^{(c)}(s) \big) ds \Big)$$
(with $F$ replaced with $F_1$ when $c=\infty$) satisfying the smallness bound
    \begin{equation}\label{eqn:thm limit c:psi small}
        \| \psi^{(c)}\|_{X_c} \lesa \delta.
    \end{equation}
Let
			$$ v^{(c)}(t) =  |\nabla|^{\frac{d-2}{2}} \lr{c^{-1} \nabla}^\frac{1}{2} \mc{V}_c(-t) \psi^{(c)}(t) - |\nabla|^{\frac{d-2}{2}} \mc{V}_\infty(-t) \psi^{(\infty)}(t). $$
Note that the smallness condition \eqref{eqn:thm limit c:psi small} implies that
	\begin{equation}\label{eqn:thm limit c:v small}
        \| v^{(c)} \|_{\ell^2 U^a} \les \| \psi^{(c)} \|_{X_c} + \| \psi^{(\infty)} \|_{X_\infty} \lesa \delta.
	\end{equation}
Our goal is to prove that
	\begin{equation}\label{eqn:thm limit c:goal}
				 \limsup_{c\to \infty} \| v^{(c)} \|_{\ell^2 U^a} \lesa \limsup_{c \to \infty} \| u^{(c)}(0)\|_{L^2_x}.
	\end{equation}
But, by repeating the argument in the proof of Theorem \ref{thm:limit}, this follows from the properties \ref{item:N1}, \ref{item:N2}, \ref{item:N3} together with the smallness assumptions \eqref{eqn:thm limit c:psi small} and \eqref{eqn:thm limit c:v small}.
\end{proof}

It only remains to give the proof of the properties \ref{item:N1}, \ref{item:N2}, and \ref{item:N3}.

\subsection{Proof of \ref{item:N1}}\label{subsec:proofn1} If $1\les c < \infty$, this follows directly from \ref{item:M1} via rescaling after noting that for any $f\in L^2$ and $\psi \in X_c$ we have
            $$ \mc{V}_c(t) f = \big( \mc{U}_{c^2}(t) \Lambda_c f\big)( c^{-1} x), \qquad \| \psi \|_{X_c} = \| \Lambda_c \psi \|_{S^{s_d, \sigma_d}_{c^2}}$$
where $(\Lambda_c f)(x) = f(cx)$ denotes the spatial dilation by $1\les c < \infty$. On the other hand, if $c=\infty$,
we first note that for any $N_d \times N_d$ matrices $\mb{A}$ and $\mb{B}$ we have the trilinear bound
$$ \| |\nabla|^{\sigma_d}\Big((\overline{\psi}_1 \mb{A} \psi_2) \mb{B} \psi_3\Big) \|_{L^1_t L^2_x} \lesa |\mb{A}| |\mb{B}| \prod_{j=1}^3 \| |\nabla|^{\sigma_d} \psi_j \|_{L^3_t L^{\frac{6d}{3d-4}}_x}. $$
Indeed, if $d=2$ this is simply H\"older's inequality. If $d=3$ this fractional Leibniz rule follows from a double application of \cite[Theorem 1]{Grafakos2014} combined with the Sobolev embedding $\| \psi_j \|_{L^3_t L^{9}_x}\lesa \| |\nabla|^{\frac12} \psi_j \|_{L^3_t L^{\frac{18}{5}}_x}$.
Therefore, the energy estimate
	$$\Big\| \ind_{[0, \infty)}(t) \int_0^t \mc{V}_\infty(t-s) G(s) ds \Big\|_{X_\infty} \lesa \| G \|_{L^1_t \dot{H}^{\frac{d-2}{2}}_x}$$
together with the Strichartz estimate
        $$ \| |\nabla|^{\sigma_d} \psi \|_{L^3_t L^{\frac{6d}{3d-4}}_x} \lesa  \||\nabla|^{\sigma_d} \mc{V}_\infty(-t)  \psi\|_{\ell^2 U^3} \lesa  \| \psi \|_{X_\infty} $$
implies that \ref{item:N1} also holds for $c=\infty$. Moreover, as the above trilinear bound does not require any null structure when $c=\infty$, we in particular have for any $k\in \{-3, -1, 1, 3\}$
			\begin{equation}\label{eqn:N1:Fk bound}
            \begin{split}
				\Big\| \int_0^t \mc{V}_\infty(t-s) \gamma^0 &\big( F_k(\varphi + \psi) - F_k(\psi) \big) ds \Big\|_{L^1_t \dot{H}^{\sigma_d}_x}\\
				&\lesa \| |\nabla|^{\sigma_d} \mc{V}_\infty(-t) \varphi \|_{\ell^2 U^3} \Big( \| |\nabla|^{\sigma_d} \mc{V}_\infty(-t)\psi\|_{\ell^2 U^3} + \||\nabla|^{\sigma_d} \mc{V}_\infty(-t)\varphi\|_{\ell^2 U^3}\Big)^2.
            \end{split}
			\end{equation}
			
\subsection{Proof of \ref{item:N2}}\label{subsec:proofn2} The first step is to observe that, unlike in the case of the massless limit, the operator $\tilde{\mc{R}}_c$ does not converge to 1 as $c\to \infty$. This is not surprising as $\tilde{\mc{R}}_c$ considers the difference in the free solution operators to the Dirac equation \eqref{eq:dirac c} and the Schr\"odinger equation \eqref{eqn:NLS}, and we only expect solutions to the Dirac equation to converge after multiplying by $e^{itc^2 \gamma^0}$. In particular, if we want convergence as $c\to \infty$, we instead need to consider the modified operator
$$ \tilde{\mc{R}}_{c}^{mod}(t) := e^{itc^2 \gamma^0} \tilde{\mc{R}}_c(t) = \lr{c^{-1} \nabla}^{-\frac{1}{2}} e^{itc^2 \gamma^0} \mc{V}_c(t) \mc{V}_\infty(-t). $$
The corrected operator $\tilde{\mc{R}}_c^{mod}$ now converges to 1, as least on compact time intervals and  bounded frequencies. More precisely, we have the following counterpart to Lemma \ref{lem:comp-norms II}.

\begin{lemma}\label{lem:comp-norms c}
Let $1\les b <\infty$. For all $T>0$, $I=[0, T)$, $R>1$, and $c\g 1$,
we have
        $$\big\| \ind_I(t)  \mc{V}_\infty(-t) \big( \tilde{\mc{R}}_c^{mod} -1 \big) P_{\les R}\psi \big\|_{\ell^2 U^b}  \lesa{}   \lr{R^2T}R^2 c^{\frac{2}{b}-1}  \|  \mc{V}_\infty(-t) \psi \|_{\ell^2 U^b},$$
 and
 		$$ \big\| \ind_I(t)   \big( \lr{c^{-1} \nabla} \tilde{\mc{R}}_c^{mod} -1 \big) P_{\les R} \varphi \big\|_{L^\infty_t L^2_x} \lesa{}    \lr{R^2T}R^3 c^{-1}  \| \varphi \|_{L^\infty_t L^2_x}.
 $$
\end{lemma}
\begin{proof}
  We define the adjusted projectors via their symbols
  \[
\tilde{\Pi}_{c,\pm}(\xi)=\frac12 \Big(I\mp \frac{1}{\lr{\xi}_c}\gamma^0(\gamma^j \xi_j+c)\Big).
  \]
 Note that $E_{\pm} = \lim_{c \to \infty}\tilde{\Pi}_{c,\mp}(\xi)$ (with opposite signs due to our notational convention).
  Then,
  \[
\mc{V}_c(t)=e^{i c t\lr{\nabla}_c}\tilde{\Pi}_{c,+}+e^{-i c t\lr{\nabla}_c}\tilde{\Pi}_{c,-}.
\]
Setting $\phi(t)= \mc{V}_\infty(-t)\psi(t)$, the first claimed estimate boils down to proving
\[
\|\ind_I(t)\big( \lr{c^{-1}\nabla}^{-\frac12}e^{-it\Delta/2\gamma^0} e^{itc^2\gamma^0} \mc{V}_c(t)-1\big)\phi \|_{\ell^2 U^{b}}\lesa\| \phi \|_{\ell^2 U^{b}}
\]
As in the proof of Lemma \ref{lem:comp-norms II} we want to prove bounds for the $\ell^2 U^{b}$ operator norms of $\ind_I(t)A_\pm$, where
\[
A_\pm:=\big( \lr{c^{-1}\nabla}^{-\frac12}e^{-i\frac{t}{2} \Delta \gamma^0} e^{itc^2\gamma^0}e^{\pm it c\lr{\nabla}_c } -1\big)\tilde{\Pi}_{c,\pm}P_{\les R}
\]
by invoking Lemma \ref{lem:comp-norms-gen}.
By symmetry it is enough to consider the case $\pm=+$ and we decompose $A_+=A_1+A_2$ with symbols
\begin{align*}
  a_1(t,\xi):={}&\big( \lr{c^{-1}\xi}^{-\frac12}e^{-i\frac{t}{2}|\xi|^2} e^{-itc^2}e^{ it c\lr{\xi}_c }-1\big)E_-\tilde{\Pi}_{c,+}(\xi)p_{\les R}(\xi),\\
  a_2(t,\xi):={}&\big( \lr{c^{-1}\xi}^{-\frac12}e^{i\frac{t}{2}|\xi|^2} e^{itc^2}e^{ it c\lr{\xi}_c } -1\big)E_+\tilde{\Pi}_{c,+}(\xi)p_{\les R}(\xi).
\end{align*}
Lemma \ref{lem:comp-norms-gen} and Remark \ref{rem:Up operator bounds} imply
\[
  \|A_j\|_{U^{b}(\R,\mc{L}(L^2))}\lesa \sup_{t \in I}\|a_j(t,\xi)\|_{L^\infty_\xi} +\sup_{t \in I}\|a_j(t,\xi)\|_{L^\infty_\xi}^{1-\frac1b}\Big(\int_I \big\|\partial_t a_j(t,\xi)\big\|_{L^\infty_\xi} dt\Big)^{\frac1b}, \; j=1,2.\]
We start with the bound for $A_1$:
By Taylor expansion we have
$
|\frac{1}{\sqrt{1+x^2}}-1|\les x^2\text{ and } |\sqrt{1+x^2}-1-x^2/2|\les x^4/2
$,
therefore
\[
|\lr{c^{-1}\xi}^{-\frac12}-1|\les \frac{|\xi|^2}{c^2} \text{ and } |c\lr{\xi}_c-c^2-|\xi|^2/2|=c^2\Big|\sqrt{1+\frac{|\xi|^2}{c^2}}-1-\frac{|\xi|^2}{2c^2}\Big|\les \frac{|\xi|^4}{2c^2}.
\]
Further, this implies
\[|\lr{c^{-1}\xi}^{-\frac12}e^{it(c \lr{\xi}_c -c^2-|\xi|^2/2)}-1|\les  \frac{|\xi|^4}{2c^2} |t| +\frac{|\xi|^2}{2c^2}\]
We obtain
\[
\sup_{t \in I}\|a_1(t,\xi)\|_{L^\infty_\xi} \lesa (R^2T+1)\frac{R^2}{c^2}\; \text{ and }\;
\int_I \big\|\partial_t a_1(t,\xi)\big\|_{L^\infty_\xi} dt\lesa \frac{R^4T}{c^2},
\]
therefore
\[
  \|A_1\|_{U^{b}(\R,\mc{L}(L^2))}
  \lesa{} (R^2T+1)\frac{R^2}{c^2}.
\]

Concerning $A_2$ we proceed similarly, but now we have to exploit the asymptotic orthogonality of $E_+$ and $\tilde{\Pi}_{c,+}(\xi)$ since there is no cancellation in the phase. More precisely,
\[
|E_+\tilde{\Pi}_{c,+}(\xi)|=\big|\frac12 E_+ (1-\frac{c}{\lr{\xi}_c})-\frac12 E_+\gamma^j \xi_j \frac{1}{\lr{\xi}_c}\big|\les \frac{|\xi|^2}{2c^2}+\frac{1}{2c},
\]
which implies
\[
  \sup_{t \in I}\|a_2(t,\xi)\|_{L^\infty_\xi} \lesa \frac{R^2}{c} \; \text{ and }\;
  \int_I \big\|\partial_t a_2(t,\xi)\big\|_{L^\infty_\xi} dt\lesa T R^2c,\]
and
\[
 \|A_2\|_{U^{b}(\R,\mc{L}(L^2))}
  \lesa{} R^2(T^{\frac1b}+1)c^{\frac2b-1},
\]
which completes the proof of the first claim.

The second claimed bound in $L^\infty_t L^2_x$ is easier to prove (and corresponds in some sense to the case $b=\infty$ above, up to the operator $\lr{c^{-1} \nabla}$). From the $L^\infty(I,L^\infty_{\xi})$-bounds for $a_1$ and $a_2$ above we conclude that
\begin{align*}
  \|(\lr{c^{-1} \nabla} \tilde{\mc{R}}_c^{mod} -1 )P_{\les R}\varphi\|_{L^\infty(I,L^2_x)}\les{}& \| \lr{c^{-1}\nabla} ( \tilde{\mc{R}}^{mod}_c - 1)P_{\les R}\varphi\|_{L^\infty(I,L^2_x)} +\| ( \lr{c^{-1}\nabla} - 1)P_{\les R}\varphi\|_{L^\infty(I,L^2_x)}\\
  \lesa{}& R^3\frac{R^2T+1}{c^2}+\frac{R^3}{c} +\frac{R^2}{2c^2},
\end{align*}
where we also used the fact that $|\lr{c^{-1}\xi} - 1|\les |\xi |^2/c^2$.
\end{proof}

We now turn to the proof of \ref{item:N2}. Fix $T>0$. As in the proof of \ref{item:M2}, in view of the uniform bound \ref{item:N1}, it suffices to prove \ref{item:N2} under the assumption $\supp \widehat{\psi} \subset \{ R^{-1} \les |\xi| \les R\}$ for some $R> 0$. Moreover, as we now have
           \begin{equation}\label{eqn:N2:X vs Ua}
                (R^{d-2} \log R)^{-\frac{1}{2}} \|  \psi \|_{X_\infty} \lesa  \| \mc{V}_\infty(t) \psi \|_{U^a} \lesa (R^{d-2} \log R)^\frac{1}{2} \|  \psi \|_{X_\infty}
           \end{equation}
(compare to \eqref{eqn:M3:Ua bdd by S} and \eqref{eqn:M3:S bdd by Ua}) and step functions are dense in $U^a$, again using the uniform bound \ref{item:N1} we see that it is enough to prove that for any $R>0$ and any $\psi \in \mc{V}_\infty(t) \mathfrak{G}$ with $\supp \widehat{\psi} \subset \{R^{-1}\les  |\xi| \les R\}$ we have
        $$
         \lim_{c \to \infty} \Big\| \int_0^T \ind_{[0, T)}(s) \mc{V}_c(t-s) P_{\les R} \Big[ \gamma^0 F\big( \tilde{\mc{R}}_c(s) \psi(s) \big) - \tilde{\mc{R}}_c(s) \gamma^0 F_1\big( \psi(s) \big) \Big] ds \Big\|_{X_c} = 0
        $$
(recall that the collection of step functions $\mathfrak{G}$ is defined in the beginning of Subsection \ref{subsec:fs}). The next step is to replace the operators $\tilde{\mc{R}}_c$ with the corrected versions $\tilde{\mc{R}}_c^{mod}$. To this end, we note that
		\begin{align*}
		\mc{V}_c(t-s)	\Big[ \gamma^0 F\big( &\tilde{\mc{R}}_c(s) \psi(s) \big) - \tilde{\mc{R}}_c(s) \gamma^0 F\big( \psi(s) \big)\Big]		\\	
		&= \mc{R}_c(t) \mc{V}_\infty(t-s) \Big[ \big( \tilde{\mc{R}}_c^{mod}(s)\big)^{-1} \gamma^0 e^{-isc^2 \gamma^0}
		F\big( e^{isc^2\gamma^0} \tilde{\mc{R}}_c^{mod}(s) \psi \big) - \gamma^0 F(\psi)\Big]
		\end{align*}
and hence in view of the decomposition \eqref{eqn:F res decomp} it suffices to prove for $k\in \{-3, -1, 1, 3\}$ we have
		\begin{equation}\label{eqn:N2:goal I}
		\lim_{c\to \infty} \Big\| \int_0^t \ind_{[0, T)}(s) \mc{V}_\infty(t-s) \gamma^0 P_{\les R} \Big[ \big( \tilde{\mc{R}}_c^{mod}(s)\big)^{-1} F_{k,c}( \tilde{\mc{R}}_c^{mod}(s) \psi ) - F_{k,c}(\psi) \Big] ds \Big\|_{X_\infty} = 0
		\end{equation}
and that for the non-resonant contributions $k\in \{-3, -1, 3\}$ we have
		\begin{equation}\label{eqn:N2:goal II}
		\lim_{c\to \infty} \Big\| \int_0^t \ind_{[0, T)}(s) \mc{V}_\infty(t-s) \gamma^0 P_{\les R} F_{k,c}(\psi)  ds \Big\|_{X_\infty} = 0
		\end{equation}
where for ease of notation we take $F_{k, c} = e^{i(1-k)sc^2 \gamma^0} F_k $. Since $|e^{i(1-k) sc^2 \gamma^0}| \les 1$, An application of \eqref{eqn:N1:Fk bound} implies that for any $\varphi \in Y_\infty$ we have
		\begin{align*}
			&\Big| \int_0^T \int_{\RR^d} \overline{\varphi} P_{\les R} \Big[ \big( \tilde{\mc{R}}_c^{mod}\big)^{-1} F_{k,c}( \tilde{\mc{R}}_c^{mod}\psi ) - F_{k,c}(\psi) \Big] dt dx \Big| \\
				={}& 	\Big| \int_0^T \int_{\RR^d} \overline{ \lr{c^{-1} \nabla} \tilde{\mc{R}}_c^{mod} P_{\les R}\varphi } F_{k,c}( \tilde{\mc{R}}_c^{mod}(t) \psi ) - \overline{P_{\les R}\varphi} F_{k,c}(\psi)  dt dx \Big| \\
				\lesa{}& \big\| \ind_{[0, T)} ( \lr{c^{-1} \nabla} \tilde{\mc{R}}_c^{mod} -1) P_{\les R}\varphi \big\|_{L^\infty_t \dot{H}^{-\sigma_d}_x} \| \ind_{[0, T)}   F_{k}( \tilde{\mc{R}}_c^{mod} \psi)\|_{L^1_t \dot{H}^{\sigma_d}_x} \\\
				&\qquad  + \| \varphi \|_{L^\infty_t \dot{H}^{-\sigma_d}_x} \big\| \ind_{[0, T)} \big(F_k(\tilde{\mc{R}}_c^{mod} \psi) - F_k(\psi) \big) \big\|_{L^1_t \dot{H}^{\sigma}_x}\\
				\lesa{}& R^{3\sigma_d} \big\| \ind_{[0, T)} ( \lr{c^{-1} \nabla} \tilde{\mc{R}}_c^{mod} -1) P_{\les R} \varphi \big\|_{L^\infty_t \dot{H}^{-\sigma_d}_x} \Big( \| \mc{V}_\infty(-t) \psi \|_{\ell^2 U^3} + \| \ind_{[0, T)}  \mc{V}_\infty(-t) ( \tilde{\mc{R}}_c^{mod}-1) \psi)\|_{\ell^2 U^3}\Big)^3\\
				\qquad{}&  +  R^{3\sigma_d} \| \varphi \|_{L^\infty_t \dot{H}^{-\sigma_d}_x} \big\| \ind_{[0, T)} \mc{V}_\infty(-t) (\tilde{\mc{R}}_c^{mod} -1) \psi\big\|_{\ell^2 U^3}\Big( \| \mc{V}_\infty(-t) \psi \|_{\ell^2 U^3} + \| \ind_{[0, T)}\mc{V}_\infty(-t)(\tilde{\mc{R}}_c^{mod} -1 )\psi \|_{\ell^2 U^3}\Big)^2.
		\end{align*}
As we clearly have $\mc{V}_\infty(t) \ell^2 U^3 \subset L^\infty L^2$, \eqref{eqn:N2:goal I} follows from Lemma \ref{lem:comp-norms c} (with $b=3$). It only remains to deal with the non-resonant contribution \eqref{eqn:N2:goal II}. We begin by observing that as $\mc{V}_{\infty}(-t) \psi \in \mathfrak{G}$, it suffices to prove that for any $0\les t_1 < t_2 \les T$ we have
         \begin{equation}\label{eqn:N2:goal II v2}
            \Big\| \int_0^t \ind_{[t_1, t_2)}(s) \mc{V}_\infty(-s) e^{ i(1-k)sc^2 \gamma^0} \gamma^0  P_{\les R} F_{k}(\psi)  ds \Big\|_{\ell^2 U^a} = 0
         \end{equation}
for any $\psi = \mc{V}_\infty(t) f$ where $\supp \widehat{f} \subset \{ R^{-1} \les |\xi| \les R\}$. To this end, we note that provided $c \gg R$, the temporal frequency support of $\mc{V}_\infty(-s) e^{ i(1-k)sc^2 \gamma^0} \gamma^0  P_{\les R} F_{k}(\psi)$ is contained in the set $\{|\tau| \gtrsim c^2\}$. Consequently, an application of Lemma \ref{lem:high freq energy ineq} implies that
         \begin{align*}
            \Big\| \int_0^t \ind_{[t_1, t_2)}(s) &\mc{V}_\infty(-s) e^{ i(1-k)sc^2 \gamma^0} \gamma^0  P_{\les R} F_{k}(\psi)  ds \Big\|_{\ell^2 U^a} \\
                &\lesa c^{\frac{2}{a} -2} \|  \mc{V}_\infty(-s) e^{ i(1-k)sc^2 \gamma^0} \gamma^0  P_{\les R} F_{k}(\psi) \|_{L^a_t L^2_x} \lesa  c^{\frac{2}{a} -2} \| \psi \|_{L^{3a}_t L^6_x}^3.
         \end{align*}
Therefore \eqref{eqn:N2:goal II v2} follows by Sobolev embedding and the Strichartz estimate for free Schr\"odinger waves
            $$ \| \psi \|_{L^{3a}_t L^6_x} \lesa R^{\frac{d}{3} - \frac{2}{3a}} \| \psi \|_{L^{3a}_t L^{\frac{6ad}{3ad-4}}_x} \lesa R^{\frac{d}{3} - \frac{2}{3a}} \| f\|_{L^2_x}. $$
This completes the proof of \ref{item:N2}.

\subsection{Proof of \ref{item:N3}}\label{subsec:proofn3} In view of the uniform bound \ref{item:N1}, and arguing as in the proof of \ref{item:N2}, it suffices to proof \ref{item:N3} under the assumption that $\mc{V}_\infty(-t) \psi \in \mathfrak{G}$ is a step function with $\supp \widehat{\psi} \subset \{R^{-1} \les |\xi| \les R\}$. Moreover, for $T>0$ sufficiently large, $\ind_{[T,\infty)} \psi = \ind_{[T, \infty)} \mc{V}_\infty(t) f$ for some $f \in L^2$. Hence, again applying the uniform bound \ref{item:N1} together with \eqref{eqn:N2:X vs Ua}, we have reduced matters to proving that \ref{item:N3} holds in the special case $\psi = \mc{V}_\infty(t) f$ with $f\in L^1 \cap L^2$ satisfying the Fourier support condition $\supp \widehat{f} \subset \{R^{-1} \les |\xi| \les R\}$. Rescaling the Klein-Gordon dispersive estimate \eqref{eqn:KG disp} immediately gives for any $1\les c \les \infty$ and $\mu \in 2^\ZZ$
            $$ \| \mc{V}_c(t) g_\mu \|_{L^\infty_x} \lesa |t|^{-\frac{d}{2}} \lr{c^{-1} \mu }^{\frac{d+2}{2}} \| g_\mu \|_{L^1_x}.$$
Consequently, we see that for any $1 \les c \les \infty$
    \begin{align*}
        \Big\| \int_0^t \ind_{[T, \infty)}(s) \mc{V}_c(-s) \gamma^0 F\big( \tilde{R}_c(s) \psi(s) \big) ds \Big\|_{\ell^2 U^a}
                &\lesa R^{s_d} \big\| \ind_{[T, \infty)}(t) F\big( \mc{V}_c(t) f \big) \big\|_{L^1_t L^2_x} \\
                &\lesa R^{s_d} \| \ind_{[T, \infty)}(t) \mc{V}_c(t) f \|_{L^2_t L^\infty_x}^2 \| f \|_{L^2}\\
                &\lesa R^{s_d + \frac{d}{3} + 2} T^{1-d} \| f \|_{L^1}^3
    \end{align*}
where the implied constant is independent of $1\les c \les \infty$. Therefore \ref{item:N3} follows.

\subsection{Proof of Corollary \ref{cor:non-rel conv}} In view of the bound
        \begin{align*}
            \| \lr{c^{-1}\nabla}^{\frac{1}{2}} e^{itc^2 \gamma^0} \psi^{(c)} - \psi^{(\infty)}\|_{L^\infty_t \dot{H}^{\sigma_d}_x} &\lesa \Big( \sum_{\lambda \in 2^{\ZZ}} \lambda^{2\sigma_d} \| \lr{c^{-1}\nabla}^{\frac{1}{2}} \psi^{(c)}_\lambda \|_{L^\infty_t L^2_x}^2 \Big)^{\frac{1}{2}}  +\Big( \sum_{\lambda \in 2^{\ZZ}} \lambda^{2\sigma_d} \| \psi^{(c)}_\lambda \|_{L^\infty_t L^2_x}^2 \Big)^{\frac{1}{2}} \\
            &\lesa \| \psi^{(c)} \|_{X_c} + \| \psi^{(\infty)} \|_{X_\infty},
        \end{align*}
and the fact that the norm $\| \psi^{(c)} \|_{X_c}$ is bounded uniformly in $c\g1$, as $\mc{V}_\infty$ is unitary on $\dot{H}^{\sigma_d}$ it suffices to show that for any fixed $\lambda \in 2^\ZZ$ we have
            $$ \lim_{c\to \infty} \big\| \ind_I(t) \mc{V}_\infty(-t) \big( \lr{c^{-1}\nabla}^{\frac{1}{2}} e^{itc^2 \gamma^0} \psi^{(c)}_\lambda - \psi^{(\infty)}_\lambda \big) \big\|_{L^\infty_t L^2_x} = 0. $$
To this end, we decompose
        \begin{align*}
            \mc{V}_\infty(-t) \big( &\lr{c^{-1}\nabla}^{\frac{1}{2}} e^{itc^2 \gamma^0} \psi^{(c)}_\lambda - \psi^{(\infty)}_\lambda\big) \\
            &= \mc{V}_\infty(-t) \big( \tilde{\mc{R}}_c^{mod}(t) -1 \big) \mc{V}_\infty(t) \mc{V}_c(-t) \psi_{\lambda}^{(c)} + \lr{c^{-1}\nabla}^{\frac{1}{2}} \mc{V}_c(-t) \psi_{\lambda}^{(c)} - \mc{V}_\infty(-t) \psi_{\lambda}^{(\infty)}.
        \end{align*}
For the first difference we apply Lemma \ref{lem:comp-norms c} and observe that for any $1\les b<\infty$
        \begin{align*}
          \big\|  \ind_I(t) \mc{V}_\infty(-t) &\lr{c^{-1}\nabla}^{\frac{1}{2}} \big( \tilde{\mc{R}}_c^{mod}(t) -1 \big) \mc{V}_\infty(t) \mc{V}_c(-t) \psi^{(c)}_{\lambda} \big\|_{L^\infty_t L^2_x}\\
                &\lesa_\lambda \big\| \mc{V}_\infty(-t) \big( \tilde{\mc{R}}_c^{mod}(t) -1 \big) \mc{V}_\infty(t) \mc{V}_c(-t) \psi^{(c)}_{\lambda} \big\|_{\ell^2 U^b} \\
                &\lesa_\lambda c^{\frac{2}{b} -1} \| \mc{V}_c(-t) \psi^{(c)}_\lambda \|_{\ell^2 U^b} \lesa_\lambda c^{\frac{2}{b} -1} \| \psi^{(c)}_\lambda \|_{X_\infty}.
        \end{align*}
Letting $c\to \infty$ and again using the fact that $\| \psi^{(c)} \|_{X_c}$ is bounded uniformly in $c\g 1$, we see that (choosing $b>2$)
        $$ \lim_{c\to \infty} \big\|  \ind_I(t) \mc{V}_\infty(-t) \lr{c^{-1}\nabla}^{\frac{1}{2}} \big( \tilde{\mc{R}}_c^{mod}(t) -1 \big) \mc{V}_\infty(t) \mc{V}_c(-t) \psi^{(c)}_{\lambda} \big\|_{L^\infty_t L^2_x} = 0. $$
The remaining term is a direct consequence of Theorem \ref{thm:limit c}.

\section*{Acknowledgements}
\begin{enumerate}
  \item
Funded by the Deutsche Forschungsgemeinschaft (DFG, German Research Foundation) -- Project-ID 317210226 -- SFB 1283.
\item
T.C.\ is supported by the Marsden Fund Council grant 19-UOO-142, managed by Royal Society Te Ap\={a}rangi.
\item
S.H.\ thanks the \emph{Laboratoire de Math\'ematique d'Orsay -- Universit\'e Paris-Saclay} for its hospitality, where part of this paper has been written.
\end{enumerate}

\bibliographystyle{amsplain}
\bibliography{unified_Dirac}
\end{document}